\documentclass[a4paper,10pt]{amsart}
\usepackage{amssymb, amsmath,stmaryrd,svninfo,enumerate,enumitem,mathtools}
\usepackage[headings]{fullpage}
\usepackage[usenames]{color}
\usepackage[notref, notcite, color]{showkeys}
\usepackage[all]{xy}
\usepackage[pagebackref, colorlinks]{hyperref}
\usepackage{diagrams}
\usepackage[french,english]{babel}

\makeatletter
\newenvironment{abstracts}{%
  \ifx\maketitle\relax
    \ClassWarning{\@classname}{Abstract should precede
      \protect\maketitle\space in AMS document classes; reported}%
  \fi
  \global\setbox\abstractbox=\vtop \bgroup
    \normalfont\Small
    \list{}{\labelwidth\z@
      \leftmargin3pc \rightmargin\leftmargin
      \listparindent\normalparindent \itemindent\z@
      \parsep\z@ \@plus\p@
      
      \itemsep\medskipamount
    }%
}{%
  \endlist\egroup
  \ifx\@setabstract\relax \@setabstracta \fi
}

\newcommand{\abstractin}[1]{%
  \otherlanguage{#1}%
  \item[\hskip\labelsep\scshape\abstractname.]%
}
\makeatother

 \let\mathbb\mathbf

\DeclareFontFamily{OT1}{rsfs}{}
\DeclareFontShape{OT1}{rsfs}{n}{it}{<-> rsfs10}{}
\DeclareMathAlphabet{\mathscr}{OT1}{rsfs}{n}{it}

 \newcommand{\mybf}{\mathbf}
\newcommand{\A}{\mybf{A}}
\newcommand{\B}{\mybf{B}}
 \newcommand{\BB}{\mybf{B}}
 \newcommand{\CC}{\mybf{C}}
 \newcommand{\DD}{\mybf{D}}
\newcommand{\D}{\mybf{D}}
 \newcommand{\EE}{\mybf{E}}
\newcommand{\F}{\mybf{F}}
 \newcommand{\FF}{\mybf{F}}
\newcommand{\Q}{\mybf{Q}}
 \newcommand{\QQ}{\mybf{Q}}
\newcommand{\R}{\mybf{R}}
\newcommand{\TT}{\mathbf{T}}
 \newcommand{\ZZ}{\mybf{Z}}
\newcommand{\Z}{\mybf{Z}}
 \newcommand{\NN}{\mybf{N}}
 \newcommand{\cH}{\mathcal{H}}
 \newcommand{\cL}{\mathcal{L}}
 \newcommand{\cK}{\mathcal{K}}
 \newcommand{\cO}{\mathcal{O}}

\newcommand{\C}{\mathcal{C}}

 \newcommand{\vp}{\varphi}

 \newcommand{\Qp}{\QQ_p}
 \newcommand{\Zp}{\ZZ_p}
 \newcommand{\Qpi}{\QQ_{p, \infty}}
 
 \newcommand{\Qpnr}{\Qp^{\mathrm{nr}}}
 \newcommand{\Qpb}{\overline{\QQ}_p}

 \newcommand{\Dcris}{\DD_{\cris}}

 \newcommand{\Drig}{\DD_{\mathrm{rig}}}
 
 \newcommand{\Bmax}{\BB_{\max}}

 \renewcommand{\AA}{\mybf{A}}
 
 \newcommand{\Lt}{\widetilde{L}}
 \newcommand{\Rt}{\widetilde{R}}
 \newcommand{\At}{\widetilde{A}}
 \newcommand{\cOt}{\widetilde{\cO}}
\newcommand{\La}{\ifmmode\Lambda\else$\Lambda$\fi}
\newcommand{\Tun}{\ifmmode \mathbb{T}_{un}\else$\mathbb{T}_{un}$\ \fi}

 \newcommand{\cris}{\mathrm{cris}}
 \newcommand{\st}{\mathrm{st}}
 \renewcommand{\max}{\mathrm{max}}
 \newcommand{\dR}{\mathrm{dR}}
 \newcommand{\Iw}{\mathrm{Iw}}
 \newcommand{\rig}{\mathrm{rig}}
 
 \newcommand{\pst}{\mathrm{pst}}
 
 \newcommand{\nr}{\mathrm{nr}}

 \DeclareMathOperator{\Fil}{Fil}
 \DeclareMathOperator{\Tor}{Tor}
 \DeclareMathOperator{\Det}{Det}
 \DeclareMathOperator{\Gal}{Gal}
 \DeclareMathOperator{\id}{id}

\DeclareMathOperator{\Mat}{Mat}
 \DeclareMathOperator{\Hom}{Hom}
 \DeclareMathOperator{\Isom}{Isom}

 \DeclareMathOperator{\tw}{Tw}
 
 \DeclareMathOperator{\ord}{ord}
 \DeclareMathOperator{\GL}{GL}

\DeclareMathOperator{\Frac}{Frac}

\DeclareMathOperator{\Spf}{Spf}
\DeclareMathOperator{\Sp}{Sp}
\DeclareMathOperator{\Spec}{Spec}
\DeclareMathOperator{\MaxSpec}{MaxSpec}

\renewcommand{\d}{\Det}
\renewcommand{\u}{\mathbf{1}}

\newcommand{\htimes}{\mathop{\widehat\otimes}}

\renewcommand{\det}{\operatorname{det}}

\theoremstyle{plain}
  \newtheorem{theorem}{Theorem}[section]
  \newtheorem{proposition}[theorem]{Proposition}
  \newtheorem{lemma}[theorem]{Lemma}
  \newtheorem{corollary}[theorem]{Corollary}

\theoremstyle{definition}
  \newtheorem{definition}[theorem]{Definition}
  \newtheorem{example}[theorem]{Example}
  \newtheorem{remark}[theorem]{Remark}

 \renewcommand{\u}{\mathbf{1}}
 \renewcommand{\H}{\mathrm{H}}

\title{Wach modules, Regulator maps and $\varepsilon$-isomorphisms in families}

\begin{document}

\author{Rebecca Bellovin}
\address[Bellovin]{180 Queen's Gate \\ Imperial College London \\ SW7 2AZ}
\email{r.bellovin@imperial.ac.uk}
\urladdr{\url{http://wwwf.imperial.ac.uk/~rbellovi/}}

\author{Otmar Venjakob}
\address[Venjakob]{Mathematisches Institut \\ Universit\"at Heidelberg \\ Im Neuenheimer Feld 288 \\69120 Heidelberg, Germany}
\email{otmar@mathi.uni-heidelberg.de}
\urladdr{\url{http://www.mathi.uni-heidelberg.de/~otmar/}}

\thanks{The first author was partially supported by an NSF Mathematical Sciences Postdoctoral Research Fellowship, and the second author was supported by the DFG-Forschergruppe "Symmetrie, Geometrie und Arithmetik" .}

\subjclass[2010]{Primary: 11R23. 
Secondary: 11S40.
}

\begin{abstracts}
\abstractin{english}
We prove the ``local $\varepsilon$-isomorphism'' conjecture of Fukaya and Kato \cite{fukayakato06} for certain crystalline families of $G_{\Qp}$-representations. This conjecture can be regarded as a local analogue of the Iwasawa main conjecture for families.  Our work extends earlier work of Kato for rank-$1$ modules (cf. \cite{venjakob11}), of Benois and Berger for crystalline $G_{\Qp}$-representations with respect to the cyclotomic extension (cf. \cite{benoisberger08}), as well as of Loeffler, Venjakob, and Zerbes (cf. \cite{loeffler-venjakob-zerbes}) for crystalline $G_{\Qp}$-representations with respect to abelian $p$-adic Lie extensions of $\Qp$. Nakamura \cite{nak13, nak15} has also formulated a version of Kato's $\varepsilon$-conjecture for affinoid families of $(\varphi,\Gamma)$-modules over the Robba ring, and proved his conjecture in the rank-$1$ case.  He used this case to construct an $\varepsilon$-isomorphism for families of trianguline $(\varphi,\Gamma)$-modules, depending on a fixed triangulation. Our results imply that this $\varepsilon$-isomorphism is independent of the chosen triangulation for certain crystalline families.  The main ingredient of our proof consists of the construction of families of Wach modules generalizing work of Wach and Berger \cite{berger04} and following Kisin's approach to the construction of potentially semi-stable deformation rings \cite{kisin-bt}.
\end{abstracts}

\selectlanguage{english}
\maketitle
\tableofcontents

\section{Introduction}

The significance of (local) $\varepsilon$-factors \`{a} la Deligne and Tate, or more generally of the
(conjectural) $\varepsilon$-isomorphism suggested by Fukaya and Kato in \cite[\S 3]{fukayakato06}, is at
least twofold: First, they are important ingredients to obtain a precise functional equation
for $L$-functions or more generally for (conjectural) $\zeta$-isomorphism (loc.~ cit., \S 2) of
motives in the context of  Tamagawa number conjectures; second, they
are essential in interpolation formulae of (actual) $p$-adic $L$-functions and for the relation
between $\zeta$-isomorphisms and (conjectural) $p$-adic $L$-functions
as discussed in (loc.~cit., \S 4). Of course the two occurrences are closely related; for a survey
on these ideas see also \cite{ven-BSD}.

The $\varepsilon$-isomorphism conjecture asserts roughly speaking the existence of a canonical trivialization of the determinant of a cohomology complex associated to any  representation $M$ of the Galois group $G_{\Qp} = \Gal(\Qpb / \Qp)$ with coefficients in a $p$-adically complete local ring $R$. This trivialization is required to be compatible with a specific ``standard'' trivialization of the corresponding complex for  de Rham representations with coefficients in finite extensions of $\Qp$ obtained by specializing $M$ at ideals, or more generally at representations, of $R$. This standard trivialization contains information about the $\varepsilon$-factor of the corresponding Weil--Deligne representation, hence the terminology ``$\varepsilon$-isomorphism'', and the Bloch-Kato-exponential map.

The aim of this article is to prove, for any    crystalline family $M$ of  Galois representations of $G_{\Qp} = \Gal(\Qpb / \Qp)$ over a complete local noetherian $\Z_p$-algebra $R$ with finite residue field,  the existence of the $\varepsilon$-isomorphism  (see Theorem \ref{thm:conclusion})
 \[	\varepsilon_{\Lambda_{R}(G), \xi}(M):\Det_{\Lambda_{\Rt}(G)}(0) \xrightarrow{\sim}
    \left[ \Det_{\Lambda_{\Rt}(G)} \left(\Lambda_{\Rt}(G)\otimes_{\Lambda_R(G)} R\Gamma_{\Iw}(K_\infty, M)\right)\right]
    \left[ \Det_{\Lambda_{\Rt}(G)} \left(\Lambda_{\Rt}(G) \otimes_R M\right)\right]	\]
such that the specializations $\varepsilon_{\Lambda_R(G),\xi}(M)\otimes_R\cO_L$ with respect to any $\Zp$-algebra homomorphism $R\to\cO_L$ into the ring of integers $\cO_L$ of a finite extension $L$ of $\Qp$  agree with $\varepsilon_{\Lambda_{\cO_L}(G), \xi}(M{\otimes_R\cO_L })$, the $\varepsilon$-isomorphisms  established in \cite{loeffler-venjakob-zerbes}. In particular $ \varepsilon_{\Lambda_{R}(G), \xi}(M)$ satisfies the above mentioned compatibility with the standard trivializations for corresponding specializations (see Proposition \ref{prop:compat2}). For further properties we refer the reader to section $\ref{properties}.$

 Here $\xi = (\xi_n)_{n \ge 1}$ denotes a compatible system of $p$-power roots of unity generating $\Zp(1)$ as a $\Zp$-module,
$ \Lambda_{R}(G)$ is the Iwasawa algebra of $G=G(K_\infty/\Qp)$ for $K_\infty=K(\mu(p))$ with $K$ any (possibly infinite)
unramified extension $K$ of $\Qp$ and    $ R\Gamma_{\Iw}(K_\infty, M)$ denotes the complex calculating local Iwasawa cohomology of $M$ over $K_\infty$. Furthermore, for an associative ring $A$ with one, $\d_A$ denotes the determinant functor  \`{a} la Knudsen-Mumford (if $A$ is commutative) or \`{a} la Deligne/Fukaya-Kato (even for non-commutative $A$) (see Appendix \ref{determinants});  from a technical point of view we want to stress  that we use the realisation via line bundles whenever it is possible in order to take advantage of (stronger) arguments and higher flexibility from  commutative algebra.   The Iwasawa algebra  $\Lambda_{\Rt}(G)$ is defined in \eqref{f:tilde}. We are mainly interested in the case,
where $G\cong \mathbb{Z}_p^2\times\Delta$ for a finite group $\Delta$.  This specific instance of the local $\varepsilon$-isomorphism conjecture can be thought of as a ``local Iwasawa main conjecture'' for $M$ over the extension $\Qp^{\mathrm{ab}} / \Qp$; compare the discussion at the end of \S 2 in \cite{venjakob11}.

The key ingredient in the construction of the local $\varepsilon$-isomorphism and the technical heart of this paper is the construction of Wach modules $\NN(M)$ for formal families of Galois representations $M$.  The theory of $(\varphi,\Gamma)$-modules and several other aspects of $p$-adic Hodge theory had been generalized to formal and rigid-analytic families by work of Dee \cite{dee01}, Berger and Colmez \cite{berger-colmez}, Kedlaya and Liu~\cite{kedlaya-liu08}, and Bellovin \cite{bellovin13}.  In \textsection~\ref{families}, we give a theory of Wach modules for formal families of crystalline representations of $G_F$, following Kisin's strategy for the construction of potentially semi-stable deformation rings~\cite{kisin-pst-def-rings}.

We then compare Iwasawa cohomology to our Wach modules, extending the classical case.  With this in hand, we can define $\varepsilon$-isomorphisms for families of crystalline representations of $G_{\Qp}$ with coefficients in rings $R$ which are complete, local, noetherian $\Zp$-algebras which are Cohen--Macaulay, normal, and flat over $\Z_p$, with reduced generic fiber, and we show that the $\varepsilon$-isomorphisms we construct are compatible with base change and with the $\varepsilon$-isomorphisms constructed in~\cite{loeffler-venjakob-zerbes}.

With  this machinery at hand we can sharpen some work of Nakamura \cite{nak13}, who had given a different construction of $\varepsilon$-isomorphism even for not necessarily \'{e}tale, but trianguline   $(\varphi,\Gamma)$-modules (in families) over the Robba ring, but in which the final integrality statement is missing and which a priori depends on the choice of a triangulation. For crystalline families our results implies the independence of such a choice (Corollary \ref{cor:nakamura}). See section \ref{sec:nakamura} for a comparison with his result.

The advantage of providing $\varepsilon$-isomorphisms for families is as follows: First of all, Iwasawa main conjectures \cite{skinner-urban} or (Equivariant) Tamagawa Number Conjectures are nowadays often considered over (e.g.~Hida-) families.
Second, we hope that our work can be used to obtain results towards the construction of $\varepsilon$-isomorphisms attached to torsion representations: In   previous work~\cite{loeffler-venjakob-zerbes} by the second author together with David Loeffler and Sarah Zerbes  we constructed an ``$\varepsilon$-isomorphism'' for crystalline $p$-adic Galois representations.  Unfortunately, we were not able to address the following general question: let $T, T'$ be two torsion-free $\Zp$-representations of $G_{\Qp}$ such that $T \cong T'$ modulo $p^k$ for some $k$. Does it follow that the $\varepsilon$-isomorphisms for $T$ and $T'$ agree modulo $p^k$?  In this paper, we give a partial answer and  show that if $T$ and $T'$ both are crystalline with Hodge-Tate weights in $[a,b]$, then the desired congruence holds, if the framed deformation ring $R_{\cris}^{\square,[a,b]}$ is normal and Cohen-Macaulay. We do this by showing that the $\varepsilon$-isomorphisms for $T$ and $T'$ are both specializations of a ``universal'' isomorphism over $R_{\cris}^{\square,[a,b]}$; see section \ref{sec:deformation}.

We expect that it is possible to extend our constructions to the setting of affinoid families of Galois representations, using techniques similar to those of~\cite{wang-erickson15}.  However, we have not done so here.

We furthermore expect that it is possible to extend our constructions to formal families of crystabelline representations, that is, representations which become crystalline over a finite abelian extension of $\Qp$.  This would require one to check that the Wach modules of~\cite{berger-breuil} behave well integrally, and that one can build formal families of them.  One would also have to verify that the results of~\cite{loeffler-venjakob-zerbes} hold in this context.  With that in place, Hu and Paskunas have shown that certain crystabelline deformation rings are Cohen--Macaulay~\cite{hu-paskunas}.  This would allow us to extend our congruence results to representations which are not necessarily crystalline with Hodge--Tate weights in the Fontaine--Laffaille range.

We quickly outline the content of the various sections: After recalling the basic rings of $p$-adic Hodge theory in section \ref{section:rings}, we construct Wach modules $\NN(M)$ for formal families of Galois representations $M$ in section \ref{families}; the main result is summarized in Theorem \ref{thm:wach}. In section \ref{sec:fibre} we relate (the base change $\NN(M)^{\rig}$ of) the Wach module $\NN(M)$ to the $(\varphi,\Gamma)$-module $\Drig^{\dagger}(M^{\rig})$ attached to the associated family $M^\rig$ of Galois representations over the rigid analytic generic fiber $R^\rig$ of $R.$ In particular, we discuss the relation with $\D_{\cris}(M^{\rig})$, generalizing parts of \cite{berger04} to families.  This permits us to define $\Dcris(M)$ \emph{without} inverting $p$. In section \ref{sec:cohomology} we explain how to express Iwasawa cohomology of $M$ in terms of $\NN(M)$, generalizing to families a variant of Fontaine's isomorphism due to Loeffler and Zerbes in \cite{loefflerzerbes11}; see Propositions~\ref{prop:iwasawacohomology} and \ref{prop:iwasawacohomologyunramified}. Again following \cite{loefflerzerbes11}, we construct the regulator map attached to $M$ in section \ref{sec:regulator} and we study its specializations in Theorem \ref{thm:specialization}. Finally, in section \ref{sect:constructionG} we put everything together in order to construct an isomorphism $\Theta$ --- roughly speaking the determinant of the regulator map --- first over the total ring of fractions and then over the Iwasawa algebra with $p$ inverted.  The isomorphism $\Theta$ relates the Iwasawa cohomology of $M$ to $\D_{\cris}(M^{\rig}).$  Using Theorem~\ref{thm:Dcris-M} which relates the determinants of $\D_{\cris}(M)[1/p]$ and $M[1/p]$, we are able to define the $\varepsilon$-isomorphism (see Definition~\ref{def:epsilon}) over the Iwasawa algebra with $p$ inverted. By a specialization argument we show its integrality in Theorem~\ref{thm:conclusion}.  Finally, we briefly discuss the application to deformation rings in section \ref{sec:deformation} before relating our $\varepsilon$-isomorphism to that of Nakamura in section \ref{sec:nakamura}.

\section*{Acknowledgements}
This project grew out of natural questions arising  from the previous work \cite{loeffler-venjakob-zerbes}, and we are very grateful to David Loeffler and Sarah Zerbes for many suggestions.  We would also like to thank Toby Gee, Brandon Levin, James Newton, and Carl Wang Erickson for helpful discussions.

\section*{Notation and Conventions}
Let $p$ be an odd prime. For $H$ a $p$-adic analytic group, and $L$ a complete discretely valued extension of $\Qp$ with ring of integers $\cO$ and uniformiser $\varpi$, we write $\Lambda_{\cO}(H)$ and $\Lambda_L(H)=L\otimes_\cO \Lambda_\cO(H)$ for the Iwasawa algebras of $H$ with $\cO$ and $L$ coefficients, and $\cH_L(H)$ for the algebra of $L$-valued locally analytic distributions on $H$, which is the completion of $\Lambda_L(H)$ in a certain Fr\'echet topology. We shall only use these constructions in cases where $H$ is abelian and $p$-torsion-free, in which case all of these algebras are reduced commutative semi-local rings, and can be interpreted as algebras of functions on the $p$-adic analytic space parameterising characters of $H$.

  We shall also need the notation $\cK_{L}(H)$, signifying the total ring of quotients of $\cH_L(H)$, which is a finite direct product of fields.

  Let $\chi_{\mathrm{cyc}}:G_{K}\twoheadrightarrow\Gamma \to \Zp^\times$ be the $p$-cyclotomic character. If $K/\Qp$ is a finite extension, let $H_{K}\subset G_{K}$ be the kernel of $\chi_{\mathrm{cyc}}$ restricted to $G_K$.
If $K=\Qp$, we define $\Qpi :=(\overline\Qp)^{H_{\Qp}}= \Qp(\mu_{p^\infty})$ and $\Gamma = \Gal(\Qpi / \Qp)$.  We write $\gamma_{-1}$ for the unique element of $\Gamma$ such that $\chi(\gamma_{-1}) = -1$.

  If $M$ denotes any $G_{\Qp}$-module over $\Zp$ and $i$ is any integer, we write $M(i)$ for the $i$th Tate twist of $M.$

By   $R$ we shall denote a complete local noetherian $\Z_p$-algebra with maximal ideal $\mathfrak{m}_R.$ We write $ H_{\Qp}$ for the kernel of  $\chi_{cyc}.$ If $A$ is any associative ring, $M$ a (left) $A$-module and $Y$ a (right) $A$-module, the we denote the base change $Y\otimes_A M$ often just by $M_Y.$ If $\varphi$ denotes an injective ring homomorphism of $A$, which is now assumed to be commutative, we shall as usual set $\varphi^*M:=A\otimes_{A,\varphi} M$ where the left copy of $A$ is viewed as $A$-module via $\varphi.$ Assume, moreover, that $M$ itself bears a $\varphi$-linear action $\varphi_M.$ Then we obtain a canonical map $\varphi^*M\to M$ sending $a\otimes m$ to $a\varphi_M(m).$ In case the latter map is injective we shall sometimes identify $\varphi^*M$ with its image in $M$ without further indication.

We write $ \widehat{\Qp^\mathrm{nr}}$ for the completion of the maximal unramified extension $\Qp^\mathrm{nr}$ of $\Qp$ while the ring of integers are denoted by $ \widehat{\Zp^\mathrm{nr}}$ and $\Zp^\mathrm{nr} $, respectively. $F$ shall usually denote a finite {\it unramified} extension of $\Qp.$

\section{Rings of $p$-adic Hodge theory}\label{section:rings}

The rings (and their properties) we will use are primarily given in~\cite{berger02}, and we refer the reader there for a more detailed discussion.

Let $F$ be a finite unramified extension of $\Qp$ with ring of integers $\cO_F$ and residue field $k_F$.  Let $\widetilde\EE:=\varprojlim_{x\rightarrow x^p}\CC_p$, and let $\widetilde \EE^+$ be the subset of $x\in\widetilde\EE$ such that $x^{(0)}\in\mathcal{O}_{\CC_p}$.  If $x:=(x^{(i})$ and $y:=(y^{(i)})$ are two elements of $\widetilde\EE$, we define
\[	(x+y)^{(i)}:=\lim_{j\rightarrow\infty}(x^{(i+j)}+y^{(i+j)})^{p^j}\text{ and } (xy)^{(i)}:=x^{(i)}y^{(i)}	\]
Then $\widetilde\EE$ is an algebraically closed field of characteristic $p$.  There is a valuation $v_\EE$ defined by $v_\EE((x^{(i)}))=v_p(x^{(0)})$, and a Frobenius (given by raising to the $p$th power).  The valuation ring of $\widetilde\EE$ with respect to this valuation is $\widetilde\EE^+$, and $\widetilde\EE$ is complete with respect to $v_\EE$.

Let $\varepsilon:=(\varepsilon^{(0)},\varepsilon^{(1)},\varepsilon^{(2)}\ldots)\in \widetilde\EE^+$ be a choice of compatible $p$th power roots of unity with $\varepsilon^{(0)}=1$ and $\varepsilon^{(1)}\neq 1$.  There is a natural map $k_F(\!(\overline\pi)\!)\rightarrow \widetilde\EE$ given by sending $\overline\pi$ to $\varepsilon-1$; we denote its image by $\EE_F$, we denote by $\EE$ the separable closure of $\EE_F$ inside $\widetilde\EE$, and we denote by $\EE^+$ the valuation ring of $\EE$.

Let $\widetilde \A^+:=W(\widetilde\EE^+)$ and $\widetilde\A:=W(\widetilde\EE)$. There are two possible topologies on $\widetilde\A^+$ and $\widetilde\A$, the $p$-adic topology or the weak topology; they are complete for both.

The \emph{$p$-adic topology} is defined by putting the discrete topology on $W(\widetilde\EE)/p^nW(\widetilde\EE)$ for all $n$, and taking the projective limit topology on $\widetilde\A$; $\widetilde\A^+$ is given the subspace topology.  The \emph{weak topology} is defined by putting the valuation topology on $\widetilde\EE$ and giving $\widetilde\A$ the product topology; $\widetilde\A^+$ is again given the subspace topology.

Alternatively, the weak topology on $\widetilde\A$ is given by taking the sets
\[	U_{k,n}:=p^k\widetilde\A+[\widetilde p]^n\widetilde\A^+\text{ for }k,n\geq 0	\]
to be a basis of neighborhoods around $0$, where $\widetilde p\in\widetilde\EE^+$ is any fixed element with $\widetilde p^{(0)}=p$ (i.e., $\widetilde p$ is a system of compatible $p$-power roots of $p$).  The weak topology on $\widetilde\A^+$ is similarly generated by the sets $U_{k,n}\cap\widetilde\A^+ = p^k\widetilde\A^++[\widetilde p]^n\widetilde\A^+$.  This is the same as the $(p,[\widetilde p])$-adic topology on $\widetilde\A^+$.  Since $\widetilde p$ and $\overline\pi$ are both pseudo-uniformizers of $\widetilde\EE$, this is also the same as the $(p,\pi)$-adic and $(p,[\overline\pi])$-adic topologies.

Both rings carry continuous actions of Frobenius (for either topology).  However, the Galois action is continuous for the weak topology, but it is not continuous for the $p$-adic topology because the Galois actions on $\widetilde\EE^+$ and $\widetilde\EE$ are not discrete.

For $r>0$, we define
\[	\widetilde\A^{(0,r]}:=\left\{x=\sum_{k=0}^\infty p^k[x_k]\in\widetilde\A : \lim_{k\rightarrow\infty}(v_{\EE}(x_k)+k/r)=\infty\right\}	\]
We equip $\widetilde\A^{(0,r]}$ with the valuation $v^{(0,r]}(x):=\inf_{k\geq 0}\{v_{\EE}(x_k)+k/r\}$; this valuation makes it into a Banach algebra with valuation ring
\[	\widetilde\A^{\dagger,s(r)}:=\left\{x=\sum_{k=0}^\infty p^k[x_k]\in\widetilde\A : v_{\EE}(x_k)+\frac{ps(r)k}{p-1}\geq 0\text{ for all }k, \lim_{k\rightarrow\infty}(v_{\EE}(x_k)+\frac{ps(r)k}{p-1})=\infty\right\}	\]
where $s(r):=(p-1)/pr$.  We also define a valuation $v_s$ on $\widetilde\A^{\dagger,s}$ via $v_s(x):=\inf_{k\geq 0}\{v_\EE(x_k)+psk/(p-1)\}$.

We define $\widetilde\B^+:=\widetilde\A^+[1/p]$, $\widetilde\B:=\widetilde\A[1/p]$, and $\widetilde\B^{\dagger,s}:=\widetilde\A^{\dagger,s}[1/p]$.  Then we can extend $v_s$ to a valuation on $\widetilde\B^{\dagger,s}$.  In fact, if $s'\geq s$, then $v_{s'}$ extends to $\widetilde\B^{\dagger,s}$, and we define $\widetilde\B_{\rig}^{\dagger,s}$ to be the Fr\'echet completion of $\widetilde\B^{\dagger,s}$ with respect to $v_{s'}$ for all $s'\geq s$.

There are versions of all of these rings with no tilde; they are imperfect versions of the rings with tildes.

Let $\pi\in\widetilde\A$ denote $[\varepsilon]-1$, where $[\varepsilon]$ denotes the Teichm\"uller lift of $\varepsilon$.  Then there is a well-defined injective map $\mathcal{O}_F[\![X]\!][X^{-1}]\rightarrow \widetilde\A$ given by sending $X$ to $\pi$; we let $\A_F$ denote the $p$-adic completion of the image.  This is a Cohen ring for $\EE_F$.

Let $\B_F:=\A_F[1/p]$.  Then $\B_F$ is a field, and we let $\B$ be the completion of the maximal unramified extension of $\B_F$ inside $\widetilde\B$.  Further, we put $\A:=\B\cap \widetilde\A$, $\A^+:=\A\cap\widetilde\A^+$, $\B^{\dagger,s}:=\B\cap\widetilde\B^{\dagger,s}$, and $\A^{\dagger,s}:=\A\cap\widetilde\A^{\dagger,s}$.  Alternatively, $\A$ is the completion of the integral closure of the image of $\Z_p[\![X]\!][X^{-1}]$ in $\widetilde\A$.

Because extensions of $\EE_F$ correspond to unramified extensions of $\B_F$, we get a natural Galois action on $\A$ and $\A^+$.  For any extension $K/\Qp$, we may therefore define $\A_K:=\A^{H_K}$ and $\B_K:=\A_K[1/p]$.  When $K$ is unramified over $\Qp$, this agrees with our original definition of these rings.  

We also set $\B_K^{\dagger,s}:=(\B^{\dagger,s})^{H_K}$.  By~\cite[Proposition 7.5]{colmez08}, if $F'\subset K$ is the maximal unramified subfield of $K$ and $s\gg 0$ (and depending on $K$), then $\B_K^{\dagger,s}$ is isomorphic to the ring of bounded functions $f(T)\in F'[\![T]\!]$ converging on the annulus $0<v_p(T)\leq 1/s$.  We let $\B_{\rig,K}^{\dagger,s}$ be its Fr\'echet completion with respect to the (restrictions of) the valuations $v_s$; for $s\gg 0$, it is isomorphic to the ring of all functions $f(T)\in F'[\![T]\!]$ converging on the annulus $0<v_p(T)\leq 1/s$.  These rings are equipped with actions of $\Gamma_K$ and with Frobenius maps $\varphi:\B_K^{\dagger,s}\rightarrow\B_K^{\dagger,ps}$ and $\varphi:\B_{\rig,K}^{\dagger,s}\rightarrow\B_{\rig,K}^{\dagger,ps}$.

In fact, we can describe $\A_K$ more explicitly: If $F'\subset K$ is the maximal unramified subfield of $K$, then every element of $\A_K$ can be written uniquely in the form $\sum_{k\in\Z}a_k\pi_K^k$, where $\{a_k\}_{k\in\Z}\in\cO_{F'}$ is a sequence tending to $0$ as $k$ tends to $-\infty$ and $\pi_K$ is a (non-canonical) lift of a uniformizer of $\EE_K$~\cite[Proposition 7.1]{colmez08}.  The subring $\{\sum_{k\geq 0}a_k\pi_K^k\}$ is contained in $\A_K^+$; on the other hand, if $f=\sum_{k\in\Z}a_k\pi_K^k\in\A_K^+$, then $f\in\widetilde{\A}^+$ and therefore $a_k\equiv 0\pmod p$ for $k<0$ (since $\pi_K^{-1}\in\EE_K\smallsetminus\EE_K^+$). But $p\widetilde \A\cap \widetilde\A^+=p\widetilde\A^+$ (which can be seen by considering Witt vector expansions), so $\sum_{n<0}p^{-1}a_n\pi_L^n\in \A_L\cap \widetilde\A^+=\A_L^+$.  Iterating this argument, we have shown the following:

\begin{corollary}
Each element of $\A_L^+$ can be written uniquely in the form $\sum_{n\geq 0}a_n\pi_L^n$ for $a_n\in\mathcal{O}_{F'}$.
\end{corollary}

We also get an action of Frobenius induced from that on $\widetilde\A$; however, because $\EE$ is imperfect, the action of Frobenius is no longer surjective.

If $R$ is a complete local noetherian $\Z_p$-algebra with maximal ideal $\mathfrak{m}_R$ and finite residue field, we define $R\htimes\A$, $R\htimes\A^+$, $R\htimes\A_{F}$, and $R\htimes\A_F^+$ to be the tensor products over $\Z_p$, completed with respect to $\mathfrak{m}_R$.  If $R$ is discrete, these completed tensor products become ordinary tensor products over $\Z_p$.

\begin{remark}
It would be more natural to define $R\htimes\A^+$ (resp. $R\htimes\A_F^+$) to be the tensor products over $\Zp$ completed $\mathfrak{m}_R\otimes \A^++R\otimes (p,\pi)$-adically.  In fact, completing $\mathfrak{m}_R$-adically and $\mathfrak{m}_R\otimes \A^++R\otimes (p,\pi)$-adically yield isomorphic rings.

Indeed, $p\in\mathfrak{m}_R$ and $\A^+/p^n$ is flat over $\Z/p^n$.  Thus, if
\[	(\Z/p^n)^{d_n'}\rightarrow (\Z/p^n)^{d_n}\rightarrow R/\mathfrak{m}_R^n\rightarrow 0	\]
is a finite presentation of $R/\mathfrak{m}_R^n$, then
\[	(\Z/p^n)^{d_n'}\otimes_{\Z/p^n}(\A^+/p^n)\rightarrow (\Z/p^n)^{d_n}\otimes_{\Z/p^n}(\A^+/p^n) \rightarrow (R/\mathfrak{m}_R^n)\otimes_{\Z/p^n}(\A^+/p^n)\rightarrow 0	\]
is exact.  Since $(R/\mathfrak{m}_R^n)\otimes_{\Z/p^n}(\A^+/p^n)$ has no $\pi$-torsion,
\[	(\Z/p^n)^{d_n'}\otimes_{\Z/p^n}(\A^+/(p^n,\pi^{n'}))\rightarrow (\Z/p^n)^{d_n}\otimes_{\Z/p^n}(\A^+/(p^n,\pi^{n'})) \rightarrow (R/\mathfrak{m}_R^n)\otimes_{\Z/p^n}(\A^+/(p^n,\pi^{n'}))\rightarrow 0	\]
is exact for all $n'\geq 1$.  Moreover, as $n'$ varies, the transition maps $\A^+/(p^n,\pi^{n'+1})\rightarrow \A^+/(p^n,\pi^{n'})$ are surjective.  The Mittag-Leffler criterion implies that
\[	(\Z/p^n)^{d_n'}\otimes_{\Z/p^n}(\A^+/p^n)\rightarrow (\Z/p^n)^{d_n}\otimes_{\Z/p^n}(\A^+/p^n) \rightarrow \varprojlim_{n'}(R/\mathfrak{m}_R^n)\otimes_{\Z/p^n}(\A^+/(p^n,\pi^{n'}))\rightarrow 0	\]
is exact, so the natural map $(R/p^n)\otimes_{\Z/p^n}(\A^+/p^n)\rightarrow \varprojlim_{n'}(R/\mathfrak{m}_R^n)\otimes_{\Z/p^n}(\A^+/(p^n,\pi^{n'}))$ is an isomorphism.
\end{remark}

This permits us to define $\D(M)$ and $\D^+(M)$ for $R$-linear representations of $G_{\Qp}$ as follows:
\[\D(M)=\left((R\htimes\A)\otimes_R M \right)^{H_{\Qp}} \mbox{ and } \D^+(M)=\left((R\htimes\A^+)\otimes_R M \right)^{H_{\Qp}}.\]
Note that there are canonical isomorphisms
\begin{equation}
\label{f:Dinverselimits} \D(M)\cong \varprojlim_n \D(M/\mathfrak{m}_R^nM) \mbox{ and } \D^+(M)\cong \varprojlim_n \D^+(M/\mathfrak{m}_R^nM).
\end{equation}

A consideration of Witt coordinates shows that the cokernel $C$ of the injection $\widetilde{\A}^+\hookrightarrow\widetilde{\A}$ has no $p$-torsion, and is therefore $\Z_p$-flat.  It follows that the cokernel of $\A^+\hookrightarrow\A$ is also $\Z_p$-flat.  As a consequence, the natural maps $R\htimes\widetilde{\A}^+\rightarrow R\htimes\widetilde{\A}$ and $R\htimes\A^+\rightarrow R\htimes\A$ are injective.

We record the following useful lemma:
\begin{lemma}\label{lemma:a-a+-mult}
Let $R$ be a discrete $\Z_p$-algebra, and let $I\subset R$ be an ideal.  Then
\[	I(R\otimes_{\Z_p}\A)\cap(R\otimes_{\Z_p}\A^+)=I(R\otimes_{\Z_p}\A^+)	\]
inside $R\otimes_{\Z_p}\A$.
\end{lemma}
\begin{proof}
Suppose that $f:=r_1f_1+\ldots+r_kf_k\in R\otimes_{\Z_p}\A^+$, where $r_i\in I$ and $f_i\in\A$, and let $\overline f_i$ denote the image of $f_i$ in the cokernel $R\otimes_{\Z_p}C$ of the map $R\otimes_{\Z_p}\A^+\rightarrow R\otimes_{\Z_p}\A$.  By assumption, $\sum_ir_i\overline f_i=0$, so there exist $\overline f_j'\in R\otimes_{\Z_p}C$ and $r_{ij}'\in R$ such that $\overline f_i=\sum_j r_{ij}'\overline f_j'$ for all $i$ and $\sum_ir_ir_{ij}'=0$ by~\cite[Theorem 7.6]{matsumura} (since $R\otimes_{\Z_p}C$ is $R$-flat).  We may choose $\widetilde f_j'\in\A$ lifting $\overline f_j'$; setting $f_i':=\sum_jr_{ij}'\widetilde f_j'$, we have $f_i-f_i'\in\A^+$ for all $i$ and $\sum_ir_i(f_i-f_i')=\sum_ir_if_i=f$ so we are done.
\end{proof}
\begin{remark}
This does not imply that if $f\in\A$ satisfies $rf\in\A^+$ then $f\in\A^+$, merely that there is some $f'\in\A^+$ such that $rf=rf'$.
\end{remark}

Furthermore, $\A/p$ is the separable closure of $\A_F/p=k(\!(\pi)\!)$, where $k$ is the residue field of $F$, and $\A^+/p$ is its ring of integers~\cite[\textsection 1.8(a)]{fontaine90}.

\begin{lemma}\label{lemma:a-a-k}
\begin{enumerate}
\item	Let $M$ be a $\Zp$-module annihilated by $p^n$ for some $n\geq 1$.  Then for any finite extension $K/F$, the natural map $R\otimes_{\Z_p}\A_K^+\rightarrow (R\otimes_{\Z_p}\A^+)^{H_K}$ is an isomorphism.
\item	If $R$ is a discrete $\Z/p^n$-algebra and $I\subset R$ is an ideal, then $(IR\otimes_{\Z_p}\A^+)^{H_K}=IR\otimes_{\Z_p}\A_K^+$.
\item	If $M$ is a $\Zp$-module annihilated by $p^n$ for some $n\geq 1$, then every element of $M\otimes_{\Z_p}\A^+$ is contained in $M\otimes_{\Z_p}\A_K^+$ for some finite extension $K/F$.
\end{enumerate}
\end{lemma}
\begin{proof}
\begin{enumerate}
\item	We proceed by induction on $n$.  If $n=1$, the result follows by choosing a basis for $R$ over $\F_p$ and using the fact that $\A^+/p$ is the ring of integers of a separable closure of $\A_F/p$.  So we may assume the result for $n-1$ and consider the commutative diagram
\[
\xymatrix{
0\ar[r] & pM\otimes_{\Z_p}\A_K^+\ar[r]\ar[d] & M\otimes_{\Z_p}\A_K^+\ar[r]\ar[d] & (M/p)\otimes_{\Z_p}\A_K^+\ar[r]\ar[d] & 0 \\
0\ar[r] & (pM\otimes_{\Z_p}\A^+)^{H_K}\ar[r] & (M\otimes_{\Z_p}\A^+)^{H_K}\ar[r] & ((M/p)\otimes_{\Z_p}\A^+)^{H_K}
}
\]
By the inductive hypothesis, the homomorphisms $pM\otimes_{\Zp}\A_K\rightarrow(pM\otimes_{\Z_p}\A^+)^{H_K}$ and $(M/p)\otimes_{\Z_p}\A_K^+\rightarrow((M/p)\otimes_{\Z_p}\A^+)^{H_K}$ are isomorphisms.  The snake lemma then implies that the map $M\otimes_{\Z_p}\A_K^+\rightarrow(M\otimes_{\Z_p}\A^+)^{H_K}$ is an isomorphism, as well.
%
\item	This is a special case of the first part of this lemma.
\item	We proceed by induction on $n$.  If $n=1$, then $M\otimes_{\Z_p}\A^+=\varinjlim_K M\otimes_{\F_p}(\A^+/p)^{H_K}$ and the result follows.  We now assume the result holds for $n-1$.  Consider $f\in M\otimes_{\Z_p}\A^+$, and let $\overline f$ denote its image in $(M/p)\otimes_{\Z_p}\A^+$.  By the inductive hypothesis, $\overline f\in (M/p)\otimes_{\Z_p}\A_K^+$ for some finite extension $K/F$, and $\overline f$ lifts to $\widetilde f\in M\otimes_{\Z_p}\A_K^+$.  Then $f-\widetilde f$ lies in the finite $(\Z/p^{n-1})\otimes_{\Z_p}\A^+$-module $pM\otimes_{\Z_p}\A^+$, so we see that $f-\widetilde f\in M\otimes_{\Z_p}\A_{K'}^+$ for some finite extension $K'/F$ and we are done.
\end{enumerate}
\end{proof}

\section{Families of Wach modules}\label{families}

We are interested in Galois representations with coefficients in rings which are quotients of power series rings over $\Zp$.  Taking the generic fiber of the formal scheme associated to such a ring yields a rigid analytic space, so we will also need to consider Galois representations in rigid analytic families.  Such families have been studied by a number of authors, including~\cite{berger-colmez}, \cite{kedlaya-liu08}, and \cite{bellovin13}.  We will use the sheaves of period rings given in~\cite{bellovin13}.

More precisely, we make the following definition:
 \begin{definition}
  Let $K/\Q_p$ be a finite extension.  If $R$ is a complete local noetherian $\Zp$-algebra with finite residue field and the profinite topology, a \emph{family of $G_{K}$-representations} over $R$ is a finite free $R$-module $M$ endowed with a continuous $R$-linear action of $G_{K}$.  If $X$ is a rigid analytic space over $\Qp$, a \emph{family of $G_{K}$-representations} over $X$ is a finite projective $\mathcal{O}_X$-module endowed with a continuous $\mathcal{O}_X$-linear action of $G_{K}$.
 \end{definition}

 Note that to a family $M$ of $G_{K}$-representations over $R$ we can naturally assign a family $M^\rig$ of $G_{K}$-representations over $X=\Spf(R)^{\rig}$; see Appendix \ref{sec:app1} for details.

For any rigid analytic space $X$ over $\Qp$, we further define the sheaves of rings $\mathscr{B}_{X,\max}^+$ and $\mathscr{B}_{X,\max}$ by setting $\mathscr{B}_{X,\max}^+(U):=\mathcal{O}_X(U)\widehat\otimes_{\Qp}\mathbf{B}_{\max}^+$ and $\mathscr{B}_{X,\max}(U):=\cup_{i\in\NN}t^{-i}\mathscr{B}_{X,\max}^+(U)$.  If $M$ is a family of $G_{\Qp}$-representations over $X$, we define a sheaf of abelian groups $\mathscr{D}_{\cris}(M)$ by setting $\mathscr{D}_{\cris}(M)(U):=\left(M(U)\widehat\otimes_{\mathcal{O}_X(U)}\mathscr{B}_{X,\max}(U)\right)^{G_{K}}$; we let $\mathbf{D}_{\cris}(M)$ denote $\mathscr{D}_{\cris}(M)(X)$.  We say that $M$ is \emph{crystalline} if $\mathscr{D}_{\cris}(M)$ is a vector bundle over $\mathcal{O}_X$ of the same rank as $M$ and the natural map $\mathscr{B}_{X,\max}\otimes_{\mathcal{O}_X}\mathscr{D}_{\cris}(M)\rightarrow \mathscr{B}_{X,\max}\otimes_{\mathcal{O}_X}M$ is an isomorphism.

 We now fix an interval $[a, b]$ of integers.

 \begin{proposition}\label{prop:crisfam} Suppose that $R[1/p]$ is reduced, and let $M$ be a {family of $G_{K}$-representations} over $R$. Then the
   following are equivalent:
  \begin{enumerate}
   \item For every homomorphism $R \to E'$, where $E' / E$ is a finite extension, the representation $E' \otimes_R M$ is crystalline with Hodge--Tate weights in $[a,b]$.
   \item For every homomorphism $R \to S$, where $S$ is a finite flat $\Z_p$-algebra, the representation $S \otimes_R M[1/p]$ is crystalline with Hodge--Tate weights in $[a, b]$.
  \end{enumerate}
 \end{proposition}

 (The same holds if ``crystalline'' is replaced by ``semistable'', although we shall not need this.)

 \begin{proof}
  The implication (2) $\Rightarrow$ (1) is clear.
  To show that $(1) \Rightarrow (2)$ we will use the results of \cite{bellovin13}. Suppose we are given a homomorphism $R \to S$ as in $(2)$. Let $R'$ be an affinoid algebra such that there is a natural flat homomorphism $R \to R'$ with dense image. By Theorem 1.1.3 of \emph{op.cit} we know that there is a unique closed rigid-analytic subvariety of $X' = \Sp (R')$ such that a homomorphism from $R'$ to an artinian $E$-algebra factors through this quotient if and only if the corresponding specialization of $M$ is crystalline with Hodge--Tate weights in $[a,b]$. By assumption, this subvariety contains every point of $X'$.  But $X'$ is reduced, so this subvariety is the whole of $X'$. Since the generic fiber $\Spf(R)^{\rig}$ can be exhausted by affinoid subspaces as above, we see that $(1) \Rightarrow (2)$ as required.
 \end{proof}

 We say $M$ is {\it crystalline with Hodge--Tate weights in $[a,b]$} if $R$ is flat over $\Z_p$ and the above equivalent conditions hold.  This is equivalent to the family of representations $M^{\rig}$ over the quasi-Stein space $\Spf(R)^{\rig}$ being crystalline with Hodge--Tate weights in $[a,b]$.  If $R$ is $\Z_p$-finite and flat and $K=F$ is unramified over $\Q_p$, there is an associated \emph{Wach module} $\NN(M)$.  We wish to extend this to the general case.

The problem is that we would like to have a good notion of Wach modules for families of crystalline representations with coefficients in power series rings.  But because the (integral) Wach module functor can fail to be exact, naive constructions do not seem to have good properties.  For example, if $R$ is a finite flat $\Z_p$-algebra and $M$ is an $R$-linear lattice in a crystalline representation of $G_{\Qp}$, there is an associated Wach module $\NN(M)$; it is a finite $R\otimes_{\Z_p}\A_{\Qp}^+$-module, but it is in general not projective over $R\otimes_{\Z_p}\A_{\Qp}^+$.  This makes it difficult to verify that $\NN(M)$ has good base change properties, or to show that $\NN(M)$ has good finiteness properties when $R$ is a power series ring.

We give a definition of Wach modules.  As we are primarily interested in positive representations (that is, Galois representations with non-positive Hodge--Tate weights, where the cyclotomic character has weight $1$), we restrict to that case.
\begin{definition}\label{def:Wach}
A \emph{Wach module} with weights in $[-h,0]$ and coefficients in a complete local noetherian $\Z_p$-algebra $R$ with finite residue field is a finite $R\htimes_{\Z_p}\A_{F}^+$-module $N$ equipped with an $R$-linear action of $\Gamma$ such that $\Gamma$ acts trivially on $N/\pi N$, and equipped with an $R$-linear Frobenius $\varphi:N\rightarrow N$ commuting with $\Gamma$ such that the cokernel $N/\varphi^\ast(N)$ is killed by $q^h$, where $q=\varphi(\pi)/\pi$.
\end{definition}
\begin{remark}
If $R=\Z_p$, this is weaker than the definition given in~\cite{berger04}, which requires $N$ to additionally be a free $\A_F^+$-module.
\end{remark}

We are interested in Wach modules because of their connection to crystalline representations.  More precisely, if $M$ is a $\Z_p$-lattice in a crystalline representation of $G_{F}$, there is an associated Wach module $\NN(M)$, which is an $\A_{F}^+$-lattice in the $(\varphi,\Gamma)$-module $\D(M)$ (which has $\A_{F}$-coefficients).

While $\D(M)$ behaves well with respect to coefficients, by the work of~\cite{dee01}, $\D^+(M)$ does not, nor does $\NN(M)$.  To get around this problem, we follow Kisin and construct a moduli space of $\A_{F}^+$-lattices in $\A_{F}$-modules.

Recall that for a $\Z_p$-finite $G_{K}$-representation $M$, the $(\varphi,\Gamma)$-module $\D(M)$ is defined by $\D(M)=(\A\otimes_{\Z_p}M)^{H_{K}}$.  By the work of Dee, if $M$ is a finite free rank-$d$ $R$-module for a finite local $\Z_p$-algebra $R$ and the Galois action is $R$-linear, then $\D(M)$ is finite projective of rank $d$ over $R\otimes_{\Z_p}\A_{K}$.  Moreover, if $R\rightarrow R'$ is a homomorphism of such rings, then $\D(R'\otimes_RM)=R'\otimes_R\D(M)$.

\begin{lemma}\label{lemma:coeffs-a-a+-i}
Let $A$ be an artinian local $\Z_p$-algebra, and let $M$ be a finite $A$-module equipped with a continuous $A$-linear action of $G_K$.  Then
\begin{enumerate}
\item	If $M'$ is a finite $A$-module (with trivial $G_K$-action), then $\D(M\otimes_AM')\cong\D(M)\otimes_AM'$.
\item	If $B$ is an $A$-algebra, then $(\AA\otimes_{\Z_p}M_B)^{H_K}\cong \D(M)\otimes_AB$.
\end{enumerate}
\end{lemma}
\begin{proof}
\begin{enumerate}
\item	If $M'$ is free over $A$, this is clear.  Otherwise, we may choose a finite presentation of $M'$ and use the exactness of the functor $\D(-)$.
\item	We may write $B=\varinjlim_iB_i$, where the $B_i$ are $A$-finite submodules of $B$ and the transition maps are injections.  Then the functor of $H_K$-invariants commutes with the direct limit, and so does the tensor product with $\D(M)$, so this follows from the first part.
\end{enumerate}
\end{proof}
Thus, under these hypotheses we may refer to $\D(M_B)$ without any ambiguity.

We give a generalization of~\cite[Lemme III.3.2]{berger04}:
\begin{lemma}
Let $A$ be an artinian local $\Z_p$-algebra with finite residue field, let $B$ be an $A$-algebra, and let $M$ be an $A$-linear representation of $G_F$ of rank $d$.  If $N\subset \D(M_B)$ is a free $B\otimes_{\Z_p}\A_F^+$-module of rank $d$, stable by the actions of $\varphi$ and $\Gamma$, such that $(B\otimes_{\Z_p}\A_F)\otimes_{B\otimes_{\Z_p}\A_F^+}N=D(M_B)$, then for any ideal $I\subset B$,
\[	\left((B\otimes_{\Z_p}\A^+)\otimes_{B\otimes_{\Z_p}\A_F^+}N\right)\cap I\left((B\otimes_{\Z_p}\A^+)\otimes_{B}M_B\right)=I\left((B\otimes_{\Z_p}\A^+)\otimes_{B\otimes_{\Z_p}\A_F^+}N\right)	\]
\end{lemma}
\begin{proof}
The proof is essentially identical to the proof of~\cite[Lemme III.3.2]{berger04}.  A basis $\{n_i\}$ of $N$ is also a basis of $\D(M_B)$, and therefore of $(B\otimes_{\Z_p}\A)\otimes_{B\otimes_{\Z_p}\A_F}\D(M_B)=(B\otimes_{\Z_p}\A)\otimes_BM_B$.  If we write $x\in (\A^+\otimes_{\Z_F^+}N)\cap I(\A^+\otimes_{\Z_p}M_B)$ with respect to this basis as $\sum_i x_in_i$ we see that $x_i\in I\A$.  But $x_i\in\A^+$, as well, so $x_i\in I\A^+$, by Lemma~\ref{lemma:a-a+-mult}.
\end{proof}

\begin{definition}
Let $A$ be an artinian local $\Z_p$-algebra with finite residue field, and let $M$ be a finite free $A$-module of rank $d$ equipped with a continuous $A$-linear action of $G_F$.  An $A\otimes_{\Z_p}\A_{F}^+$-lattice in $\D(M)$ of height $\leq h$ is a finite projective $A\otimes_{\Z_p}\A_{F}^+$-submodule $\NN$ of rank $d$ which generates $\D(M)$ as an $A\otimes_{\Z_p}\A_{F}$-module, such that $\NN$ is $\varphi$- and $\Gamma$-stable, the cokernel of $\varphi^\ast\NN\rightarrow \NN$ is killed by $q^h$, and $\Gamma$ acts trivially on $\NN/\pi\NN$.

We define a functor on $A$-algebras by letting $L^{\leq h}_{M}(B)$ denote the set of $B\otimes_{\Z_p}\A_{F}^+$-lattices in $\D(M\otimes_AB)$ of height $\leq h$.
\end{definition}

\begin{proposition}
The functor $L^{\leq h}_M$ is representable by a projective $A$-scheme $\mathscr{L}^{\leq h}_M$.  If $A\rightarrow A'$ is a map of local artinian rings with finite residue field, then there is a canonical isomorphism $\mathscr{L}^{\leq h}_M\times_AA'\rightarrow \mathscr{L}^{\leq h}_{M_{A'}}$.  Moreover, $\mathscr{L}^{\leq h}_M$ is equipped with a canonical very ample line bundle.
\end{proposition}
\begin{proof}
The proof is essentially identical to the proof of~\cite[Proposition 1.3]{kisin-pst-def-rings} and~\cite[Proposition 2.1.7]{kisin-bt}.  We begin by noting that since $p$ is torsion in $A$, $A\otimes_{\Z_p}\A_{F}=(A\otimes_{\Z_p}\cO_F)[\![\pi]\!][\pi^{-1}]$.  Choosing a basis for $\D(M)$ yields a distinguished projective rank-$d$ $A\otimes_{\Z_p}\A_{F}^+$-module $\NN_0\subset\D(M)$ which spans $\D(M)$, and every $A\otimes_{\Z_p}\A_{F}^+$-lattice $\NN$ satisfies $\pi^{-i}\NN_0\subset \NN\subset \pi^i\NN_0$ for some $i$.  Thus, all $B\otimes_{\Z_p}\A_{F}^+$-lattices are points of the affine Grassmannian for $\GL_d$ over $A$, which is an ind-projective scheme.

Furthermore, $\varphi$- and $\Gamma$-stability define closed conditions on the affine Grassmannian, as do the requirements that $\NN/\varphi^\ast\NN$ be killed by $q^h$ and that $\Gamma$ act trivially on $\NN/\pi\NN$.  Finally, the $\varphi$-stability of $A\otimes_{\Z_p}\A_{F}^+$-lattices implies that if $r$ is the least integer such that $\pi^r\NN_0\subset (1\otimes\varphi)\varphi^\ast(\NN_0)\subset \pi^{-r}\NN_0$, $s$ is the least integer such that $p^s=0$ in $A$, and $i=\max\{\frac{(p-1)sh+r}{p-1},\frac{r}{p-1}\}$, then $\pi^{-i}\NN_0\subset \NN\subset \pi^i\NN_0$ for every $A\otimes_{\Z_p}\A_{F}^+$-lattice $\NN$.

Indeed, let $B$ be an $A$-algebra, let $\NN$ be a $B\otimes_{\Z_p}\A_{F}^+$-lattice, and let $i$ be the smallest integer such that $B\htimes_A\NN_0\subset \pi^{-i}\NN$.  Then the smallest integer $j$ such that $(1\otimes\varphi)\varphi^\ast(B\htimes_A\NN_0)\subset \pi^{-j}(1\otimes\varphi)\varphi^\ast\NN$ is $ip$.  Moreover,
\[	(1\otimes\varphi)\varphi^\ast(B\htimes_A\NN_0)\subset \pi^{-r}(B\htimes_A\NN_0)\subset \pi^{-i-r}\NN = q^{-h}\pi^{-i-r}(q^h\NN)\subset \pi^{-i-r-h(p-1)s}(1\otimes\varphi)\varphi^\ast\NN	\]
since $\pi^{(p-1)s}\in q^h(A\otimes_{\Z_p}\A_F^+$.  Thus, $ip\leq i+r+h(p-1)s$, and so $i\leq \frac{r+h(p-1)s}{p-1}$.

Similarly, if $i$ is the smallest integer such that $\NN\subset\pi^{-i}(B\htimes_A\NN_0)$, then
\[	(1\otimes\varphi)\varphi^\ast(\NN)\subset\NN\subset\pi^{-i}(B\htimes_A\NN_0)\subset \pi^{-i-r}(1\otimes\varphi)\varphi^\ast(B\htimes_A\NN_0)	\]
so $ip\leq i+r$ and $i\leq\frac{r}{p-1}$.

  Therefore, $L_M^{\leq h}$ is representable by a closed subscheme of a Grassmannian, which is projective with a canonical very ample line bundle.
\end{proof}

Let $\mathscr{N}^{\leq h}(M)$ denote the universal $\mathcal{O}_{\mathscr{L}_M^{\leq h}}\otimes_{\Z_p}\A_{F}^+$-lattice on $\mathscr{L}_M^{\leq h}$.  It is a sheaf of coherent $\mathcal{O}_{\mathscr{L}_M^{\leq h}}\otimes_{\Z_p}\A_{F}^+$-modules, and on any open affine $\Spec A'\subset\mathscr{L}_M^{\leq h}$, it is an $A'\otimes_{\Z_p}\A_{F}^+$-lattice; in particular, it is equipped with actions of $\varphi$ and $\Gamma$.   We write the structure morphism $\Theta_A:\mathscr{L}_M^{\leq h}\rightarrow \Spec A$, and we write $\widetilde A:=\Gamma(\mathscr{L}_M^{\leq h},\Theta_{A*}\mathcal{O}_{\mathscr{L}_M^{\leq h}})$.

\begin{lemma}\label{lemma:wach-rep-artin}
There is a canonical isomorphism $M_{\widetilde A}\cong\left(\A\otimes_{\A_F^+}\Theta_{A\ast}\mathscr{N}^{\leq h}(M)\right)^{\varphi=1}$.
\end{lemma}
\begin{proof}
By construction,
\[	\Theta_A^\ast\A_F\otimes_{\Theta_A^\ast\A_F^+}\mathscr{N}^{\leq h}(M)\cong \D(\mathcal{O}_{\mathscr{L}_M^{\leq h}}\otimes_AM)	\]
so
\[	\Theta_A^\ast\A\otimes_{\Theta_A^\ast\A_F^+}\mathscr{N}^{\leq h}(M)\xrightarrow{\sim} \Theta_A^\ast\A\otimes_{\mathcal{O}_{\mathscr{L}_M^{\leq h}}}\Theta_A^\ast M	\]
Pushing forward yields
\[	\A\otimes_{\A_F^+}\Theta_{A\ast}(\mathscr{N}^{\leq h}(M))\xrightarrow{\sim}\Theta_{A\ast}\left(\Theta_A^\ast\A\otimes_{\Theta_A^\ast\A_F^+}\mathscr{N}^{\leq h}(M)\right) \xrightarrow{\sim} \A\otimes_{\Zp}M_{\widetilde A}	\]
by the projection formula, and taking the $\varphi=1$ part on each side yields the desired result.
\end{proof}

\begin{lemma}\label{lemma:wach-loc-free}
There is a finite covering $\{U_i=\Spec A_i\}$ of $\mathscr{L}_M^{\leq h}$ by open affine subsets such that $\mathscr{N}^{\leq h}(M)(U_i)$ is a free $A_i\otimes_{\Z_p}\A_F^+$-module of rank $d$.
\end{lemma}
\begin{proof}
Choose any affine covering $\{U_j'=\Spec A_j'\}$ of $\mathscr{L}_M^{\leq h}$; by construction, $\mathscr{N}^{\leq h}(M)(U_j')$ is a finite projective $A_j'\otimes_{\Z_p}\A_F^+$-module.  Every maximal ideal of $A_j'\otimes_{\Z_p}\A_F^+$ is of the form $(\mathfrak{m}',\pi)$, where $\mathfrak m'$ is a maximal ideal of $A_j'\otimes_{\Z_p}\cO_F$, and $\mathfrak{m}'\cap A_j'$ is a maximal ideal $\mathfrak{m}$ such that $\mathfrak{m}(A_j'\otimes_{\Z_p}\cO_F)=\prod_k\mathfrak{m}_k$ is a finite product of maximal ideals of $A_j'\otimes_{\Z_p}\cO_F$.  Therefore, the localization $(A_j'\otimes_{\Z_p}\A_F^+)\otimes_{A_j'}{(A_j')}_{\mathfrak{m}}$ has Jacobson radical generated by $(\mathfrak{m},\pi)$.  Then Nakayama's lemma implies that $\mathscr{N}^{\leq h}(M)(U_j')\otimes_{A_j'}{(A_j')}_{\mathfrak{m}}$ is free of rank $d$, so there is an open affine neighborhood $U_{j,\mathfrak{m}}=\Spec A_{j,\mathfrak{m}}\subset U_j'$ of the point of $\mathscr{L}_M^{\leq h}$ corresponding to $\mathfrak{m}$ such that $\mathscr{N}^{\leq h}(M)(U_{j,\mathfrak{m}})$ is free of rank $d$ over $A_{j,\mathfrak{m}}\otimes_{\Z_p}\A_F^+$.
\end{proof}

\begin{proposition}\label{prop:wach-d+-inclusions-artin}
There are natural inclusions of sheaves $\pi^h\left((\mathcal{O}_{\mathscr{L}_M^{\leq h}}\otimes_{\Z_p}\A^+)\otimes_AM\right)\subset \A^+\otimes_{\A_{F}^+}\mathscr{N}^{\leq h}(M)\subset (\mathcal{O}_{\mathscr{L}_M^{\leq h}}\otimes_{\Z_p}\A^+)\otimes_AM$, and therefore natural inclusions of sheaves
$\pi^h\D^+(\mathcal{O}_{\mathscr{L}_M^{\leq h}}\otimes_AM)\subset \mathscr{N}^{\leq h}(M)\subset \D^+(\mathcal{O}_{\mathscr{L}_M^{\leq h}}\otimes_AM)$.
\end{proposition}
\begin{proof}
We first need to check that the presheaf $U\mapsto \A^+\otimes_{\A_{F}^+}\mathscr{N}^{\leq h}(M)(U)$ is actually a sheaf.  But this follows from Lemma~\ref{lemma:a-a-k} and the description of $\A_K^+$ in \textsection~\ref{section:rings} (which implies that $A\otimes\A_K^+$ is finite flat over $A\otimes\A_F^+$).

By Lemma~\ref{lemma:wach-loc-free}, $\mathscr{N}^{\leq h}(M)$ is free locally on $\mathscr{L}_M^{\leq h}$; let $\{U_i\}$ be an affine cover of $\mathscr{L}_M^{\leq h}$ such that $\mathscr{N}^{\leq h}(M)(U_i)$ is free for all $i$.

Let $X_{U_i}\in \Mat_{d\times d}(\mathcal{O}_{\mathscr{L}_M^{\leq h}}(U_i)\otimes_{\Z_p}\A)$ be a matrix of a basis of $\mathscr{N}^{\leq h}(M)(U_i)$ with respect to a basis of $\mathcal{O}_{\mathscr{L}_M^{\leq h}}(U_i)\otimes_AM$.  Since $p$ is torsion in $A$, $\mathcal{O}_{\mathscr{L}_M^{\leq h}}(U_i)\otimes_{\Z_p}\A=\mathcal{O}_{\mathscr{L}_M^{\leq h}}(U_i)\otimes_{\Z_p}\A^+[1/\pi]$.  Moreover, there is some integer $s>0$ such that $\varphi^s(\pi)\in \pi^p(\mathcal{O}_{\mathscr{L}_M^{\leq h}}(U_i)\otimes_{\Z_p}\A^+)$, because $\varphi^s(\pi)=(1+\pi)^{p^s}-1$ and $\ord_p\binom{p^s}{n}=s$ for $1\leq n\leq p-1$; let $P_{U_i}$ be the matrix of $\varphi^s$ with respect to the chosen basis of $\mathscr{N}^{\leq h}(M)(U_i)$.  Since the action of Frobenius is trivial on $M$, we see that $X_{U_i}P_{U_i}=\varphi^s(X_{U_i})$, so $\varphi^s(\pi^hX_{U_i}^{-1})=((\varphi^{s-1}(q)\cdots\varphi(q)q)^hP_{U_i}^{-1})(\pi^hX_{U_i}^{-1})$.

If $X_{U_i}$ has a coefficient in $\pi^{-r}\mathcal{O}_{\mathscr{L}_M^{\leq h}}(U_i)\otimes_{\Z_p}\A^+\smallsetminus\pi^{-r+1}\mathcal{O}_{\mathscr{L}_M^{\leq h}}(U_i)\otimes_{\Z_p}\A^+$ for $r\geq 0$ (and no coefficients with poles of higher order), then $\varphi^s(X_{U_i})$ has a  coefficient in $\pi^{-rp}\mathcal{O}_{\mathscr{L}_M^{\leq h}}(U_i)\otimes_{\Z_p}\A^+\smallsetminus\pi^{-rp+1}\mathcal{O}_{\mathscr{L}_M^{\leq h}}(U_i)\otimes_{\Z_p}\A^+$, whereas $X_{U_i}P_{U_i}$ has coefficients in $\pi^{-r}\mathcal{O}_{\mathscr{L}_M^{\leq h}}(U_i)\otimes_{\Z_p}\A^+$.  It follows that $r=0$ and $\mathscr{N}^{\leq h}(M)(U_i)\subset \A^+\htimes_{\Zp}(\mathcal{O}_{\mathscr{L}_M^{\leq h}}(U_i)\otimes_AM)$.

The other inclusion follows similarly.  Since $\mathscr{N}^{\leq h}(M)/\varphi^\ast\mathscr{N}^{\leq h}(M)$ is killed by $q^h$ by assumption, $((\varphi^{s-1}(q)\cdots\varphi(q)q)^hP_{U_i}^{-1})$ has coefficients in $\mathcal{O}_{\mathscr{L}_M^{\leq h}}(U_i)\otimes_{\Z_p}\A_{F}^+$.  Now if $\pi^hX_{U_i}^{-1}$ has a coefficient with a pole of order $r$ (and no coefficients with poles of higher order), then $\varphi^s(\pi^hX_{U_i}^{-1})$ has a coefficient with a pole of order at least $pr$, whereas $((\varphi^{s-1}(q)\cdots\varphi(q)q)P_{U_i}^{-1})(\pi^hX_{U_i}^{-1})$ has coefficients in $\pi^{-r}\mathcal{O}_{\mathscr{L}_M^{\leq h}}(U_i)\otimes_{\Z_p}\A^+$.  It follows that $r=0$ and therefore $\pi^h\left((\mathcal{O}_{\mathscr{L}_M^{\leq h}}(U_i)\otimes_{\Zp}\A^+)\otimes_AM\right)\subset \A^+\htimes_{\A_{F}^+}\mathscr{N}^{\leq h}(M)(U_i)$.

Now we take $H_{F}$-invariants to obtain the second pair of inclusions.  To see that $\left(\A^+\otimes_{\A_{F}^+}\mathscr{N}^{\leq h}(M)\right)^{H_F}=\mathscr{N}^{\leq h}(M)$, we observe that the presheaf $U\mapsto \left(\A^+\otimes_{\A_{F}^+}\mathscr{N}^{\leq h}(M)(U)\right)^{H_F}$ is a sheaf on $\mathscr{L}_M^{\leq h}$ and the set of open affine subsets $U\subset \mathscr{L}_M^{\leq h}$ such that $\mathscr{N}^{\leq h}(M)(U)$ is free is a basis for the Zariski topology.  But if $\mathscr{N}^{\leq h}(M)(U)$ is free, it is clear that $\left(\A^+\otimes_{\A_{F}^+}\mathscr{N}^{\leq h}(M)(U)\right)^{H_F}=\mathscr{N}^{\leq h}(M)(U)$, so we are done.
\end{proof}

\begin{lemma}\label{lemma:artin-wach-pi-h}
Let $A\rightarrow A'$ be a local homomorphism of local artinian rings with finite residue fields.  Then there is a homomorphism $A'\otimes_A\Gamma(\mathscr{L}_M^{\leq h},\mathscr{N}^{\leq h}(M))\rightarrow\Gamma(\mathscr{L}_{M_{A'}}^{\leq h},\mathscr{N}^{\leq h}(M_{A'}))$, and the image contains the image of $\pi^h\D^+(M_{A'})$.
\end{lemma}
\begin{proof}
By the  base change properties of $\Theta_A$ and $\mathscr{L}_M^{\leq h}$ we have a commutative diagram
\[
\xymatrix{
A'\otimes_A\pi^h((A\otimes_{\Z_p}\A^+)\otimes_AM)\ar[r]\ar@{=}[d] & A'\otimes_A\pi^h((\widetilde A\otimes_{\Z_p}\A^+)\otimes_AM)\ar[r]\ar[d] & A'\otimes_A\left(\A^+\otimes_{\A_{F}^+}\mathscr{N}^{\leq h}(M)(\mathscr{L}_M^{\leq h})\right)\ar[d]	\\
\pi^h((A'\otimes_{\Z_p}\A^+)\otimes_{A'}M_{A'})\ar[r] & \pi^h((\widetilde A'\otimes_{\Z_p}\A^+)\otimes_{A'}M_{A'})\ar@{^{(}->}[r] & \A^+\otimes_{\A_{F}^+}\mathscr{N}^{\leq h}(M_{A'})(\mathscr{L}_{M_{A'}}^{\leq h})
}
\]
We first observe that if $A'$ is a finite free $A$-module, then $\D^+(M_{A'})=A'\otimes_A\D^+(M)$ and the result follows by taking $H_F$-invariants of the diagram.  Thus, we may assume that $A\rightarrow A'$ induces an isomorphism on residue fields.  Indeed, if $k'$ denotes the residue field of $A'$, then $W(k')$ is a finite free $\Z_p$-module and $A'$ is a $W(k')$-algebra by Hensel's lemma.  Then $W(k')\otimes_{\Z_p}A$ is a finite free $A$-module and there is a natural homomorphism $W(k')\otimes_{\Z_p}A\rightarrow A'$; the source is semi-local, but its local summands are summands of a finite free module, hence projective, hence free, and they have residue field $k'$.

Moreover, for some $N\gg0$ and some $k\geq 0$, there is a surjection $A[X_1,\ldots,X_k]/(X_1,\ldots,X_k)^N\twoheadrightarrow A'$ sending the $X_i$ to elements of $\mathfrak{m}_{A'}$.  Since $A[X_1,\ldots,X_k]/(X_1,\ldots,X_k)^N$ is finite free over $A$, we may assume $A\rightarrow A'$ is surjective.

So suppose that $A\rightarrow A'$ is surjective with kernel $I$.
Given an element $\beta\in\A^+\otimes_{\A_{F}^+}\Gamma(\mathscr{L}_{M_A'}^{\leq h},\mathscr{N}^{\leq h}(M_{A'}))$ in the image of $\pi^h\D^+(M_{A'})$, by the commutativity of the above diagram it lifts to an element $\widetilde\beta\in\A^+\otimes_{\A_{F}^+}\Gamma(\mathscr{L}_M^{\leq h},\mathscr{N}^{\leq h}(M))$.  Since $\widetilde \beta$ is defined up to an element of $\pi^hI(\A^+\otimes_AM)$, we need to show that
\[	\widetilde\beta\in \mathscr{N}^{\leq h}(M)(\mathscr{L}_M^{\leq h})+I(A\otimes_{\Z_p}\A^+)\otimes_{A\otimes_{\Z_p}\A_F^+}\mathscr{N}^{\leq h}(M)(\mathscr{L}_M^{\leq h})	\]

 In order to show this, we first note that for any $h\in H_F$,
\[	h(\widetilde\beta)-\widetilde\beta\in I((A\otimes_{\Z_p}\A^+)\otimes_AM)\cap \A^+\otimes_{\A_F^+}\mathscr{N}^{\leq h}(M)(\mathscr{L}_M^{\leq h})	\]
since $\beta$ is fixed by $H_F$.  If $U=\Spec A_U\subset \mathscr{L}_M^{\leq h}$ is an open affine subspace such that $\mathscr{N}^{\leq h}(M)(U)$ is free (compare Lemma  \ref{lemma:wach-loc-free}), then $h(\widetilde\beta)-\widetilde\beta|_U\in I(\A^+\otimes_{\A_F^+}\mathscr{N}^{\leq h}(M)(U))$ by Lemma~\ref{lemma:coeffs-a-a+-i}.  Indeed, if $n_1,\ldots,n_d$ is a basis for $\mathscr{N}^{\leq h}(M)(U)$ and $\widetilde\beta=\sum_ib_i\otimes n_i$ for some $b_i\in A_U\otimes_{\Z_p}\A^+$, then $h(b_i)-b_i\in I(A_U\otimes_{\Z_p}\A^+)$.  Since $(A_U/I)\otimes_{\Z_p}\A_F^+\rightarrow ((A_U/I)\otimes_{\Z_p}\A^+)^{H_F}$ is an isomorphism, $b_i\in A_U\otimes_{\Z_p}\A_F^+ + IA_U\otimes_{\Z_p}\A^+$.  Therefore, \[\widetilde\beta|_U\in \mathscr{N}^{\leq h}(M)(U))+ I(A_U\otimes_{\Z_p}\A^+)\otimes_{A_U\otimes_{\Z_p}\A_F^+}\mathscr{N}^{\leq h}(M)(U)).\]

 Now consider $\widetilde\beta|_{U\cap U'}$.  If $\widetilde\beta|_U=n_U+ b_U$ and $\widetilde\beta|_{U'}=n_{U'}+ b_{U'}$, with $n_U,n_{U'}\in \mathscr{N}^{\leq h}(M)(U\cap U'))$ and $b_U,b_{U'}\in I(A_{U\cap U'}\otimes_{\Z_p}\A^+)\otimes_{A_{U\cap U'}\otimes_{\Z_p}\A_F^+}\mathscr{N}^{\leq h}(M)(U\cap U'))$, then $n_U-n_{U'}=b_{U'}-b_U$ as elements of the free $(A_{U\cap U'}\otimes_{\Z_p}\A^+)$-module $(A_{U\cap U'}\otimes_{\Z_p}\A^+)\otimes_{A_{U\cap U'}\otimes_{\Z_p}\A_F^+}\mathscr{N}^{\leq h}(M)(U\cap U'))$.  But $n_U-n_{U'}$ is fixed by $H_F$, by construction, so $b_{U'}-b_U$ must have coefficients in $(I(A_{U\cap U'}\otimes_{\Z_p}\A^+))^{H_F}=I(A_{U\cap U'}\otimes_{\Z_p}\A_F^+)$.   Thus, $U\cap U'\mapsto b_{U'}-b_U$ is a $1$-cocycle valued in $I\mathscr{N}^{\leq h}(M)$, and when viewed as a $1$-cocycle valued in $I(\A^+)\otimes_{\A_F^+}\mathscr{N}^{\leq h}(M))$, it is trivial (since it is the $1$-coboundary of $U\mapsto b_U$).  But all of the $b_U$ have coefficients in $IA\otimes_{\Z_p}\A_K^+$ for some finite extension $K/F$ by Lemma~\ref{lemma:a-a-k}, and $\A_K^+$ is a finite free $\A_F^+$-module by the description of $\A_K^+$ in~\ref{section:rings}.  Therefore, there exist $\widetilde b_U\in IA\otimes_{\Z_p}\A_F^+$ such that the $1$-coboundary $U\cap U'\mapsto b_{U'}-b_U$ is equal to the $1$-coboundary $U\cap U'\mapsto \widetilde b_{U'}-\widetilde b_U$ (we may choose $\widetilde b_U$ by projecting the coefficients of $b_U$ down to $IA\otimes_{\Z_p}\A_F^+$).  Now we rewrite $\widetilde\beta|_U=(n_U+\widetilde b_U)+(b_U-\widetilde b_U)$ and we see that $\{n_U+\widetilde b_U\}$ glues to a global section of $\mathscr{N}^{\leq h}(M)$, so we are done.
\end{proof}

Now we wish to consider the situation when $A$ is a complete local noetherian $\Z_p$-algebra with finite residue field and maximal ideal $\mathfrak{m}_A$.  We may extend the definition of the functor $L_M^{\leq h}$ to $A$-algebras $B$ such that $\mathfrak{m}_A^iB=0$ for some $i$ in a natural way.  It follows from formal GAGA that $L_M^{\leq h}$ is also representable in this setting by a projective $A$-scheme.  Namely, we consider the formal scheme $\widehat{\mathscr{L}}_M^{\leq h}:=\{\mathscr{L}^{\leq h}_{M_{A/\mathfrak{m}^i}}\}$, which is equipped with a morphism $\widehat{\Theta}_{A\ast}:\widehat{\mathscr{L}}_M^{\leq h}\rightarrow\Spf(A)$; it is a formal scheme equipped with a very ample line bundle, so it is the $\mathfrak{m}_A$-adic completion of a projective $A$-scheme $\mathscr{L}_M^{\leq h}$.

As in the artinian case, we make the following pair of definitions:
\begin{definition}
Let $\Theta_A:\mathscr{L}_M^{\leq h}\rightarrow \Spec A$ denote the structure morphism, and let $\widetilde A$ denote the global sections $\Gamma(\mathscr{L}_M^{\leq h},\Theta_{A\ast}\mathcal{O}_{\mathscr{L}_M^{\leq h}})$.
\end{definition}

Then $\widehat{\mathscr{L}}_M^{\leq h}$ carries a universal $\mathcal{O}_{\widehat{\mathscr{L}}_M^{\leq h}}\widehat\otimes\A_F^+$-lattice $\widehat{\mathscr{N}}^{\leq h}(M)$.  We may view $\widehat{\mathscr{N}}^{\leq h}(M)$ as a formal coherent sheaf on $\widehat{\mathscr{L}}_M^{\leq h}\times \Spf \A_F^+\cong \widehat{\mathscr{L}}_M^{\leq h}\times_{\Spf A} \Spf (A\widehat\otimes\A_F^+)$, and by formal GAGA, $\widehat{\mathscr{N}}^{\leq h}(M)$ is the completion of a coherent sheaf (in fact, a vector bundle) on $\mathscr{L}_M^{\leq h}\times_{\Spec A}\Spec(A\widehat\otimes\A_F^+)$.

If $A\rightarrow A'$ is a local homomorphism of local rings and $A'$ is $\mathfrak{m}_A$-adically complete, then an $A'$-point of $\mathscr{L}_M^{\leq h}$ induces a $\{A'/\mathfrak{m}_A^i\}$-point of $\widehat{\mathscr{L}}_M^{\leq h}$, and therefore a system of $(A'/\mathfrak{m}_A^i)\otimes_{\Z_p}\A_F^+$-lattices in $\{\D(M_{A'/\mathfrak{m}^i})\}$.  But this is the same as an $A'\widehat\otimes\A_F^+$-lattice in $\D(M_{A'})$, where $A'$ is given the $\mathfrak{m}_A$-adic topology.  Thus, we may view $\mathscr{L}_M^{\leq h}$ as a moduli space of $A\widehat\otimes\A_F^+$-lattices in $\D(M)$.  In particular, $\Z_p$-points of $\mathscr{L}_M^{\leq h}$ (if they exist) correspond to Wach modules of $M_{\Z_p}$ in the sense defined in~\cite{berger04}, i.e., satisfying   the properties given in Definition~\ref{def:Wach}.

\begin{lemma}\label{lemma:wach-rep}
There is a canonical isomorphism $M_{\widetilde A}\cong\left(\A\otimes_{\A_F^+}\Theta_{A\ast}\mathscr{N}^{\leq h}(M)\right)^{\varphi=1}$.
\end{lemma}
\begin{proof}
This follows from Lemma~\ref{lemma:wach-rep-artin} by taking limits.
\end{proof}

\begin{lemma}\label{lemma:rat-wach-loc-free}
Let $A$ be a finite flat $\Z_p$-algebra, let $B:=A[1/p]$, and let $M$ be a finite free $A$-module of rank $d$ equipped with a continuous $A$-linear action of $G_F$ such that the underlying $\Z_p$-linear Galois representation is a lattice in a crystalline representation with Hodge--Tate weights in $[-h,0]$.  Then $\Theta_{A\ast}\mathscr{N}^{\leq h}(M)[1/p]$ is free over $B\otimes_{\Qp}\B_F^+$ of rank $d$.
\end{lemma}
\begin{proof}
We first forget the $A$-linear structures on $M$ and $\Theta_{A\ast}\mathscr{N}^{\leq h}(M)$ and consider them as $\Z_p$- and $\A_F^+$-modules, respectively.  Then $\mathscr{L}_M^{\leq h}\rightarrow \Spec\Z_p$ has a unique point in the fiber over $\Spec\Qp$.  Indeed, any $\Qp$-point extends to a $\Z_p$-point of $\mathscr{L}_M^{\leq h}$, which induces an $\A_F^+$-submodule $\NN$ of $\D(M)$ with actions of $\varphi$ and $\Gamma$, such that $\NN/\varphi^\ast$ is annihilated by $q^h$ and $\Gamma$ acts trivially on $\NN/\pi$.  Therefore, $\Theta_{\Z_p\ast}\mathscr{N}^{\leq h}(M)_{\Qp}$ is the unique rational Wach module $\NN(M[1/p])$ constructed in~\cite{berger04}.

Next, recall that $\D_{\cris}(M_B)\cong\NN(M)[1/p]/\pi\NN(M)[1/p]$.  But by~\cite[Proposition 4.1.3]{bellovin13}, $\Dcris(M)$ is free over $B\otimes_{\Qp}F$ of rank $d$, so we are done by Nakayama's lemma.
\end{proof}

\begin{lemma}\label{lemma:wach-saturate}
Let $N$ be a finite torsion-free $\A_{F}^+$-module of generic rank $d$ and define $N^{\rm{sat}}:=N[1/p]\cap (N\otimes_{\A_{F}^+}\A_{F})$, where the intersection is taken inside the finite $\B_{F}$-vector space $(N\otimes_{\A_{F}^+}\A_{F})[1/p]$.  Then $N^{\rm{sat}}$ is a finite free $\A_{F}^+$-module of rank $d$.
\end{lemma}
\begin{proof}
Since $N$ is finite and torsion-free over $\A_F^+$ and $\A_F^+$ is a regular two-dimensional local ring, it follows from ~\cite[Lemmes 5 and 6]{serre95} that there is an embedding $N\hookrightarrow N'$, where $N'$ is a finite free $\A_F^+$-module.  Then $N^{\rm{sat}}\hookrightarrow (N')^{\rm{sat}}=N'$, which implies that $N^{\rm{sat}}$ is a finite $\A_F^+$-module.  It follows that $N^{\rm{sat}}/\pi N^{\rm{sat}}$ is a finite $\mathcal{O}_F$-module.  Furthermore, the kernel of $N^{\rm{sat}}\rightarrow N[1/p]/\pi$ is $\pi N^{\rm{sat}}$ since $\pi$ is invertible in $\A_F$, which implies that $N^{\rm{sat}}/\pi N^{\rm{sat}}$ is $p$-torsion-free.  Hence $N^{\rm{sat}}/\pi N^{\rm{sat}}$ is a finite free $\mathcal{O}_F$-module, so Nakayama's lemma implies that there is  a surjection $(\A_{F}^+)^{\oplus d'}\twoheadrightarrow N^{\rm{sat}}$ for some $d'$.  The kernel of this surjection must be $p$-torsion, hence trivial, so the surjection must be an isomorphism and we must have $d'=d$.
%
%
\end{proof}

The key result is the following, which is the same argument as in the proof of~\cite[Proposition 1.6.4]{kisin-pst-def-rings}:
\begin{proposition}
Let $A$ be a complete local noetherian $\Z_p$-algebra together with a finite free $A$-module $M$ of rank $d$, and a continuous $A$-linear action of $G_{F}$ on $M$.
\begin{enumerate}
\item	The map $\Theta_A:\mathscr{L}_M^{\leq h}\rightarrow \Spec A$ becomes a closed immersion after inverting $p$.
\item	Let $A^{\leq h}$ denote the quotient of $A$ corresponding to the scheme-theoretic image of $\Theta_A$.  Then for any $\Qp$-finite artinian ring $B$, a map $A\rightarrow B$ factors through $A^{\leq h}$ if and only if $M_B$ is crystalline with Hodge--Tate weights in $[-h,0]$.
\end{enumerate}
\end{proposition}
\begin{proof}
The first part essentially follows from the uniqueness of Wach modules for crystalline representations of $G_{F}$ cf.~\cite[Proposition III.4.2]{berger04}.  Indeed, if $B$ is a $\Qp$-finite local artinian ring, a $B$-point $x:\Spec B\rightarrow\Spec A[1/p]$ is induced by a homomorphism $A\rightarrow B^{\circ}$, where $B^{\circ}\subset B$ is $\Z_p$-finite and flat.  Any $B$-point of $\mathscr{L}_M^{\leq h}$ over $x$ extends to a $B^\circ$-point, by the valuative criterion for properness.  This corresponds to a system of $(B^\circ/\mathfrak{m}_A^n)\otimes_{\Z_p}\A_F^+$-lattices in $\D(M\otimes_AB^\circ/\mathfrak{m}_A^n)$, and therefore induces a $B^\circ\otimes_{\Zp}\A_F^+$-submodule of $\D(M\otimes_AB^\circ)$.  Inverting $p$ and forgetting about the $B$-linear structure, we obtain a $\B_F^+$-lattice in $\D(M\otimes_RB)$.  But there is at most one such lattice, by~\cite[Proposition III.4.2]{berger04}.

Thus, $\Theta_A[1/p]:\mathscr{L}_M^{\leq h}[1/p]\rightarrow \Spec A[1/p]$ is quasi-finite; since it is projective, it is therefore finite (and is an injection on closed points).  Setting $B=E[\varepsilon]/\varepsilon^2$, it also induces injections on tangent spaces, and is therefore a closed immersion.

For the second part, we first consider a $\Qp$-finite local artinian ring $B$ with residue field $E$ together with a map $\Spec B\rightarrow \Spec A$ which factors through the image of $\mathscr{L}_M^{\leq h}$.  We obtain a Wach module with coefficients in $B\otimes_{\Z_p}\A_{F}^+$.  Forgetting the $B$-linear structure, we get a Wach module, so again by ~\cite[Proposition III.4.2]{berger04}, $M_B$ is crystalline with Hodge--Tate weights in $[-h,0]$.

It remains to show that if $M_B$ is crystalline with Hodge--Tate weights in $[-h,0]$, it is induced by a point of $\mathscr{L}_M^{\leq h}$.  For this, we introduce some new notation:  Let $B^\circ\subset B$ denote the subring of elements mapping to $\mathcal{O}_E\subset E$, and we let $\mathrm{Int}_B$ denote the set of $\mathcal{O}_E$-finite subrings of $B^\circ$ (this makes sense because $B$ is canonically an $E$-algebra).  Each such $C$ is a finite flat $\mathcal{O}_E$-algebra, since $B$ contains no $p$-torsion.

Observe that the map $A\rightarrow B$ factors through some $C\in\mathrm{Int}_B$, and we may assume that the map $C\rightarrow \cO_E$ is a surjection.  Then $M_C$ is a lattice in the crystalline representation $M_B$, so there is some Wach module $\NN(M_C)\subset \D(M_C)$.  It is an $\A_{F}^+$-lattice in $\D(M_C)$ (forgetting the $C$-linear structure), and $\NN(M_C)[1/p]$ is projective of rank $d$ over $(C\otimes_{\Z_p}\A_{F}^+)[1/p]$ by Lemma~\ref{lemma:rat-wach-loc-free}.  Let $\NN_{\mathcal{O}_E}'$ be the image of $\NN(M_C)$ under the map $C\rightarrow \cO_E$ induced by the projection $B\rightarrow E$, and let $\NN_{\mathcal{O}_E}$ be the saturation of $\NN_{\mathcal{O}_E}'$:
\[	\NN_{\mathcal{O}_E}:=\A_{F}\otimes_{\A_{F}^+}\NN_{\mathcal{O}_E}' \cap \NN_{\mathcal{O}_E}'[1/p]	\]
where the intersection is taken inside $\D(M_{\mathcal{O}_E})[1/p]$.

It turns out that $\NN_{\mathcal{O}_E}$ is a finite free $\mathcal{O}_E\otimes_{\Z_p}\A_{F}^+$-module of rank $d$.  Indeed, it is free over $\AA_{F}^+$ of rank $d\rm{rk}_{\Z_p}\cO_E$ by Lemma~\ref{lemma:wach-saturate}, so $\NN_{\mathcal{O}_E}/\pi$ is $p$-torsion-free, hence a lattice in the free $E\otimes_{\Qp}F$-module $(\NN_{\mathcal{O}_E}/\pi)[1/p]$, hence free over $\mathcal{O}_E\otimes_{\Z_p}\cO$ of rank $d$.

If we choose an $\mathcal{O}_E\otimes_{\Z_p}\A_{F}^+$-basis of $\NN_{\mathcal{O}_E}$, we can lift it to a $(C\otimes_{\Z_p}\A_{F}^+)[1/p]$-basis of $\NN(M_C)[1/p]$.  After enlarging $C$ to $C'\in\mathrm{Int}_B$, we may assume that the matrices of $\varphi$ and $\Gamma$ with respect to this basis have coefficients in $C'\otimes_{\Z_p}\A_{F}^+$.  Let $\NN_{C'}$ denote the $C'\otimes_{\Z_p}\A_{F}^+$-module spanned by this basis; $\NN_{C'}$ is free over $C'\otimes_{\Z_p}\A_{F}^+$, so in particular is flat over $\A_{F}^+$.  Then $\NN_{C'}$ is $\varphi$- and $\Gamma$-stable.

It is clear that $\Gamma$ acts trivially on $\NN_{C'}/\pi$, since this is true for $\NN(M_C)$ and therefore for $\NN(M_C)[1/p]$.  It remains to see that $q^h$ annihilates $\NN_{C'}/\varphi^\ast\NN_{C'}$.  Certainly this is true after inverting $p$, since $\NN_{C'}[1/p]=\NN(M_C)[1/p]$ is a Wach module.  So we need to see that $\NN_{C'}/\varphi^\ast\NN_{C'}$ has no $p$-torsion.

To see this, we observe that there is a surjection $C'\twoheadrightarrow \mathcal{O}_E$, with kernel $\mathfrak{n}$, and there is a descending filtration $\{\mathfrak{n}_i\}$ of $\mathfrak{n}$ by subideals such that $\mathfrak{n}_i/\mathfrak{n}_{i+1}\cong \mathcal{O}_E$.  Since the action of $\varphi$ on $\NN_{C'}$ is $C'$-linear, $\mathfrak{n}_{i}\NN_{C'}/\mathfrak{n}_{i+1}$ is isomorphic to $\NN_{\mathcal{O}_E}$ as a $\varphi$-module.  It follows that $\NN_{C'}/\varphi^\ast\NN_{C'}$ is an extension of copies of $\NN_{\mathcal{O}_E}/\varphi^\ast\NN_{\mathcal{O}_E}$, and is therefore $p$-torsion free.
\end{proof}

The upshot of all of this is that if $M$ is a rank-$d$ family of $A$-linear representations of $G_{F}$, $\mathscr{L}_M^{\leq h}$ carries a family $\mathscr{N}^{\leq h}(M)$ of $\mathcal{O}_{\mathscr{L}_M^{\leq h}}\widehat\otimes_{\Z_p}\A_{F}^+$-modules whose $\mathfrak{m}_A$-adic completion is the universal family of $\mathcal{O}_{\widehat{\mathscr{L}}_M^{\leq h}}\widehat\otimes_{\Z_p}\A_{F}^+$-lattices of height $\leq h$.  The image of $\Theta:\mathscr{L}_M^{\leq h}[1/p]\rightarrow\Spec A[1/p]$ is exactly the crystalline points with Hodge--Tate weights in the interval $[-h,0]$.

\begin{definition}
Let $M$ be a family of crystalline representations over $A$ with Hodge--Tate weights in $[-h,0]$. We define the \emph{total Wach module} to be $\NN(M)^{\rm{total}}:=\Theta_{A\ast}\mathscr{N}^{\leq h}(M)$.  If $M$ is a family of crystalline representations with Hodge--Tate weights in $[a,b]$, we define $\NN(M)^{\rm{total}}:=\pi^{-b}\NN(M(-b))^{\rm{total}}$.

Let $\mathcal{I}\subset \mathcal{O}_{\mathscr{L}_M^{\leq h}}$ denotes the ideal sheaf generated by all $p$-power torsion sections, and let $\mathscr{L}_M^{\leq h,\rm{t.f.}}$ denote the corresponding closed subscheme.  Let $\mathscr{N}^{\leq h,\rm{t.f.}}(M)$ denote the restriction of $\mathscr{N}^{\leq h}(M)$ to $\mathscr{L}_M^{\leq h,\rm{t.f.}}$, and write $\NN(M)^{\rm{t.f.}}:=\Theta_{A\ast}\mathscr{N}^{\leq h,\rm{t.f.}}(M)=\NN(M)^{\mathrm{total}}/\NN(M)^{p-\mathrm{tors}}$.
\end{definition}
We observe that by the theorem on formal functions, the $\mathfrak{m}_A$-adic completion of $\NN(M)^{\mathrm{total}}$ is $\widehat{\Theta_{A\ast}}\widehat{\mathscr{N}}^{\leq h}(M)$.

Since $\Theta_A$ is a projective map, $\Theta_{A\ast}\mathcal{O}_{\mathscr{L}_M^{\leq h}}$ is $A$-finite, and $\NN(M)^{\rm{total}}$ is a finite $A\widehat\otimes_{\Z_p}\A_{F}^+$-module.  Since $\Theta_A[1/p]:\mathscr{L}_M^{\leq h}[1/p]\rightarrow\Spec A[1/p]$ is an isomorphism, $\NN(M)^{\rm{total}}[1/p]$ is projective of rank $d$ and its formation commutes with base change on $A[1/p]$.  However, $\NN(M)^{\rm{total}}$ may not be projective itself.

We record a useful lemma about $\NN(M)^{\rm{total}}[1/p]$:
\begin{lemma}\label{lemma:wach-loc-free-mixed}
Let $\NN$ be a finite $A\htimes\A_F^+$-module such that $\NN[1/p]$ is projective.  Then there is an affine cover $\{U_k=\Spec A[1/p,1/f_k]\}$ of $\Spec A[1/p]$ such that $\NN\otimes_AA[1/p,1/f_k]$ is free.
\end{lemma}
\begin{proof}
It suffices to show that if $\mathfrak{m}\subset A[1/p]$ is a maximal ideal, then $\NN[1/p]\otimes_{A[1/p]}A[1/p]_{\mathfrak{m}}$ is free of rank $d$.  If $K:=A[1/p]/\mathfrak{m}$, then $K$ is a finite extension of $\Qp$ with ring of integers $\cO$ and $(A\htimes\A_F^+)[1/p]/\mathfrak{m}\cong (\cO_K\otimes_{\Z_p}\cO_F)[\![\pi]\!][1/p]$, which is a product of finitely many principal ideal domains.  Therefore, $\NN[1/p]\otimes_{A[1/p]}A[1/p]/\mathfrak{m}$ is free of rank $d$.  But $(A\htimes\A_F^+)[1/p]\otimes_{A[1/p]}A[1/p]_{\mathfrak{m}}$ has Jacobson radical generated by $\mathfrak{m}$, so by Nakayama's lemma, we can lift $d$ generators of $\NN[1/p]\otimes_{A[1/p]}A[1/p]/\mathfrak{m}$ to generators of $\NN[1/p]\otimes_{A[1/p]}A[1/p]_{\mathfrak{m}}$.  Since a surjection of projective modules of the same rank is an isomorphism, we are done.
\end{proof}

\begin{lemma}\label{lemma:wach-d+-inclusions}
There are natural inclusions of sheaves
\[	\pi^h\left((\mathcal{O}_{\widehat{\mathscr{L}}_M^{\leq h}}\widehat\otimes\A^+)\otimes_AM\right)\subset (\mathcal{O}_{\widehat{\mathscr{L}}_M^{\leq h}}\htimes\A^+)\otimes_{\mathcal{O}_{\widehat{\mathscr{L}}_M^{\leq h}}\htimes\A_{F}^+}\widehat{\mathscr{N}}^{\leq h}(M)\subset (\mathcal{O}_{\widehat{\mathscr{L}}_M^{\leq h}}\widehat\otimes\A^+)\otimes_AM	\]
and therefore natural inclusions of sheaves
\[	\pi^h\D^+(\mathcal{O}_{\widehat{\mathscr{L}}_M^{\leq h}}\otimes_AM)\subset \widehat{\mathscr{N}}^{\leq h}(M)\subset \D^+(\mathcal{O}_{\widehat{\mathscr{L}}_M^{\leq h}}\otimes_AM)	\]
\end{lemma}
\begin{proof}
This follows from Proposition~\ref{prop:wach-d+-inclusions-artin} by taking limits.
\end{proof}

\begin{remark}
Combined with Lemma~\ref{lemma:wach-rep}, this implies that the representation $M_{\widetilde A}$ is finite height, i.e., $\D(M_{\widetilde A})$ is generated by elements of $\D^+(M_{\widetilde A})$.
\end{remark}

\begin{proposition}\label{prop:wach-pi-h}
Let $A\rightarrow A'$ be a local homomorphism of complete local noetherian $\Z_p$-algebras with finite residue field.  Then there is a natural base change map $A'\htimes_A\NN(M)^{\mathrm{total}}\rightarrow \NN(M_{A'})^{\mathrm{total}}$, and its image contains the image of $\pi^h\D^+(M_{A'})$.
\end{proposition}
\begin{proof}
This follows from Lemma~\ref{lemma:artin-wach-pi-h} by taking limits.
\end{proof}

We wish to compare this construction with Berger's construction when $A=\Z_p$ and $M$ is a lattice in a crystalline representation with Hodge--Tate weights in $[-h,0]$.  To show they agree, we would have to prove that $\NN(M)$ is free over $\A_{F}^+$.  There are two problems: A priori, there could be connected components of $\mathscr{L}_M^{\leq h}$ supported on the special fiber and killed by inverting $p$, which would lead to $p$-power torsion in $\NN(M)^{\mathrm{total}}$.  In addition, even after killing $p$-torsion we would only know that Berger's Wach module is the $p$-saturation of ours in $\D(M)$.

We now consider only the situation when $A$ is flat over $\Z_p$, $A$ is integrally closed in $A[1/p]$, and $M$ is a crystalline family of $G_F$-representations with coefficients in $A$.  In particular, $A\rightarrow\widetilde A$ is injective, $A[1/p]\rightarrow\widetilde A[1/p]$ is an isomorphism, and $A\rightarrow \widetilde A^{\mathrm{t.f.}}$ is an isomorphism.


By construction, there is a natural inclusion $\widehat{\mathscr{N}}^{\leq h}(M)\rightarrow \D(\mathcal{O}_{\widehat{\mathscr{L}}^{\leq h}_M}\otimes_AM)$, and therefore a natural inclusion $\NN(M)^{\rm{t.f.}}[1/p]=\NN(M)^{\rm{total}}[1/p]\rightarrow\D(M)[1/p]$ (we use the fact that $A\rightarrow \widetilde A$ is a finite morphism).  We define the \emph{saturation} of $\NN(M)^{\rm{total}}$ by setting $\NN(M)^{\rm{sat}}:=\NN(M)^{\rm{total}}[1/p]\cap \D(M)$, where the intersection takes place inside $\D(M)[1/p]$.  Note that if $A$ is $\Z_p$-finite and flat, this agrees with Berger's definition of the Wach module.

\begin{lemma}
Let $A$ and $M$ be as above.  Then $\NN(M)^{\mathrm{sat}}[1/p]=\NN(M)^{\mathrm{total}}[1/p]$.
\end{lemma}
\begin{proof}
Clearly, $\NN(M)^{\mathrm{sat}}[1/p]\subset\NN(M)^{\mathrm{total}}[1/p]$.  For the other direction, choose $n\in \NN(M)^{\mathrm{total}}\subset \D(M_{\widetilde A})$.  Since $A\rightarrow \widetilde A$ is finite and $A[1/p]\rightarrow \widetilde A[1/p]$ is an isomorphism, there is some integer $k\geq 0$ such that $p^k(\widetilde A\htimes\A)\subset A\htimes\A$.  This implies that $p^k\D(M_{\widetilde A})\subset \D(M)$, so that $p^kn\in\D(M)\cap\NN(M)^{\mathrm{total}}\subset \NN(M)^{\mathrm{sat}}$.
\end{proof}

\begin{lemma}
The natural map $(A\htimes\A_F)\otimes_{A\htimes\A_F^+}\NN(M)^{\mathrm{sat}}\rightarrow\D(M)$ is surjective.
\end{lemma}
\begin{proof}
In the proofs of Lemma~\ref{lemma:wach-rep-artin} and Lemma~\ref{lemma:wach-rep} we showed that the natural map $(\widetilde A\htimes\A_F)\otimes_{\widetilde A\htimes\A_F^+}\NN(M)^{\mathrm{total}}\rightarrow\D(M_{\widetilde A})$ is surjective.  It follows that the natural map $(A\htimes\A_F)\otimes_{A\htimes\A_F^+}\NN(M)^{\mathrm{t.f.}}\rightarrow\D(M)$ is surjective, and since $\NN(M)^{\mathrm{t.f.}}\subset \NN(M)^{\mathrm{sat}}$, the result follows.
\end{proof}

\begin{lemma}
The action of $\Gamma$ on $\NN(M)^{\rm{sat}}/\pi$ is trivial, and the cokernel of $\varphi^\ast\NN(M)^{\rm{sat}}\rightarrow\NN(M)^{\rm{sat}}$ is annihilated by $q^h$.
\end{lemma}
\begin{proof}
The kernel of $\NN(M)^{\rm{sat}}\rightarrow \NN(M)^{\rm{total}}[1/p]/\pi$ is $\pi\NN(M)^{\rm{sat}}$ (since $\pi$ is invertible in $A\widehat\otimes\A_{F}$), so the natural map $\NN(M)^{\rm{sat}}/\pi\rightarrow \NN(M)^{\rm{total}}[1/p]/\pi$ is injective.  Since the $\Gamma$-action on the right is trivial, so is the $\Gamma$-action on the left.

For the second assertion, we observe that the linearization of $\varphi$ is an isomorphism on $\D(M)$ (and on $\D(M)[1/p]$), and the cokernel of $\varphi^\ast\NN(M)^{\rm{total}}[1/p]\rightarrow \NN(M)^{\rm{total}}[1/p]$ is annihilated by $q^h$.
\end{proof}

\begin{lemma}\label{lemma:wach-sat-d+-inclusions}
Suppose that $A$ is flat over $\Z_p$.  Then there are natural inclusions $\pi^h\D^+(M)\subset \NN(M)^{\mathrm{sat}}\subset\D^+(M)$, and natural inclusions $\pi^h\left((A\htimes\A^+)\otimes_AM\right)\subset \A^+\htimes_{\A_F^+}\NN(M)^{\mathrm{sat}}\subset(A\htimes\A^+)\otimes_AM$.
\end{lemma}
\begin{proof}
In order to show that there is and inclusion $\pi^h\D^+(M)\subset \NN(M)^{\mathrm{sat}}$ is clear, we first observe that $A\rightarrow \widetilde A$ is an injection (since it is an isomorphism after inverting $p$ and $A$ is assumed flat over $\Z_p$).  Moreover, $\D^+(M)$ is $p$-torsion-free, and so
\[	\pi^h\D^+(M)\subset\pi^h\D^+(M)[1/p]\cap \D(M)\subset\pi^h\D^+(\widetilde A\otimes_AM)[1/p]\cap \D(M)\subset \NN(M)^{\rm{sat}}.	\]
where we use $\pi^h \D^+(\widetilde A\otimes_AM)\subseteq \NN(M)^{\mathrm{total}}$, which follows from Lemma \ref{lemma:wach-d+-inclusions} by taking global sections.  More precisely, formal GAGA implies that the injective map of coherent $\mathcal{O}_{\widehat{\mathscr{L}}_M^{\leq h}}\htimes\A^+$-modules $\pi^h (\mathcal{O}_{\widehat{\mathscr{L}}_M^{\leq h}}\htimes\A^+)\otimes_AM\rightarrow (\mathcal{O}_{\widehat{\mathscr{L}}_M^{\leq h}}\htimes\A^+)\otimes_{\mathcal{O}_{\widehat{\mathscr{L}}_M^{\leq h}}\htimes\A_F^+}
\widehat{\mathscr{N}}^{\leq h}$ over $\widehat{\mathscr{L}}_M^{\leq h}\times_{\Spf A}\Spf(A\htimes\A^+)$ arises from an injective map of coherent sheaves over $\mathscr{L}_M^{\leq h}\times_{\Spec A}\Spec(A\htimes\A)$, since $\mathfrak{m}_A$-adic completion is fully faithful.  Taking $H_F$-invariants and pushing forward by $\Theta_A$ yields the desired statement.

For the inclusion $\NN(M)^{\mathrm{sat}}\subset\D^+(M)$, we first recall   from Lemma \ref{lemma:wach-d+-inclusions}  that
\[	\widehat{\mathscr{N}}^{\leq h}(M)\subset \D^+(\mathcal{O}_{\widehat{\mathscr{L}}_M^{\leq h}}\otimes_AM)\subset (\mathcal{O}_{\widehat{\mathscr{L}}_M^{\leq h}}\widehat\otimes\A^+)\otimes_AM	\]
so $\NN(M)^{\rm{total}}[1/p]\subset (M\otimes_A(\widetilde A\widehat\otimes\A^+))[1/p] = (M\otimes_A(A\widehat\otimes\A^+))[1/p]$, and $\NN(M)^{\rm{total}}[1/p]$ is fixed by $H_{F}$.  On the other hand,
\[	\D(M) = ((A\widehat\otimes\A)\otimes_AM)^{H_{F}}\subset (A\widehat\otimes\A)\otimes_AM	\]
Therefore, $\NN(M)^{\rm{sat}}\subset \left((A\widehat\otimes\A^+)[1/p]\cap (A\widehat\otimes\A)\right)\otimes_AM$, so it suffices to show that $A\widehat\otimes\A^+=(A\widehat\otimes\A^+)[1/p]\cap (A\widehat\otimes\A)$ inside $(A\widehat\otimes\A)[1/p]$, or equivalently, that $p^k(A\htimes\A^+)=(A\htimes\A^+)\cap p^k(A\htimes\A)$ inside $A\htimes\A$ for every integer $k\geq 0$.  But $A\htimes\A^+=\varprojlim_n (A/\mathfrak{m}_A^n)\otimes\A^+$ and $A\htimes\A=\varprojlim_n(A/\mathfrak{m}_A^n)\otimes\A$, and intersections are limits, as well.  Therefore, Lemma~\ref{lemma:a-a+-mult} implies the desired result.

Now we turn to the second pair of inclusions.  By Lemma~\ref{lemma:wach-d+-inclusions}
\[	\NN(M)^{\mathrm{sat}}[1/p]\subset((A\htimes\A^+)\otimes_{A\htimes\A_F^+}\NN(M)^{\mathrm{sat}})[1/p]\subset ((A\htimes\A^+)\otimes_AM)[1/p]	\]
In addition,
\[	\NN(M)^{\mathrm{sat}}\subset\D(M)\subset (A\htimes\A)\otimes_{A\htimes\A_F}\D(M)\cong (A\htimes\A)\otimes_AM	\]
Since $M$ is free over $A$ and we have just shown that $(A\htimes\A^+)=(A\htimes\A^+)[1/p]\cap (A\htimes\A)$, we see that $\NN(M)^{\mathrm{sat}}\subset (A\htimes\A^+)\otimes_AM$, and there is a natural map
\[	(A\htimes\A^+)\otimes_{A\htimes\A_F^+}\NN(M)^{\mathrm{sat}}\rightarrow (A\htimes\A^+)\otimes_AM	\]
This map is injective after inverting $p$, and $(A\htimes\A^+)\otimes_{A\htimes\A_F^+}\NN(M)^{\mathrm{sat}}$ has no $p$-torsion, so it is an inclusion.

It remains to show that $\pi^h(A\htimes\A^+)\otimes_AM\subset (A\htimes\A^+)\otimes_{A\htimes\A_F^+}\NN(M)^{\mathrm{sat}}$. By Lemma~\ref{lemma:wach-d+-inclusions} we know that
\[	\pi^h(A\htimes\A^+)\otimes_AM\subset (A\htimes\A^+)\otimes_{A\htimes\A_F^+}\NN(M)^{\mathrm{total}}	\]
Since $\pi^h(A\htimes\A^+)\otimes_AM$ has no $p$-torsion, it is contained in $(A\htimes\A^+)\otimes_{A\htimes\A_F^+}\NN(M)^{\mathrm{t.f.}}\subset (A\htimes\A^+)\otimes_{A\htimes\A_F^+}\NN(M)^{\mathrm{sat}}$ and we are done.
\end{proof}

As in~\cite{berger04}, we write
\[	\alpha(h):=\inf_{\gamma\in\Gamma_F}\sum_{j=1}^hv_p(\chi(\gamma)^j-1)	\]

\begin{corollary}\label{cor:wach-wach-sat}
$\NN(M)^{\rm{sat}}$ is a finite $A\widehat\otimes\A_{F}^+$-module, and $p^{\alpha(h)}\NN(M)^{\mathrm{sat}}\subset\NN(M)^{\rm{total}}$.
\end{corollary}
\begin{proof}
By Lemma~\ref{lemma:wach-d+-inclusions}, $\D^+(M\otimes_A\widetilde A)$ is finite over $A\widehat\otimes\A_{F}^+$ (since $A\widehat\otimes\A_{F}^+$ is noetherian).  Since $A\rightarrow \widetilde A$ is injective, $\D^+(M)\subset\D^+(M\otimes_A\widetilde A)$ and so $\NN(M)^{\rm{sat}}$ is finite over $A\widehat\otimes\A_{F}^+$.

The proof of the second assertion is very similar to the proof of~\cite[Proposition IV.1.3]{berger04}.  We note that we have a set of inclusions
\[
\xymatrix{
&&	\NN(M)^{\mathrm{total}}\ar@{^{(}->}[r]	&	\D^+(\widetilde A\otimes_AM)	\\
\pi^h\D^+(M)\ar@{^{(}->}[urr]\ar@{^{(}->}[drr]	&&&	\\
&&	\NN(M)^{\mathrm{sat}}\ar@{^{(}->}[r]	&	\D^+(M)\ar@{^{(}->}[uu]
}
\]
with the action of $\Gamma$ stabilizing $\NN(M)^{\rm{total}}$ and $\NN(M)^{\rm{sat}}$ and acting trivially on $\NN(M)^{\rm{total}}/\pi$ and $\NN(M)^{\rm{sat}}/\pi$.  In particular, $\pi^h\NN(M)^{\rm{sat}}\subset\NN(M)^{\rm{total}}$ (inside $\D^+(\widetilde A\otimes_AM)$).  We may then prove by induction on $i$ that for all $\gamma\in\Gamma_F$
\[	\pi^{h-i}\prod_{j=0}^{i-1}(\chi(\gamma)^{h-j}-1)\NN(M)^{\rm{sat}}\subset\NN(M)^{\rm{total}}	\]
Taking $i=h$ and observing that $p^{\alpha(h)}$ generates the same ideal in $\Z_p$ as $\prod_{j=0}^{h-1}(\chi(\gamma)^{h-j}-1)$, we see that $p^{\alpha(h)}\NN(M)^{\rm{sat}}\subset\NN(M)^{\rm{total}}$, as desired.
\end{proof}

We can also study the behavior of $\NN(M)^{\mathrm{sat}}$ under change of coefficients.  As the formation of $\mathscr{L}_M^{\leq h}$ is clearly functorial in the coefficients, for any homomorphism of complete local noetherian $\cO$-algebras $A\rightarrow A'$, we have a natural base change morphism $A'\htimes_A\NN(M)^{\rm{sat}}\rightarrow \NN(M_{A'})^{\rm{sat}}$.  However, it need not be an isomorphism.

\begin{lemma}\label{lemma:wach-sat-pi-h}
Let $A\rightarrow A'$ be a homomorphism of complete local noetherian flat $\Z_p$-algebras with finite residue fields which are integrally closed in $A[1/p]$ and $A'[1/p]$, respectively.  Then the image of the natural map $(A'\widehat\otimes\A_{F}^+)\otimes_{A\widehat\otimes\A_F^+}\NN^{\rm{sat}}(M_A)\rightarrow \NN^{\rm{sat}}(M_{A'})$ contains $\pi^h\D^+(M_{A'})$.
\end{lemma}
\begin{proof}
By Proposition~\ref{prop:wach-pi-h}, the image of $(A'\htimes\A_F^+)\otimes_{A\htimes\A_F^+}\NN(M)^{\mathrm{total}}$ contains the image of $\pi^h\D^+(M_{A'})$, and since $\pi^h\D^+(M_{A'})$ has no $p$-torsion, the image of $(A'\htimes\A_F^+)\otimes_{A\htimes\A_F^+}\NN(M)^{\mathrm{t.f.}}$ contains the image of $\pi^h\D^+(M_{A'})$.  Since $\NN(M)^{\mathrm{t.f.}}\subset \NN(M)^{\mathrm{sat}}$, we are done.
\end{proof}

\begin{corollary}
Let $A\rightarrow A'$ be a local homomorphism of rings satisfying the above hypotheses.  Then there is a natural base change morphism $(A'\widehat\otimes\A_{F}^+)\otimes_{A\widehat\otimes\A_{F}^+}\NN(M)^{\rm{sat}}\rightarrow \NN(M_{A'})^{\rm{sat}}$, and its cokernel is annihilated by $p^{\alpha(h)}$.
\end{corollary}
\begin{proof}
This follows from Lemma~\ref{lemma:wach-sat-pi-h} by the same argument as in the proof of Lemma~\ref{cor:wach-wach-sat}.
\end{proof}

Suppose in addition that $A=\varprojlim_i A/J^i$, for some ideal $J\subsetneq A$ such that $A/J^i$ is a finite flat $\Z_p$-module.  There is a natural map $\NN(M)^{\rm{sat}}\rightarrow \varprojlim_i \NN(M_{A/J^i})^{\rm{sat}}$.
\begin{proposition}
The natural map $\NN(M)^{\rm{sat}}\rightarrow \varprojlim_i \NN(M_{A/J^i})^{\rm{sat}}$ is an isomorphism.
\end{proposition}
\begin{proof}
Consider an element $(n_i)\in\varprojlim_i\NN(M_{A/J^i})^{\rm{sat}}$.  Then Corollary~\ref{cor:wach-wach-sat} implies that $(p^{\alpha(h)}n_i)\in \varprojlim_i \NN(M_{A/J^i})=\NN(M)$, so $(n_i)$ comes from an element $n\in\NN(M)[1/p]$.  But $(n_i)$ also comes from an element of $\D(M)$, because $\D(M)=\varprojlim_i\D(M_{A/J^i})$, so $n\in \NN(M)^{\rm{sat}}$.
\end{proof}

Let $\NN(M)^{\rm{p-tors}}$ denote the $p$-power torsion in $\NN(M)^{\rm{total}}$.
\begin{lemma}
If $A=\varprojlim_iA/J^i$, for some ideal $J\subsetneq A$ such that $A/J^i$ is $\Z_p$-flat, $\NN(M)^{\rm{p-tors}}=\varprojlim_i\NN(M_{A/J^i})^{\rm{p-tors}}$.
\end{lemma}
\begin{proof}
Clearly if $n\in\NN(M)^{\rm{total}}$ is $p$-power torsion, so is its image in $\NN(M_{A/J^i})^{\rm{total}}$ for all $i$.

On the other hand, the natural map $\NN(M)^{\rm{total}}[1/p]\otimes_{A[1/p]}(A/J^i)[1/p]\rightarrow \NN(M_{A/J^i})^{\rm{total}}[1/p]$ is an isomorphism for all $i$, so to prove the converse, we show that $\NN(M)^{\rm{total}}\cap\bigcap_i (J^i[1/p])\NN(M)$ is $p$-power torsion.  Since $A/J$ is assumed $\Z_p$-flat, $J[1/p]\subset A[1/p]$ is a proper ideal.  Then by a theorem of Krull~\cite[Theorem 8.9]{matsumura} there is some $a\in A[1/p]$ which is $1$ modulo $J[1/p]$ and which annihilates $\cap_i (J^i[1/p])\NN(M)[1/p]$.  Again because $A/J$ is assumed $\Z_p$-flat, $J[1/p]\cap A=J$, so there is some power of $p$ such that $a':=p^na$ is an element of $A$ which is congruent to $p^n$ modulo $J$ and annihilates $\ker(\NN(M)^{\rm{total}}\rightarrow \NN(M_{A/J^i})^{\rm{total}}[1/p])$ for all $i$.  Then $a'$ is congruent to $p^n$ modulo the maximal ideal of $A$, so multiplying $a'$ by a $1$-unit of $A$, we see that $p^n$ annihilates $\ker(\NN(M)^{\rm{total}}\rightarrow \NN(M_{A/J^i})^{\rm{total}}[1/p])$ for all $i$ and we are done.
\end{proof}

From now on, we assume that $A$ is $\Z_p$-flat and integrally closed in $A[1/p]$.  We refer to $\NN(M)^{\mathrm{sat}}$ as the Wach module of $M$ and we drop the superscript, writing $\NN(M)$ instead.  For the convenience of the reader, we collect some useful properties of $\NN(M)$:
\begin{theorem}\label{thm:wach}
Let $M$ be a crystalline family of $G_F$-representations over $A$ (of rank $d$) with Hodge--Tate weights in $[-h,0]$. Then there exists a canonical finitely generated, $\varphi$- and $\Gamma$-stable  $A\htimes\A_{F}^+$-submodule \[\NN(M)\subseteq \D(M)\] satisfying the following properties:
\begin{enumerate}
\item $(A\htimes\A_{F})\otimes_{A\htimes\A_{F}^+} \NN(M)\twoheadrightarrow \D(M)$, i.e., $\NN(M)$ generates $\D(M)$ as an $A\htimes\A_{F}$-module,
\item $\NN(M)[1/p]$ is a finitely generated projective $(A\htimes\A_{F}^+)[1/p]$-module of rank $d$,
\item $\NN(M)$ is $p$-saturated in $\D(M),$ i.e., $\NN(M)[1/p]\cap \D(M)=\NN(M),$
\item $\Gamma$ acts trivially on $\NN(M)/\pi \NN(M)$,
\item $q^h$ annihilates the cokernel of $\varphi^*\NN(M)\to \NN(M)$,
\item $\pi^h\D^+(M)\subseteq \NN(M)\subseteq \D^+(M).$
\item	For any local homomorphism $A\to A'$ of complete local noetherian flat $\Z_p$-algebras with finite residue fields, integrally closed in $A[1/p]$ and $A'[1/p]$, respectively, there is a natural homomorphism $(A'\htimes\A_F^+)\otimes_{A\htimes\A_F^+}\NN(M)\to \NN(M_{A'})$, and its cokernel is annihilated by $p^{\alpha(h)}$.
\end{enumerate}
Moreover, these properties uniquely characterize $\NN(M)$.
\end{theorem}
\begin{proof}
We have shown that $\NN(M)$ exists, so it only remains to see that the list of properties uniquely characterizes $\NN(M)$.

Property $(3)$ implies that it suffices to prove that $\NN(M)[1/p]$ is uniquely characterized by the above list of properties.  Moreover, the functoriality required by property $(7)$ implies that $\NN(M)[1/p]$ is characterized by $\NN(M_{A'})[1/p]$ when $A'$ is a finite flat $\Zp$-algebra.  But then $M_{A'}[1/p]$ is a finite-dimensional $\Qp$-linear representation of $G_F$, so~\cite[Corollary III.4.2]{berger04} implies that $\NN(M_{A'})[1/p]$ is uniquely determined by its $\varphi$- and $\Gamma$-stability and properties $(2)$, $(4)$, and $(5)$.
\end{proof}

\section{Generic fibre}\label{sec:fibre}

%

\subsection{Families of Wach modules and $(\varphi,\Gamma)$-modules}

Given a family of positive crystalline representations $M$ of $G_F$ over $R$ of rank $d$, we get a family $M^{\rig}$ of Galois representations over the Fr\'echet-Stein algebra $R^{\rig}$ given by taking the rigid analytic generic fiber of $R$.  

If $V$ is a $d$-dimensional $\Q_p$-vector space equipped with a continuous action of $G_K$, then the main result of~\cite{cherbonnier-colmez} implies that for $s\gg 0$, $\D^{\dagger,s}(V):=(\B^{\dagger,s}\otimes_{\Qp}V)^{H_K}\subset\widetilde\B_{\rig}^{\dagger,s}\otimes_{\Qp}V$ is a free $\B_K^{\dagger,s}$-module of rank $d$.  Berger and Colmez~\cite{berger-colmez} extended this to the case where $V$ is a $G_K$-representation with coefficients in a $\Qp$-affinoid algebra.

Since $(\Spf R)^{\rig}$ is a rigid analytic space, we obtain a family of $(\varphi,\Gamma)$-modules $\Drig^{\dagger}(M^{\rig})$ over $ R^{\rig}$.  More precisely, we exhaust the quasi-Stein space $\Spf(R)^{\rig}$ by a rising sequence of affinoids $\{U_i=\Sp(A_i)\}$ and we note that for each $i$ there is some $s_i>0$ such that over $U_i$, $\Drig^{\dagger}(M^{\rig}|_{U_i})$ is generated by $\D_{F}^{\dagger,s_i}(M^{\rig}|_{U_i})$.   Letting $C^{(0,s_i]}$ denote the half-open annulus $\{0<v_p(\pi)\leq s_i\}$ over $F$, we obtain a coherent sheaf on the quasi-Stein space $U_i\times C^{(0,s]}$.  Taking global sections and extending scalars from $\mathcal{O}_{\Spf R^{\rig}}(U_i)\widehat\otimes\B_{\rig,F}^{\dagger,s}$ to $\mathcal{O}_{\Spf R^{\rig}}(U_i)\widehat\otimes\B_{\rig,F}^{\dagger}$ yields a sheaf $\Drig^{\dagger}(M^{\rig})$ on $\Spf(R)^{\rig}$.

Similarly, to the Wach module $\NN(M)$ of $M$, we associate a finite projective module
\[\NN(M)^{\rig}:=(R^{\rig}\htimes\B_{\rig,F}^+)\otimes_{(R \htimes \AA^+_{F})}\NN(M)\]
over $R^\rig\htimes\B_{\rig,F}^+$.  In fact, since $\NN(M)$ is a finite $R \htimes \AA^+_{F}$-module the functor ${(-)}^{\rig}$ associates to it a coherent sheaf on the rigid space $X=\mathrm{Spf}( R \htimes \AA^+_{F})^\rig$ with $\mathcal{O}_X(X)=R^\rig\htimes\B_{\rig,F}^+$; here we identify this sheaf $ \NN(M)^{\rig}$ with its global sections, see Lemma \ref{lemma:rig}.  Since $\NN(M)[1/p]$ is a projective $(R\htimes\AA_F^+)[1/p]$-module, the coherent sheaf is a vector bundle on $X$.

\subsection{Comparisons between $\NN(M)$ and $\D_{\cris}(M^{\rig})$}

The paper \cite{berger02} explains how to relate $(\varphi,\Gamma)$-modules to the functors $\Dcris(-)$ and $\D_{\st}(-)$, and \cite{berger04} refines this for positive crystalline representations of $G_F$.  Indeed, if $V$ is a finite-dimensional $\Qp$-linear representation of $G_K$, we have
\[	\Dcris(V)=\Drig^\dagger(V)[1/t]^\Gamma	\]
by \cite[Th\'eor\`eme 3.6]{berger02}.

If $T\subset V$ is a lattice in a finite-dimensional $\Qp$-linear positive crystalline representation of $G_F$, then the Wach module $\NN(V)=\NN(T)[1/p]$ is contained in the $(\varphi,\Gamma)$-module $\Drig^\dagger(V)$, and it generates it over $\B_{\rig,F}^\dagger$.  Further:
\begin{proposition}[{\cite[Proposition II.2.1]{berger04}, \cite[Proposition III.2.1]{berger04}}]
Let $0\leq r_1\leq\ldots\leq r_d$ be the opposites of the Hodge--Tate weights of $V$.  Then
\begin{enumerate}
\item	$\Dcris(V)\subset \B_{\rig,F}^+\otimes_{\B_F^+}\NN(V)$, and
\item	the natural map $\B_{\rig,F}^+\otimes_F\Dcris(V)\rightarrow \B_{\rig,F}^+\otimes_{\B^+}\NN(V)$ is an injection and its cokernel is isomorphic to $\oplus_{i=1}^d \B_{\rig,F}^+/(\lambda_i)$, where $\lambda_i=(t/\pi)^{r_i}$.
\end{enumerate}
\end{proposition}

This implies that $\Dcris(V)=\left(\B_{\rig,F}^+\otimes_{\B_F^+}\NN(V)\right)^\Gamma$.  But we have another relation:
\begin{theorem}[{\cite[Th\'eor\`eme]{berger04}}]
If $V$ is a positive crystalline representation of $G_F$ and we equip the Wach module $\NN(V)$ with the filtration
\[	\Fil^i\NN(V):=\{x\in\NN(V) \mid \varphi(x)\in q^i\NN(V)\}	\]
then the natural map $\Dcris(V)\rightarrow \NN(V)/\pi\NN(V)$ is an isomorphism of filtered $\varphi$-modules.
\end{theorem}
In particular, this implies that $\Dcris(T):=\NN(T)/\pi\NN(T)$ defines a $\cO_F$-lattice in the $F$-vector space $\Dcris(V)$.

We wish to extend these results to formal families of positive crystalline representations.  To precisely define $\D_{\cris}(M^{\rig})$ we must again exhaust $\Spf(R)^{\rig}$ by affinoid subdomains $\{U_i\}$, consider the finite projective modules $\D_{\cris}(M^{\rig}|_{U_i})$ over each $U_i$, and take the limit to obtain a finite projective filtered $\varphi$-module over $ R^{\rig}$.  We also obtain an explicit description $\D_{\cris}(M^{\rig})=\left(( R^{\rig}\widehat\otimes\B_{\max})\otimes_{ R^{\rig}}M^{\rig}\right)^{G_{F}}$.  Then~\cite[Theorem 4.2.9]{bellovin13} shows that $\Dcris(M^{\rig})=\Drig^\dagger(M^\rig)^\Gamma$.

\begin{lemma}
If $M$ is positive, there is a natural homomorphism $\NN(M)^{\rig}\rightarrow \Drig^{\dagger}(M^{\rig})$.
\end{lemma}
\begin{proof}
Recall that the Wach module $\NN(M)$ induces a finite projective $R^\rig\htimes\B_{\rig,F}^+$-module $\NN(M)^{\rig}$ by extension of scalars (we implicitly use the Mittag--Leffler condition for modules over Fr\'echet-Stein algebras to commute the tensor products and inverse limit).  Equivalently, it induces a vector bundle on $\Spf(R\htimes\A_{F}^+)^{\rig}$.

Let $\cup_n \Sp(A_m)$ be a rising admissible covering of $\Spf(R)^{\rig}$, so that $ R^{\rig}=\cap_m A_m$, and let $A_m^\circ$ denote the ring of power-bounded elements of $A_m$.  Then for any $m$ and any $s$, the extension of scalars $(A_m\htimes\B_{\rig,F}^{\dagger,s})\otimes_{R^\rig\htimes\B_{\rig,F}^+}\NN(M)^{\rig}$ remains finite projective.  Moreover, since $\NN(M)\subset(R\htimes\A^+)\otimes_RM$, we have a natural map
\[	(A_m\htimes\B_{\rig,F}^{\dagger,s})\otimes_{R^\rig\htimes\B_{\rig,F}^+}\NN(M)^{\rig}\rightarrow (A_m\htimes\widetilde{\B}_{\rig}^{\dagger,s})\otimes_{A_m}(A_m\otimes_RM)	\]

On the other hand, the family of Galois representations $M_{A_m}:=A_m\otimes_RM$ over $A_m$ induces a family of $(\varphi,\Gamma)$-modules $\Drig^{\dagger,s}(M_{A_m})$ over $A_m\htimes\B_{\rig,F}^{\dagger,s}$ for $s\gg0$.  This family of $(\varphi,\Gamma)$-modules is finite projective, and $\Drig^{\dagger,s}(M_{A_m})$ is naturally a sub-$A_m\htimes\B_{\rig,F}^{\dagger,s}$-module of $(A_m\htimes\widetilde{\B}_{\rig}^{\dagger,s})\otimes_{A_m}(A_m\otimes_RM)$.  Equivalently, $\Drig^{\dagger,s}(M_{A_m})$ is a vector bundle on the product $X_s$ of a half-open annulus with $\Sp(A_m)$.

We may therefore consider the sub-$A_m\htimes\B_{\rig,F}^{\dagger,s}$-module \linebreak $\Drig^{\dagger,s}(M_{A_m})+\left((A_m\htimes\B_{\rig,F}^{\dagger,s})\otimes_{R^\rig\htimes\B_{\rig,F}^+}\NN(M)^{\rig}\right)$ of $(A_m\htimes\widetilde{\B}_{\rig}^{\dagger,s})\otimes_{A_m}(A_m\otimes_RM)$ generated by the images of $\NN(M)^{\rig}$ and $\Drig^{\dagger,s}(A_m\otimes_RM)$.  For ease of notation, we write it as $\Drig^{\dagger,s}(M_{A_m})+\NN(M)_{A_m}^{\rig}$.  Then $\Drig^{\dagger,s}(M_{A_m})+\NN(M)_{A_m}^{\rig}$ remains finite over $A_m\htimes\B_{\rig,F}^{\dagger,s}$, and we have an exact sequence of coherent sheaves on the quasi-Stein space $X_s$
\[	\Drig^{\dagger,s}(M_{A_m})\rightarrow \Drig^{\dagger,s}(M_{A_m})+\NN(M)_{A_m}^{\rig}\rightarrow Q\rightarrow 0	\]
for some quotient sheaf $Q$.  But the formations of $\Drig^{\dagger,s}(M_{A_m})$ and $\NN(M)^{\rig}$ commute with specialization on $A_m$, and in the classical case the Wach module is contained in the $(\varphi,\Gamma)$-module.  Thus, $Q$ vanishes at every point of $X_s$, so is trivial, and $\Drig^{\dagger,s}(M_{A_m})\rightarrow \Drig^{\dagger,s}(M_{A_m})+\NN(M)_{A_m}^{\rig}$ is surjective.
\end{proof}

\begin{lemma}
If $M$ is positive crystalline, then the image of $\NN(M)^{\rig}$ inside $\Drig^{\dagger}(M^{\rig})$ generates it over $R^\rig\htimes\B_{\rig,F}^{\dagger}$.
\end{lemma}
\begin{proof}
Let $U=\Sp(A)$ be an affinoid subdomain of $\Spf(R)^{\rig}$.  Then there is some $s\gg0$ such that $\Drig^\dagger(M^{\rig}|_U)$ is generated by $\Drig^{\dagger,s}(M^{\rig}|_U)$.
After extending scalars on $\NN(M)^{\rig}$ from $R^\rig\htimes\B_{\rig,F}^+$ to $A\htimes\B_{\rig,F}^{\dagger,s}$, we have a homomorphism of coherent sheaves of modules on the quasi-Stein space attached to $A\htimes\B_{\rig,F}^{\dagger,s}$.  Its cokernel is likewise a coherent sheaf.  But by specializing at points of $U$ and appealing to the classical case, we conclude that the cokernel must vanish.
\end{proof}

\begin{corollary}
The natural map $\NN(M)^{\rig}\rightarrow \D_{\rig}^{\dagger,s}(M^{\rig})$ is injective.
\end{corollary}
\begin{proof}
Let $U=\Sp(A)\subset (\Spf R)^{\rig}$ be an affinoid subdomain.  Then there is some $s\gg 0$ such that the $(\varphi,\Gamma)$-module $\D^{\dagger,s}(M^{\rig}|_U)$ is defined, and we have an exact sequence
\[	0\rightarrow K\rightarrow (A\htimes\B_{\rig,F}^{\dagger,s})\otimes_{R^\rig\htimes\B_{\rig,F}^+}\NN(M|_U)^{\rig}\rightarrow \Drig^{\dagger,s}(M^{\rig}|_U)\rightarrow 0	\]
of coherent sheaves on the quasi-Stein space attached to $A\htimes\B_{\rig,F}^{\dagger,s}$.  Since $\Drig^{\dagger,s}(M^{\rig}|_U)$ is projective, this exact sequence remains exact after specializing at points of $\Sp(A)$.  But in the classical case the Wach module is contained in the $(\varphi,\Gamma)$-module, so $K$ vanishes at every point of $\Sp(A)$.  It follows that $K$ vanishes.
\end{proof}

Note that $\Drig^{\dagger,s}(M^{\rig}|_U)$ is projective as an $A\htimes\B_{\rig,\Qp}^{\dagger,s}$-module; we do not know whether it is $A$-locally free.  However, it follows from the construction of families of $(\varphi,\Gamma)$-modules that there is some finite extension $L/F$ (depending on $U$) such that $\Drig^{\dagger,s}(M^{\rig}|_{U,G_L})$ is a free $A\htimes\B_{\rig,L}^{\dagger,s}$-module of rank $n$.

\begin{proposition}
If $M$ is positive crystalline, then $\D_{\cris}(M^{\rig})\subset \NN(M)^{\rig}$ inside $\Drig^\dagger(M^{\rig})$.
\end{proposition}
\begin{proof}
Since $\Dcris(M^{\rig})$ is a finite projective module over $R^{\rig}\otimes_{\Qp}F$ and $\NN(M)[1/p]$ can be trivialized locally on $\Spec R[1/p]$ by Lemma~\ref{lemma:wach-loc-free-mixed}, we can find an affinoid covering $\{U_m=\Sp(A_m)\}_{m\in I}$ of $\Spf(R)^{\rig}$ such that $\Dcris(M^{\rig}|_{U_m})$ and $\NN(M)^{\rig}|_{U_m}$ are free for each $i$.  We may also assume that $\D_{\cris}(M^{\rig})$ is free and $\Drig^\dagger(M^{\rig})$ is generated by $\Drig^{\dagger,s}(M^{\rig})$ for some $s\gg0$.

Now we choose a basis $\{\vec{e}_1,\ldots,\vec{e}_r\}$ of $\NN(M)^{\rig}|_{U_m}$ and a basis $\{\vec{f}_1,\ldots,\vec{f}_r\}$ of $\D_{\cris}(M^{\rig}|_{U_m})$.  Let $G=(g_{ij})\in \Mat_{r\times r}(\B_{\rig,F}^{\dagger,s}\widehat\otimes  A_m)$ be a matrix carrying $\{\vec{e}_1,\ldots,\vec{e}_r\}$ to $\{\vec{f}_1,\ldots,\vec{f}_r\}$.  We wish to show that $G$ actually has entries in $\B_{\rig,F}^+\widehat\otimes A_m$.

We may write $g_{ij}(\pi)$ in the form $g_{ij}(\pi)=\sum_{n\in\Z}a_n\pi^n$, where $a_n\in A_m$ for all $n$, and $g_{ij}$ converges for $p^{-1/s}\leq |\pi|<1$.  But by the proof of~\cite[II.2.1]{berger04}, if we specialize $G$ at any point $x\in\Sp(A_m)$, then $g_{ij}(\pi)(x)\in\B_{\rig,F}^+$.  In particular, if $n<0$, then $a_n(x)=0$ for all $x\in\Sp(A_m)$.  It follows that $g_{ij}$ actually converges for $|\pi|<1$.
\end{proof}

Now that we have an inclusion $\D_{\cris}(M^{\rig})\subset \NN(M)^{\rig}$, we wish to understand the cokernel of the induced map $(R^\rig\htimes\B_{\rig,F}^+)\otimes_{R^{\rig}}\D_{\cris}(M^{\rig})\rightarrow \NN(M)^{\rig}$.  When $R=\Z_p$, Berger showed~\cite[\textsection III]{berger04} that the quotient is isomorphic to $\oplus_{i=1}^n\B_{\rig,F}^+/\lambda_i$, where $\lambda_i=(t/\pi)^{r_i}$.  However, as such a decomposition is non-canonical, there is no reason to expect it to vary well in families.

However, we can prove the following weaker results:
\begin{proposition}\label{quot-nm-dcris}
Let $M$ be a family of positive crystalline representations of $G_F$ over $R$.  Then:
\begin{enumerate}
\item	Let $Q$ denote the cokernel of the map $\mathrm{inc}_M:(R^\rig\htimes\B_{\rig,F}^+)\otimes_{ R^{\rig}}\D_{\cris}(M^{\rig})\rightarrow \NN(M)^{\rig}$.  Then $Q$ is annihilated by $\lambda_d$.  In particular, $Q$ is $t$-torsion.
\item	The map $\mathrm{inc}_M$ is injective.
\item	The map $\mathrm{inc}_M$ induces an isomorphism of line bundles
\[	\left(\frac{\pi}{t}\right)^{r_1+\cdots+r_d}\cdot\det(\mathrm{inc}_M):(R^\rig\htimes\B_{\rig,F}^+)\otimes_{R^{\rig}}\Det_{R^{\rig}}\D_{\cris}(M^{\rig})\xrightarrow{\sim} \Det_{R^{\rig}\htimes\B_{\rig,F}^+}\NN(M)^{\rig}	\]
\end{enumerate}
\end{proposition}
\begin{proof}
For the first part, we consider the corresponding morphism of coherent sheaves on the rigid analytic space corresponding to $R^\rig\htimes\B_{\rig,F}^+$.  As the formations of $\NN(M)^{\rig}$ and $\Dcris(M^{\rig})$ are compatible with specialization on $ R^{\rig}$, Berger's result implies that $\lambda_dQ=0$.

We turn to the second part.  Since $(R^\rig\htimes\B_{\rig,F}^+)\otimes_{R^{\rig}}\D_{\cris}(M^{\rig})$ and $\NN(M)^{\rig}$ are finite projective modules of the same rank, the map is an isomorphism after inverting $t$.  Thus, the kernel is also $t$-torsion.  But $(R^\rig\htimes\B_{\rig,F}^+)\otimes_{R^{\rig}}\D_{\cris}(M^{\rig})$ is a projective module over $R^\rig\htimes\B_{\rig,F}^+$ so has no $t$-torsion.

For the last part, it suffices to consider the case where $M$ is rank $1$.  But then compatibility of $\Dcris(M^{\rig})$ and $\NN(M)^{\rig}$ with specialization on $ R^{\rig}$, combined with Berger's result, yields the desired result.
\end{proof}

\begin{corollary}\label{cor:det-frob-wach}
Let $M$ be a family of positive crystalline representations of $G_F$ over $R$.  Then
\[	q^{-(r_1+\cdots+r_d)}\cdot\det\left(\varphi^\ast\NN(M)[1/p]\rightarrow\NN(M)[1/p]\right)	\]
is an isomorphism of line bundles over $(R\htimes \A_F^+)[1/p]$.
\end{corollary}
\begin{proof}
Consider the commutative diagram
\[
\xymatrixcolsep{5pc}\xymatrix{
(R^{\rig}\htimes\B_{\rig,F}^+)\otimes_{R^{\rig}}\Dcris(M^{\rig}) \ar[r]^-{\mathrm{inc}_M}& \NN(M)^{\rig}	\\
\varphi^\ast\left((R^{\rig}\htimes\B_{\rig,F}^+)\otimes_{R^{\rig}}\Dcris(M^{\rig})\right) \ar[u]\ar[r]^-{\varphi^\ast\mathrm{inc}_M}& \varphi^\ast\NN(M)^{\rig} \ar[u]
}
\]
It implies that
\begin{multline*}
\det\left(\varphi^\ast\mathrm{inc}_M\right)\cdot\det\left(\varphi^\ast\NN(M)[1/p]\rightarrow\NN(M)[1/p]\right)=	\\\det\left(\varphi^\ast\left((R^{\rig}\htimes\B_{\rig,F}^+)\otimes_{R^{\rig}}\Dcris(M^{\rig})\right)\rightarrow (R^{\rig}\htimes\B_{\rig,F}^+)\otimes_{R^{\rig}}\Dcris(M^{\rig})\right)\cdot\det\left(\mathrm{inc}_M\right)
\end{multline*}

It is clear that $\varphi^\ast\left((R^{\rig}\htimes\B_{\rig,F}^+)\otimes_{R^{\rig}}\Dcris(M^{\rig})\right)\rightarrow (R^{\rig}\htimes\B_{\rig,F}^+)\otimes_{R^{\rig}}\Dcris(M^{\rig})$ is an isomorphism, because $\varphi$ acts bijectively on $\Dcris(M^{\rig})$.  Furthermore, Proposition~\ref{quot-nm-dcris} implies that
\[	(\pi/t)^{r_1+\cdots+r_d}\cdot\det\left(\mathrm{inc}_M\right)	\]
and
\[	(\varphi(\pi)/t)^{r_1+\cdots+r_d}\cdot\det\left(\varphi^\ast\mathrm{inc}_M\right)	\]
are isomorphisms of line bundles over $R^{\rig}\htimes\B_{\rig,F}^+$.  Since $\pi/t=(\varphi(\pi)/t)\cdot q^{-1}$, we see that
\[	q^{-(r_1+\cdots+r_d)}\cdot\det\left(\varphi^\ast\NN(M)[1/p]\rightarrow \NN(M)[1/p]\right)	\]
becomes an isomorphism after extending scalars to $R^{\rig}\htimes\B_{\rig,F}^+$.  But since the stalks at closed points of $\Spec(R\htimes\A_F^+)[1/p]$ are noetherian and their completions agree with the completed stalks of $(\Spf (R\htimes\A_F^+))^{\rig}$ by~\cite[Lemma 7.1.9]{dejong}, this implies the desired result.
\end{proof}

\begin{corollary}\label{cor:det-wach-rep}
Let $M$ be a family of positive crystalline representations of $G_F$ over $R$.  Then the inclusion
\[	\left((R\htimes\A^+)\otimes_{R\htimes\A_F^+}\NN(M)\right)\subset \left((R\htimes\A^+)\otimes_RM\right)	\]
induces an isomorphism of line bundles over $(R\htimes\A^+)[1/p]$
\[	\pi^{-(r_1+\cdots+r_d)}\cdot\det\left(((R\htimes\A^+)\otimes_{R\htimes\A_F^+}\NN(M))[1/p]\rightarrow \left((R\htimes\A^+)\otimes_RM\right)[1/p]\right)	\]
\end{corollary}
\begin{proof}
We use Lemma~\ref{lemma:wach-loc-free-mixed} to find an affine cover $\{U_i=\Spec R[1/p,1/f_i]\}$ of $\Spec R[1/p]$ such that $M[1/p]|_{U_i}$ and $\NN(M)[1/p]|_{U_i}$ are free.  Choose bases, and let $X_i\in\Mat_d((R\htimes\A)[1/p,1/f_i])$ be a matrix whose columns are the basis elements of $\NN(M)[1/p]|_{U_i}$ with respect to the basis of $M[1/p]|_{U_i}$.  Let $P_i$ be the matrix of Frobenius with respect to the chosen basis of $\NN(M)[1/p]|_{U_i}$.  Then $\varphi(X_i)=X_iP_i$ because Frobenius acts trivially on $M$.  This implies that $\varphi(\det(X_i))/\det(X_i)=\det(P_i)=q^{r_1+\cdots+r_d}$ up to a unit of $(R\widehat\otimes\A_F^+)[1/p,1/f_i]$, by Corollary~\ref{cor:det-frob-wach}.  But $\pi^h\left((R\htimes\A)[1/p,1/f_i]\right)\otimes_{R[1/p,1/f_i]}M[1/p]|_{U_i}\subset \left((R\htimes\A)[1/p,1/f_i]\right)\otimes_{(R\htimes\A_F^+)[1/p,1/f_i]}\NN(M)[1/p]|_{U_i}$ by Lemma~\ref{lemma:wach-sat-d+-inclusions}, so $(\det(X_i))$ is trivial after inverting $\pi$.  This implies that $\det(X_i)=\pi^{r_1+\cdots+r_d}$ up to a unit of $(R\widehat\otimes\A^+)[1/p,1/f_i]$.
\end{proof}

We can also use Proposition~\ref{quot-nm-dcris} to relate $\NN(M)^{\rig}$ and $\D_{\cris}(M^{\rig})$ when $M^{\rig}$ is a crystalline family with non-negative Hodge--Tate weights.
\begin{corollary}\label{nm-cris-pos-ht}
Suppose that $M$ is a crystalline family with non-negative Hodge--Tate weights.  Then $\varphi^\ast \NN(M)^{\rig}\subset (\B_{\rig,\Qp}^+\widehat\otimes  R^{\rig})\otimes_{ R^{\rig}}\D_{\cris}(M^{\rig})$.
\end{corollary}
\begin{remark}
Since $\varphi^\ast\NN(M)^{\rig}\rightarrow (\B_{\rig,\Qp}^+\widehat\otimes  R^{\rig})\otimes_{ R^{\rig}}\D_{\cris}(M^{\rig})$ is a map of finite modules over a Fr\'echet--Stein algebra, continuity is automatic.
\end{remark}
\begin{proof}
Suppose that the Hodge--Tate weights of $M^{\rig}$ lie in the interval $[0,h]$.  Then by definition $\NN(M)^{\rig}=\pi^{-h}\NN(M(-h))^{\rig}$, and $M^{\rig}(-h)$ is a positive crystalline family, so
\[	\varphi^\ast\NN(M)^{\rig}=\varphi(\pi)^{-h}\varphi^\ast\NN(M(-h))^{\rig}\subset t^{-h}\varphi^\ast\left((\B_{\rig,\Qp}^+\widehat\otimes  R^{\rig})\otimes_{ R^{\rig}}\D_{\cris}(M^{\rig}(-h))\right) 	\]
by Proposition~\ref{quot-nm-dcris} (since $\lambda_i=(t/\pi)^{r_i}$ and $r_i\leq h$).  But $\varphi^\ast\Dcris(M^{\rig})=\Dcris(M^{\rig})$ and $t^{-h}\D_{\cris}(M^{\rig}(-h)) = \D_{\cris}(M^{\rig})$, so we are done.
\end{proof}

We now turn to the other comparison between Wach modules and $\Dcris$, namely~\cite[Th\'eor\`eme III.4.4]{berger04}, which states that if $V$ is a positive crystalline representation of $G_{\Qp}$, then $\NN(V)/\pi\NN(V)\cong \Dcris(V)$.

\begin{theorem}\label{thm:wach-pi-dcris}
Let $M$ be a family of positive crystalline $G_{F}$-representations over $R$.  Then the composition of the natural maps
\[	\Dcris(M^{\rig})\hookrightarrow \NN(M)^{\rig}\twoheadrightarrow \NN(M)^{\rig}/\pi\NN(M)^{\rig}	\]
is an isomorphism.
\end{theorem}
\begin{proof}
The map of interest is a homomorphism of finite projective modules of the same rank over $R^{\rig}\otimes_{\Qp}F$.  But the formations of $\Dcris(M^{\rig})$ and $\NN(M)^{\rig}$ commute with base change on $R^{\rig}$ and the map becomes an isomorphism after specialization, so we are done.
\end{proof}

Since $\NN(M)$ provides a natural integral structure on $\NN(M)^{\rig}$, we obtain an integral version of $\Dcris$.  Namely, we define $\Dcris(M) := \NN(M) / \pi \NN(M)$, which is a finitely generated $R\otimes_{\Z_p}\cO_F$-module.  Note, however, that we do not know whether $\Dcris(M)$ is projective over $R$ (though $\Dcris(M)[1/p]$ is projective over $R[1/p]$).

\begin{corollary}\label{DcrisMandMrig}
Let $M$ be a family of positive crystalline $G_{F}$-representations over $R$.  Then \[\Dcris(M^{\rig})\cong  R^{\rig}\otimes_R \Dcris(M).\]
\end{corollary}

Theorem~\ref{thm:wach-pi-dcris} defines an isomorphism $\NN(M)^{\rig}/\pi\NN(M)^{\rig}\xrightarrow{\sim}\left(\NN(M)^{\rig}\right)^\Gamma$.  Since $\NN(M)^{\rig}$ has an underlying $R$-module structure, it is natural to ask whether this isomorphism can be defined over a smaller ring, such as $R[1/p]\htimes\B_{\rig,F}^+$.

%

Suppose that $L/\Qp$ is a finite extension, and suppose that $T$ is a positive crystalline representation of $G_F$ of rank $d$ over $\cO_L$, with Hodge--Tate weights in $[-h,0]$.  Then $\NN(T)$ is free of rank $d$ over $\cO_L\otimes_{\Zp}\A_F^+$; choose a basis $\{e_1,\ldots,e_d\}$ of $\NN(T)$ and let its image modulo $\pi$ be a basis of $\Dcris(T)$.  Write $A$ for the isomorphism $\NN(T)^{\rig}/\pi\NN(T)^{\rig}\xrightarrow{\sim}\left(\NN(T)^{\rig}\right)^\Gamma$ with respect to these bases.

Recall that $t:=\log(1+\pi)\in\B_{\rig,F}^+$.
\begin{lemma}
Let the notation be as above, and write $A=1+\sum_{m\geq 1} {A}^{(m)}t^m$ where ${A}^{(m)}$ is a matrix over $\cO_L\otimes_{\Z_p}\cO_F$ with entries ${a}_{ij}^{(m)}$.  Then $v_p({a}_{ij}^{(m)})\geq -\frac{(2p-1)}{(p-1)^2}m$, and in particular, the entries of $A$ (viewed as elements of $L\otimes_{\Qp}\B_{\rig,F}^+$) are bounded by $1$ on the disk $v_p(\pi)>p-1$.
\end{lemma}
\begin{proof}
Write $G=1+\sum_{m\geq 1}G^{(m)}t^m$ for the matrix of a non-torsion element $\gamma\in\Gamma$, where $G^{(m)}$ has entries $g_{ij}^{(m)}$ in $L\otimes_{\Qp}F$.  Now the exponential series $\pi=\exp(t)-1$ converges for any $t$ with $v_p(t)>\frac{1}{p-1}$ and defines a bijection between the affinoid disks $v_p(t)\geq \rho$ and $v_p(\pi)\geq \rho$ for any rational $\rho >\frac{1}{p-1}$.  Since $\Gamma$ acts by automorphisms on $\NN(M\otimes_R\cO_L)$, $G$ must be invertible over $\cO_L\otimes_{\Zp}\A_F^+$, and therefore its entries must be bounded by $1$ on any such disk.  Therefore $v_p(g_{ij}^{(m)})\geq -\frac{1}{p-1}m$.

Now the image of $A$ is $\Gamma$-invariant, so we must have $G\gamma(A)=A$, or equivalently,
\[	\sum_{r=0}^{m-1}G^{(m-r)}{A}^{(r)}\chi(\gamma)^r = (1-\chi(\gamma)^m)A^{(m)}	\]
for all $m$.  By induction on $m$, this implies that the matrix ${A}^{(m)}$ can be written as a sum of terms of the form
\[ \frac{\prod_{1 \le a \le M} G^{(\lambda_a)} \chi(\gamma)^N}{\prod_{b \in B} (1 - \chi(\gamma)^b)}\]
where $M$ and $N$ are integers, the $\lambda_a$ are positive integers satisfying $\lambda_1 + \dots + \lambda_M = m$, and $B$ is a subset of $\{1, \dots m\}$. Using the previous inequality, any such term has valuation bounded below by
\[ -\frac{1}{p-1} m - \sum_{1 \le r \le m} v_p\left(1 - \chi(\gamma)^r\right) \ge -\left(\frac{1}{p-1} + \frac{p}{(p-1)^2}\right)m = - \frac{2p-1}{(p-1)^2} m,\]
using the bound $\sum_{1 \le r \le m} v_p\left(1 - \chi(\gamma)^r\right) \le \frac{pm}{(p-1)^2}$ given in \cite[\S IV.1]{berger04}.
\end{proof}

Now we consider the $\varphi$-action on $\NN(T)$.  Let $P\in\Mat_{d\times d}(\cO_L\otimes_{\Zp}\A_F^+)$ be the matrix of $\varphi$ on $\NN(T)$ with respect to $\{e_i\}$ and let $\Phi\in \Mat_{d\times d}(\cO_L\otimes_{\Zp}\cO_F)$ be the matrix of $\varphi$ on $\NN(T)/\pi$; $\Phi$ is the ``constant term'' of $P$.  Then $\Phi$ is invertible over $L\otimes_{\Qp}F$, and more precisely, $p^h\Phi^{-1}\in \Mat_{d\times d}(\cO_L\otimes_{\Zp}\cO_F)$.

Let $\rho>0$.  Since $\varphi(\pi)=(1+\pi)^p-1\in \pi^p+p\Z_p[\![\pi]\!]$, there is some $N_\rho$ such that $\varphi^{N_\rho}$ carries the disk $v_p(\pi)\geq\rho$ into a sub-disk of $v_p(\pi)>\frac{2p-1}{(p-1)^s}$, and the induced map the other way on rings carries functions converging for $v_p(\pi)>\frac{2p-1}{(p-1)^s}$ to functions converging for $v_p(\pi)\geq\rho$.
\begin{lemma}
Let the notation be as above. The entries of $A$ have valuation bounded below by $-N_\rho h$ on the disk $v_p(\pi)\geq \rho$.
\end{lemma}
\begin{proof}
Since $A$ is Frobenius-equivariant, we have the relation $A\Phi=P\varphi(A)$.  Iterating $N_\rho$ times and multiplying by $p^{N_\rho h}$, we obtain $p^{N_\rho h}A=(P\varphi(P)\cdots\varphi^{N_\rho}(P))\varphi^{N_\rho}(A)(\varphi^{N_\rho}(p^h\Phi^{-1})\cdots p^h\Phi^{-1})$.  Thus, $p^{N_\rho h}A$ has entries which are bounded by $1$ on the disk $v_p(\pi)\geq \rho$.
\end{proof}

Now we return to $M$ over $R$.
Let $U=\Spec R[1/f_U,1/p]\subset\Spec R[1/p]$ with $f_U\in R\smallsetminus pR$ be an open affine subspace such that $\NN(M)[1/p]|_{U}$ is free.  Choose a basis $\{e_1',\ldots,e_d'\}$ for $\NN(M)[1/p]|_{U}$ and write $A'$ for the matrix of the isomorphism above with respect to the image of $\{e_i'\}$ in $\Dcris(M^{\rig})|_U$ and $\{e_i'\}$; we may assume $\{e_i'\}\subset \NN(M)[1/f_U]$.  Writing $\NN_{\{e_i'\}}$ for the $R\htimes\A_F^+$-span of $\{e_i'\}$, there is some integer $\beta$ such that $p^{-\beta}\NN_{\{e_i'\}}\cap \NN(M)[1/f_U]=\NN(M)[1/f_U]$, since $\NN(M)[1/f_U]$ is finitely generated over $(R\htimes\A_{\Qp}^+)[1/f_U]$ and $\{e_i'\}$ is a basis for $\NN(M)[1/f_U,1/p]$.

Let $x\in U$ be a closed point with residue field $L$ and let $\{e_1,\ldots,e_d\}$ be a basis of $\NN(M\otimes_R\cO_L)$.
Recall that $p^{\alpha(h)}\NN(M\otimes_R\cO_L)\subset \NN(M)\otimes_R\cO_L$; it follows that $p^{\alpha(h)+\beta}\NN(M\otimes_R\cO_L)\subset \NN_{\{e_i'\}}\otimes_R\cO_L$.  Thus, if $X$ is the matrix for the image of $\{e_i'\}$ in $\NN_{\{e_i'\}}\otimes_R\cO_L$ with respect to $\{e_i\}$, $X$ and $p^{\alpha(h)+\beta}X^{-1}$ have coefficients in $\cO_L\otimes_{\Z_p}\A_F^+$.  It follows that $A'\otimes_{R[1/p]}L$ is a matrix with entries in $L\otimes_{\Qp}\B_{\rig,F}^+$ that have valuation bounded below by $-(\alpha(h)+\beta+N_\rho h)$ on the disk $v_p(\pi)\geq \rho$.

In other words, there is some integer $k$ such that $f_U^kA'$ has entries in $R^{\rig}\htimes\B_{\rig,F}^+$ which are bounded on closed disks $v_p(\pi)\geq \rho$.  Thus, $A'$ has entries in $(R[1/p]\htimes\B_{\rig,F}^+)[1/f_U]$, by Proposition~\ref{prop:partial-gen-fiber-desc}.

\begin{corollary}\label{cor:det-dcris-nm-bdd}
There is an isomorphism of line bundles over $R[1/p]\htimes\B_{\rig,F}^+$
\[	\Det_{R[1/p]\htimes\B_{\rig,F}^+}(R[1/p]\htimes\B_{\rig,F}^+)\otimes_R\Dcris(M)\xrightarrow{\sim} \Det_{R[1/p]\htimes\B_{\rig,F}^+}(R[1/p]\htimes\B_{\rig,F}^+)\otimes_{R\htimes_{\Zp}\A_F^+}\NN(M)	\]
which agrees with the isomorphism $(\pi/t)^{r_1+\cdots+r_d}\cdot\det(\mathrm{inc}_M)$ of Proposition~\ref{quot-nm-dcris}.
\end{corollary}

\section{Galois cohomology}\label{sec:cohomology}

Let $K_\infty/\Qp$ be an abelian $p$-adic Lie extension with Galois group $G:=G(K_\infty/\Qp)$ and associated Iwasawa algebra $\Lambda:=\Zp[\![G]\!].$  We write $J_G$ for its Jacobson radical. Let $R$ be a complete local noetherian $\Z_p$-algebra with finite residue field, flat over $\Z_p$, and let $\mathfrak{m}_R$ be its maximal ideal.  Then we define the ring
\[	\Lambda_R(G):=\Lambda\htimes_{\Zp}R	\]
which has Jacobson radical $\mathfrak{J}= J_G\Lambda_R(G) + \mathfrak{m}_R\Lambda_R(G)$.  If $M$ is a family of $G_{\Qp}$-representations over $R$, we define the $\Lambda_R(G)$-module
\[\TT(M):=\Lambda\htimes_{\Zp}M\]
where the $\Lambda_R(G)$-module structure is induced by $(\lambda\otimes r)\cdot (\lambda' \otimes m):= \lambda\lambda'\otimes rm.$ In addition we endow the latter with the following action of $G_{\Qp}:$ for $\sigma$ in the latter group we set $\sigma(\lambda\otimes m):=\lambda \overline{\sigma}^{-1}\otimes \sigma m.$ Now we define the Iwasawa cohomology of a family $M$ of $G_{\Qp}$-representations as
\[\H_{\Iw}^i(M):=\H_{\Iw}^i(K_\infty,M):=\H^i(\Qp,\TT(M)).\]
Fontaine proved that when $R=\Z_p$, there is a natural isomorphism $\D(M)^{\psi=1}\xrightarrow{\sim}\H_{\Iw}^1(M)$, where the map $\psi:\NN(M)\rightarrow \NN(M)$ is defined to be the composition
\[	\NN(M)\subset\varphi^\ast\NN(M)\xrightarrow{\psi\otimes 1}\NN(M)	\]
Dee~\cite[Theorem 3.3.4]{dee01} extended this isomorphism to formal families.

In fact, if $M$ is crystalline, then Iwasawa cohomology can be computed from the Wach module of $M$:
\begin{proposition}\label{prop:iwasawacohomology}
Suppose that $K_\infty=\Qp(\mu_{p^\infty})$ and $M$ is crystalline with Hodge--Tate weights in $[0,h]$.  Then there is a canonical isomorphism
\[	\H_{\Iw}^1(M)\xrightarrow{\sim} (\pi^{-1}\NN(M))^{\psi=1}	\]
functorial in the coefficients.

Suppose in addition that for every maximal ideal $\mathfrak{m}\subset R[1/p]$ the specialization $(R[1/p]/\mathfrak{m})\otimes_RM$ has no quotient isomorphic to the trivial representation.  Then $\H_{\Iw}^1(M)\xrightarrow{\sim} \NN(M)^{\psi=1}$.
\end{proposition}
Here we use the non-negativity of the Hodge--Tate weights of $M$ to assume $\NN(M)\subset\varphi^\ast\NN(M)$.

\begin{proof}
It suffices to show that $(\pi^{-1}\NN(M))^{\psi=1}=\D(M)^{\psi=1}$, and to see this, it suffices to show that $(\pi^{-1}\NN(M)[1/p])^{\psi=1}=\D(M)[1/p]^{\psi=1}$.  But $\D(M)^{\psi=1}=\H_{\Iw}^1(M)$ is a finite $\Lambda_R(G)$-module by~\cite[Proposition 1.6.5(2)]{fukayakato06}, so $(\pi^{-1}\NN(M))^{\psi=1}\subset \D(M)^{\psi=1}$ is, as well.

Let $x\in\D(M)[1/p]^{\psi=1}$.  Then for every maximal ideal $\mathfrak{m}\subset R[1/p]$, the image of $x$ in $\D((R/\mathfrak{m})\otimes_RM)[1/p]$ lands in $\D((R[1/p]/\mathfrak{m})\otimes_RM)[1/p]^{\psi=1}\subset\pi^{-1}\NN((R[1/p]/\mathfrak{m})\otimes_RM)[1/p]$.  Since the formation of $\NN(M)[1/p]$ commutes with base change on $R[1/p]$, it follows that $x\in \pi^{-1}\NN(M)[1/p]$.

The second claim follows similarly.
\end{proof}

 Let $d$ be the $R$-rank of $M$.

\begin{proposition}\label{prop:HIwrk}
The Iwasawa cohomology groups $\H_{\Iw}^i(K_\infty,M)$ are finite $\Lambda_{R}(G)$-modules.  In fact, $\H_{\Iw}^i(K_\infty,M)=0$ for $i\notin\{1,2\}$, $\H_{\Iw}^1(K_\infty,M)[1/p]$ has generic rank $d$, and $\H_{\Iw}^2(K_\infty,M)$ is $R$-finite (and annihilated by an ideal of $\Lambda_{\Z_p}(G)$). 
\end{proposition}
\begin{proof}
The $\Lambda_{R}(G)$-finiteness statement follows from \cite[Proposition 1.6.5(2)]{fukayakato06}.  The vanishing statements follow from their classical analogs.

A consideration of Tate local duality implies that $\H_{\Iw}^2(K_{\infty},M)\cong (M^\vee(1)^G)^\vee\cong M(-1)_G$, which is $R$-finite.  The perfectness of $R\Gamma_{\Iw}(K_{\infty},M):=R\Gamma(\Qp,\TT(M))$ combined with the fact that $\H_{\Iw}^i(K_\infty,M)=0$ for $i\geq 3$ implies that the formation of $\H_{\Iw}^2(K_\infty,M)$ commutes with base change on $\Lambda_{R}(G)$, and in particular, with base change on $R$.  For a closed point $x\in\Spec R[1/p]$, let $\cO_{L_x}$ denote the ring of integers in the residue field at $x$.  Since $\H_{\Iw}^2(K_\infty,M)$ is finite as an $R$-module, there is a finitely generated ideal $I_x\subset\Lambda_{\Z_p}(G)$ which annihilates $\H_{\Iw}^2(K_\infty,M\otimes_R\cO_{L_x})=\H_{\Iw}^2(K_\infty,M)\otimes_R\cO_{L_x}$.  This implies that $I_x$ annihilates $\H_{\Iw}^2(K_\infty,M)$ in a Zariski open neighborhood of $x$, and since $\Spec \Lambda_R(G)$ is quasi-compact, we can multiply finitely many ideals together to get a finitely generated ideal of $\Lambda_{\Z_p}(G)$ which annihilates $\H_{\Iw}^2(K_\infty,M)$.

The preceding paragraph implies that over an open dense subspace of $\Spec \Lambda_{R}(G)[1/p]$ (obtained by inverting elements coming from $\Lambda_{\Z_p}(G)[1/p]$), the formation of $\H_{\Iw}^1(K_\infty,M)[1/p]$ also commutes with base change on $R[1/p]$.  But then comparison with the classical case shows that $\H_{\Iw}^1(K_\infty,M)[1/p]$ has generic rank $d$.
%
\end{proof}

For our later calculations we need the convergent cohomological spectral sequence for any   $(R',R)$-bimodule $Y$:
  \begin{equation}
   \label{f:spectralsequence} E_2^{i,j}:=\mathrm{Tor}_{-i}^{R}(Y,H^j_\Iw(M))\Rightarrow \H^{i+j}_\Iw(Y\otimes_{R}M), \end{equation}
  which is induced from the isomorphism
  \[Y\otimes^\mathbb{L}_R R\Gamma_\Iw(M)\cong R\Gamma_\Iw(Y\otimes_{R} M),\] proved in \cite{fukayakato06}. We obtain   the following exact sequence of terms in lower degree
\begin{multline}\label{f:lowdegree}
          \xymatrix{
     0\ar[r]&  {\mathrm{Tor}_2^R( Y,\H^2_\Iw(M))}  \ar[r]& Y\otimes_R\H^1_\Iw(M) \ar[r]^{ } & {\H^1_\Iw(Y\otimes_R M )} \ar[r]^{ } &  {\mathrm{Tor}_1^R(Y,\H^2_\Iw(M))} \ar[r]^{ } & 0   }
\end{multline}
as well as the isomorphisms
\[ Y\otimes_R\H^2_\Iw(M)\cong \H^2_\Iw(Y\otimes_R M)\;\;\;\mbox{ and }\;\;\;{\mathrm{Tor}_i^R(Y,\H^1_\Iw(M))}\cong\mathrm{Tor}_{i+2}^R(Y,\H^2_\Iw(M))  \]
for $i\geq 1.$ By Proposition \ref{prop:HIwrk} we obtain pseudo-isomorphisms, i.e., isomorphisms up to  pseudo-null modules,
\begin{equation}\label{H1modulopseudonull}
Y\otimes_R\H^1_\Iw(M)    \sim{\H^1_\Iw(Y\otimes_R M )} \;\;\;\mbox{ and }\;\;\; \mathrm{Tor}_i^R(Y,\H^1_\Iw(M))\sim 0
\end{equation}
for $i\geq 1.$

In the Fontaine-Lafaille range we obtain the following result which is not needed elsewhere in this article.

 \begin{proposition}\label{prop:phiN}
 If $M$ is crystalline  with Hodge–Tate weights in $[a,b]$ satisfying $h=b-a<p-1$, then $(\varphi^*\NN(M))^{\psi=0}$ is a free $R\widehat{\otimes}\Lambda(G)$-module of rank $r.$
 \end{proposition}

 \begin{proof}
 From \cite{SSV} we obtain the isomorphisms
 \[(\varphi^*\NN(M/I^nM))^{\psi=0}=\bigoplus_{j=1}^{p-1} (1+\pi)^j\otimes \NN(M/I^nM) \] which induces after taking projective limits an isomorphism of $R\widehat{\otimes}\Lambda(G)$-modules
 \[(\varphi^*\NN(M))^{\psi=0}=\bigoplus_{j=1}^{p-1} (1+\pi)^j\otimes \NN(M) .\] Now choose lifts $n_1,\ldots, n_r$ in $\NN(M)$ of any $\AA^+_{\Qp}$-basis $\bar{n}_1,\ldots, \bar{n}_r$ of $\NN(M/IM)$. Since \begin{eqnarray*}
  (\varphi^*\NN(M))^{\psi=0}/I(\varphi^*\NN(M))^{\psi=0}&=&\bigoplus (1+\pi)^j\otimes \NN(M)/I\NN(M)\\
  &\cong&\bigoplus (1+\pi)^j\otimes \NN(M /I M)
 \end{eqnarray*}
 is isomorphic to $\Lambda(G)^r$ with basis $\bar{n}_1,\ldots, \bar{n}_r$ by (loc.\ cit.) and the fact that by assumption the base change map for $\NN$ is an isomorphism, we see by the Nakayama lemma that the map \[(R\widehat{\otimes}\Lambda(G))^r \to (\varphi^*\NN(M))^{\psi=0}\] induced by the choice of $n_1,\ldots, n_r$ is surjective with kernel, say, $C$. By induction on $d$ and using torsionfreeness of $\varphi^*\NN(M)$ we conclude that $C/X_dC =0$ whence $C=0$ by the Nakayama lemma again. The proposition follows and we even get an explicit basis.
 \end{proof}

{

We now specialize to the following setting (which will be crucial for the construction of the regulator maps): Let $F_\infty$ be the unramified $\Zp$-extension of $\Qp$. We set $K_\infty = F_\infty(\mu_{p^\infty})$, $U = \Gal(F_\infty / \Qp )\cong \Gal(K_\infty/\Qpi)$ and $G = \Gal(K_\infty / \Qp)$. We regard $\Gamma$ as a subgroup of $G$, by identifying it with $\Gal(K_\infty / F_\infty)$, so we have $G = U \times \Gamma$.

Following  \cite{loefflerzerbes11} we define
  \[\NN_\infty(M):=\NN(M)\widehat{\otimes}_{\Zp} S_\infty\]
Here $S_\infty:=\varprojlim_n \cO_{F_n}$ with the transition maps being the trace maps.  This is called the \emph{Yager module}, and it is free of rank one over $\Lambda(U)$ and naturally endowed with a Frobenius action.  Moreover, $S_\infty$ is naturally contained in $\Lambda_{\mathcal{O}_{\widehat{F}_\infty}}(U)$, and the induced subspace topology coincides with  the inverse limit topology (see the discussion in \S 3 in loc.~cit.). Indeed, setting
\[\Lambda_{\Zp}(U)_{\tau_p}:=\{x\in \widehat{\mathbb{Z}_p^{\nr}}\widehat\otimes_{\Zp} \Lambda_{\Zp}(U)| (\phi\otimes 1)x=(1\otimes \overline{\tau_p})x\},\]
where $ \overline{\tau_p}$  denotes the image  of  the unique lift $\tau_p$  in $G(\mathbb{Q}_p^{\mathrm{ab}}/\Qp(\mu(p)))$ of the arithmetic Frobenius $\sigma\in G(\mathbb{Q}_p^{\nr}/\Qp)$   one checks (\cite[Prop.\ 3.6]{loefflerzerbes11}) that \begin{equation}\label{f:Yager} S_\infty \cong \Lambda_{\Zp}(U)_{\tau_p}.\end{equation}

 Similarly, we consider for the the algebras of $L$-valued locally analytic
distributions $\cH_{\Qp}(U)$ and $\cH_{\Qp}(G)$ on $U$ and $G,$ respectively,
\[\cH_{\Qp}(U)_{\tau_p}:=\{x\in \widehat{\mathbb{Z}_p^{\nr}}\widehat\otimes_{\Zp} \cH_{\Qp}(U)| (\sigma\otimes 1)x=(1\otimes \overline{\tau_p})x\}\]
and
\[\cH_{\Qp}(G)_{\tau_p}:=\{x\in \widehat{\mathbb{Z}_p^{\nr}}\widehat\otimes_{\Zp} \cH_{\Qp}(G)| (\sigma\otimes 1)x=(1\otimes \overline{\tau_p})x\},\]
where in the latter case $\overline{\tau_p} $ denotes the image of $\tau_p$ in $G:=\Gamma\times U$ with $\Gamma\cong\Zp.$

\begin{lemma}\label{lem:twists}
\begin{enumerate}
 \item $\Lambda_{\Zp}(U)_{\tau_p}$ and  $\cH_{\Qp}(U)_{\tau_p}$ are  free rank one modules over $\Lambda_{\Zp}(U) $ and $\cH_{\Qp}(U)$, respectively.
 \item $\Lambda_{\Zp}(G)_{\tau_p}$ and  $\cH_{\Qp}(G)_{\tau_p}$ are  free rank one modules over $\Lambda_{\Zp}(G) $ and $\cH_{\Qp}(G)$, respectively.
 \end{enumerate}
 Moreover, we have  \[ \cH_{\Qp}(U)_{\tau_p}\widehat\otimes_{\Qp}\cH_{\Qp}(\Gamma)=\cH_{\Qp}(G)_{\tau_p}.\]
\end{lemma}

\begin{proof} (1) Let $\cO$ be either $\Zp$ or $\widehat{\Zp^{nr}}$ and set $L:=\cO[\frac{1}{p}].$
For the purpose of the proof we identify $\Lambda_{\cO}(U)$ and $\cH_{L}(U) $ with $\cO[[T]]$ and the subring of $L[[T]]$ consisting of power series which converge on the open unit disk, respectively. Then the defining condition of e.g. $\Lambda_{\Zp}(U)_{\tau_p} $ is equivalent to $(\sigma-1)\lambda=T\lambda$ whence  we obtain the isomorphism
\[\Lambda_{\Zp}(U)_{\tau_p}\cong \{(a_i)\in \varprojlim_{i\geq 0,\sigma-1} \widehat{\Zp^{\nr}}| a_0\in\Zp\}\] sending $\lambda=\sum a_iT^i$ to $ (a_i)$. Fix any such $\lambda_0$ with $a_0=1,$ the existence of which stems from the exact sequence
\[0\to\Zp\to\widehat{\Zp^{\nr}}\xrightarrow{\sigma-1} \widehat{\Zp^{\nr}}\to 0.\] Then $\lambda_0$ is a unit in $\Lambda_{\widehat{\Zp^{\nr}}}(U)$ and for any $\lambda $ in $\Lambda_{\Zp}(U)_{\tau_p}$ or  $\cH_{\Qp}(U)_{\tau_p}$ we have that
\[\sigma(\lambda\lambda_0^{-1})=\sigma(\lambda)\sigma(\lambda_0)^{-1}= \overline{\tau_p}\lambda \overline{\tau_p}^{-1}\lambda_0^{-1}=\lambda\lambda_0^{-1},\]
i.e., $\lambda\lambda_0^{-1}$ belongs to $\Lambda_{\Zp}(U)$ or  $\cH_{\Qp}(U)$, respectively. This implies that $\lambda_0$ forms a basis in both cases.
(2) Since $\lambda_0\in \Lambda_{\Zp}(U)_{\tau_p}\subseteq \Lambda_{\Zp}(G)_{\tau_p}$  the last argument applies in the same way. From this the last statement is clear, too.
\end{proof}

Using the same argument as in the proof of \cite[Theorem 4.4, Proposition 4.5]{loefflerzerbes11} we obtain the following sharpening of Proposition~\ref{prop:iwasawacohomology}:

\begin{proposition}\label{prop:iwasawacohomologyunramified}
 If $M$ is crystalline with Hodge-Tate weights in $[0,h]$, then there is a canonical isomorphism \[\H^1_{\Iw}(K_\infty,M)\cong \NN_\infty(M)^{\psi=1}.\]
\end{proposition}
}

\section{Regulator maps for families}\label{sec:regulator}

We have to extend the various regulator maps to families and thus generalize the main result from \cite{loefflerzerbes11}. Let $F$ be any finite unramified extension of $\Qp$, and let $F_\infty$ be the unramified $\Zp$-extension of $F$. We set $K_\infty = F_\infty(\mu_{p^\infty})$, $U = \Gal(K_\infty / \Qpi)$ and $G = \Gal(K_\infty / \Qp)$. We regard $\Gamma$ as a subgroup of $G$, by identifying it with $\Gal(K_\infty / F_\infty)$, so we have $G = U \times \Gamma$.

\subsection{$\Gamma$-regulator}

Let $M$ now be a crystalline family with non-negative Hodge-Tate weights and not containing the trivial representation as a quotient (of  any $M_n$). Then we have the following map:
\[   \mathcal{L}^{\Gamma}_{\Qp, M,\xi}:     \xymatrix{
        \H^1_{\Iw}(  \Qp(\mu_{p^\infty})/\Qp, M )\cong \NN(M)^{\psi=1}\ar[r]^(0.65){1-\varphi}  & (\varphi^*\NN(M))^{\psi=0} \ar@{^(->}[r]^{} & (\cH_{\Qp}(\Gamma)\widehat\otimes  R^{\rig})\otimes_{ R^{\rig}}\DD_{\cris}(M^{\rig})  }                              \]
and the last inclusion follows from Corollary~\ref{nm-cris-pos-ht} and the identification
\begin{equation}
\label{BrigDistributions}
(\BB_{\rig,\Qp}^+\widehat\otimes R^{\rig})^{\psi=0}\otimes_{ R^{\rig}}\DD_{\cris}(M^{\rig})\cong (\cH_{\Qp}(\Gamma)\widehat\otimes R^{\rig})\otimes_{ R^{\rig}}\DD_{\cris}(M^{\rig}).
\end{equation}
 Note that the latter is a projective $\cH_{\Qp}(\Gamma)\widehat\otimes  R^{\rig}$-module because $M$ is assumed crystalline.

For any finite extension $E / \Qp$ contained in $F_\infty$,  we may apply the  above construction to the induced representation $M':=\mathrm{Ind}^{\Qp}_EM$. Using Shapiro's Lemma and Frobenius  reciprocity (see \S 4.3 in \cite{loefflerzerbes11} for a similar construction), we then obtain the following map:
\[   \mathcal{L}^{\Gamma}_{E, M,\xi}:     \xymatrix{
        \H^1_{\Iw}(  E(\mu_{p^\infty})/E, M ) \ar[r] & (\cH_{E}(\Gamma)\widehat\otimes  R^{\rig})\otimes_{ R^{\rig}}\DD_{\cris}(M^{\rig}).  }                              \]

\subsection{$G$-regulator}

Assume first that $F=\Qp$, i.e., $U\cong\Zp.$ Extending the observations from (loc.\ cit.) one easily shows that
\[(\varphi^*(\NN_\infty(M))^{\psi=0}\cong(\varphi^*\NN(M))^{\psi=0}\widehat{\otimes}_{\Zp} S_\infty.\]
    Using  Proposition \ref{prop:iwasawacohomologyunramified}    we obtain the following map
   \[   \mathcal{L}^{G}_{ M,\xi}:     \xymatrix{
        \H^1_{\Iw}(  K_\infty, M )\cong \NN_\infty(M)^{\psi=1}\ar[r]^(0.65){1-\varphi}  & (\varphi^*\NN_\infty(M))^{\psi=0} \ar@{^(->}[r]^{} & (\cH_{\Lt}(G)\widehat\otimes_{\Qp}  R^{\rig})\otimes_{ R^{\rig}}\DD_{\cris}(M^{\rig})  }                              \]
generalizing definition 4.6 in (loc.~cit.). More precisely, as in \S 4.2 of (loc.~cit.) one has an embedding
\begin{equation}
\label{eq:embeddingSinfty}
S_\infty \hookrightarrow \cH_{\Qp}(U)_{\tau_p} \hookrightarrow \cH_{\Lt}(U)
\end{equation}
which is continuous with respect to the  inverse limit topology   on $S_\infty$ and the Fr\'{e}chet topology on the target. Hence the natural map on the algebraic tensor product
\[(\varphi^*\NN(M))^{\psi=0}\otimes_{\Zp}S_\infty \to    (\cH_{\Qp}(\Gamma)\widehat\otimes R^{\rig})\otimes_{ R^{\rig}}\DD_{\cris}(M^{\rig})\otimes_{\Qp}\cH_{\Lt}(U) \]
induced from the embedding in Corollary~\ref{nm-cris-pos-ht} plus \eqref{BrigDistributions} and \eqref{eq:embeddingSinfty} extends to a continuous $R\widehat{\otimes}_{\Zp}\Lambda(U\times\Gamma)$-equivariant map
\[(\varphi^*\NN(M))^{\psi=0}\widehat\otimes_{\Zp}S_\infty \to    (\cH_{\Lt}(U)\widehat\otimes_{\Qp}\cH_{\Qp}(\Gamma)\widehat\otimes_{\Qp} R^{\rig})\otimes_{ R^{\rig}}\DD_{\cris}(M^{\rig})\cong (\cH_{\Lt}(G) \widehat\otimes_{\Qp}  R^{\rig})\otimes_{ R^{\rig}}\DD_{\cris}(M^{\rig})\]
The image of this map is contained in \[(\cH_{\Qp}(U)_{\tau_p}\widehat\otimes_{\Qp}\cH_{\Qp}(\Gamma)\widehat\otimes_{\Qp}  R^{\rig})\otimes_{ R^{\rig}}\DD_{\cris}(M^{\rig})=(\cH_{\Qp}(G)_{\tau_p} \widehat\otimes_{\Qp}  R^{\rig})\otimes_{ R^{\rig}}\DD_{\cris}(M^{\rig})\]
Since $\cH_{\Qp}(G)_{\tau_p} \widehat\otimes_{\Qp}  R^{\rig}$  is a free rank one $\cH_{\Qp}(G)\widehat\otimes_{\Qp}  R^{\rig}$-module by Lemma \ref{lem:twists}, the source and target are finite $\cH_{\Qp}(G)\widehat\otimes_{\Qp}  R^{\rig}$-modules of the same generic rank.

The extension to finite unramified extensions $F/\Qp$ (which we may and do assume to be linearly disjoint from the unramified $\Zp$-extension of $\Qp$) is carried out as usual by applying the above construction to the induced representation $M':=\mathrm{Ind}^{\Qp}_FM$ and using Shapiro's lemma and Frobenius  reciprocity; see \S 4.3 in (loc.\ cit.) for the details.

The first part of the following theorem generalizes Thm. 4.7 in (loc.\ cit.) to families:

  \begin{theorem}\label{thm:specialization}
   Assume that $M$ is crystalline with Hodge-Tate weights $\geq 0$. Then the following hold:
   \begin{enumerate}
   \item
   For any finite extension $E / \Qp$ contained in $F_\infty$, we have a commutative diagram
   \[
    \begin{diagram}
     \H^1_{\Iw}(K_\infty, M) &\rTo^{\mathcal{L}_{M,\xi}^G} &(\cH_{\Lt}(G)\widehat\otimes  R^{\rig})\otimes_{ R^{\rig}}\DD_{\cris}(M^{\rig})\\
     \dTo & &\dTo \\
     \H^1_{\Iw}(E(\mu_{p^\infty}), M)& \rTo^{\mathcal{L}_{M,\xi}^{G'}} & (\cH_{\Lt}(G')\widehat\otimes  R^{\rig})\otimes_{ R^{\rig}}\DD_{\cris}(M^{\rig}).
    \end{diagram}
   \]
   Here $G' = \Gal(E(\mu_{p^\infty}) / \Qp)$, the right-hand vertical arrow is the map on distributions corresponding to the projection $G \rightarrow G'$, and the map $\mathcal{L}_{M,\xi}^{G'}$ is defined by
   \[ \mathcal{L}_{M,\xi}^{G'} = \sum_{\sigma \in \Gal(E / \Qp)} [\sigma] \cdot \mathcal{L}^{\Gamma}_{E, M,\xi}(\sigma^{-1} \circ x),\]
   where $ \mathcal{L}^{\Gamma}_{E, M,\xi}$ is the cyclotomic regulator map for $E(\mu_{p^\infty}) / E$. Moreover, the map $\cL^G_{M, \xi}$ is injective.
   \item\label{diag:reg-comm}  For any ideal $P\subset R$ such that $R/P$ is flat over $\Z_p$, there is a commutative diagram
   \[
    \begin{diagram}
     \H^1_{\Iw}(K_\infty, M) &\rTo^{\mathcal{L}_{M,\xi}^G} &(\cH_{\Lt}(G)\widehat\otimes  R^{\rig})\otimes_{ R^{\rig}}\DD_{\cris}(M^{\rig})\\
     \dTo & &\dTo \\
     \H^1_{\Iw}(K_\infty, M/P)& \rTo^{\mathcal{L}_{M/P,\xi}^{G}} & (\cH_{\Lt}(G)\widehat\otimes  R^{\rig}/P  R^{\rig})\otimes_{ R^{\rig}/P  R^{\rig}}\DD_{\cris}((M/P)^{\rig}).
    \end{diagram}
       \]
    \item For any $x\in \H^1_{\Iw}(K_\infty, M)$, we have \[\sigma( \mathcal{L}_{M,\xi}^G(x))=\tau_p\cdot\mathcal{L}_{M,\xi}^G(x)\] where $\sigma$ denotes the arithmetic Frobenius.
   \end{enumerate}
   \end{theorem}

  By abuse of notation we also write $\cL^G_{M, \xi}$ for the induced map with source $(\cH_{\Lt}(G)\widehat\otimes  R^{\rig})\otimes \H^1_{\Iw}(K_\infty,M)$.

\begin{proof}
The proof of Thm. 4.7 in (loc.\ cit.) carries over almost literally to show the first part apart from the injectivity statement, which is proved in the same way as Prop. 4.12 in (loc.\ cit.).

 The commutativity of the  diagram in~\ref{diag:reg-comm} is just the fact that the construction of $\cL^G_{M, \xi}$ is functorial in the coefficients of $M$: More precisely, we have to check that the following diagrams are commutative.
 \begin{equation}\label{HIwN}
  \xymatrix{
   \H^1_{\Iw}(K_\infty,M)\otimes_R R/P \ar[d]_{ } \ar[r]^{\sim} & {\NN_\infty(M)^{\psi=1}\otimes_R R/P}  \ar[d]^{ } \\
   \H^1_{\Iw}(K_\infty,M\otimes_R R/P) \ar[r]^{\sim} & {\NN_\infty(M\otimes_R R/P )^{\psi=1}}   }
 \end{equation}
 (where the right vertical map is induced from the natural base change map $N(M)\otimes_R R/P\to N(M/PM)$ from section \ref{families} and the $R$-linearity of $\psi$ while the left vertical map comes from the obvious base change spectral sequence using the flatness of $M$ over $R,$ see e.g. \cite[\S 8.4.8.3]{nekovar06} or \eqref{f:spectralsequence}, \eqref{f:lowdegree}),
  \begin{equation}\label{NtoNpsi=0}
  \xymatrix{
   {\NN_\infty(M)^{\psi=1}\otimes_R R/P}\ar[d]_{ } \ar[r]^{1-\varphi} & {(\varphi^*\NN_\infty(M))^{\psi=0}\otimes_R R/P}  \ar[d] \\
   {\NN_\infty(M\otimes_R R/P )^{\psi=1}} \ar[r]^{1-\varphi} & {(\varphi^*\NN_\infty(M\otimes_R R/P ))^{\psi=0}}   }
 \end{equation}
  and
  \begin{equation}\label{Npsi=0Dcris}
  \xymatrix{
   {(\varphi^*\NN_\infty(M))^{\psi=0}\otimes_R R/P}\ar[d] \ar@{^(->}[r]^{} & {(\cH_{\Lt}(G)\widehat\otimes  R^{\rig})\otimes_{ R^{\rig}}\DD_{\cris}(M^{\rig})\otimes_R R/P}  \ar[d] \\
   {(\varphi^*\NN_\infty(M\otimes_R R/P ))^{\psi=0}} \ar@{^(->}[r]^{} & {(\cH_{\Lt}(G)\widehat\otimes  R^{\rig})\otimes_{ R^{\rig}}\DD_{\cris}(\left(M\otimes_R R/P\right)^{\rig})}  }
 \end{equation}

The commutativity of diagram \eqref{HIwN} follows from the  functoriality of the Fontaine isomorphism.  For diagram~\eqref{NtoNpsi=0}, commutativity follows immediately from the $R$-linearity of $\varphi$. The commutativity of \eqref{Npsi=0Dcris} follows from the commutativity of the diagram
\begin{equation*}
  \xymatrix{
   {(\varphi^*\NN (M))^{\psi=0}\otimes_R R/P}\ar[d] \ar[r]^{} & {(\cH_{\Qp}(\Gamma)\widehat\otimes  R^{\rig})\otimes_{ R^{\rig}}\DD_{\cris}(M^{\rig})\otimes_R R/P}  \ar[d]^{\cong } \\
   {(\varphi^*\NN (M\otimes_R R/P ))^{\psi=0}} \ar@{^(->}[r]^{} & {(\cH_{\Qp}(\Gamma)\widehat\otimes  R^{\rig})\otimes_{ R^{\rig}}\DD_{\cris}(\left(M\otimes_R R/P\right)^{\rig})}  }
 \end{equation*}
But this follows from the functoriality of the maps $\NN(M)\rightarrow \D_{\rig}^\dagger(M^{\rig})$ and $\D_{\cris}(M^{\rig})\rightarrow \D_{\rig}^\dagger(M^{\rig})$, which is clear.

The third part is clear from the construction.
\end{proof}

\section{Construction of the isomorphism}
 \label{sect:constructionG}

The aim of this section is as follows. Let $M$ be a crystalline $R$-linear Galois representation  with $R$ a complete local noetherian $\Z_p$-algebra with finite residue field, which is Cohen--Macaulay, normal, and flat over $\Z_p$, such that $R^{\rig}$ is an integral domain.  Let $\widetilde R:=\widehat{\Z_p^{\nr}}\htimes_{\Z_p}R$ and let \begin{equation}\label{f:tilde}
  \Lambda_R(G):=\Lambda_{\Z_p}(G)\htimes_{\Z_p} R\;\;\;\mbox{ and }\;\;\;\Lambda_{\widetilde R}(G):=\Lambda_{\Z_p}(G)\htimes_{\Z_p}\widetilde R.
 \end{equation}

We shall construct a canonical isomorphism of determinants over the ring $\Lambda_{\widetilde R,\Qp}(G):=\Lambda_{\widetilde R}(G)[1/p]$:
 \begin{multline*}
  \Theta_{\Lambda_{\widetilde R,\Qp}(G), \xi}(T) : \Det_{\Lambda_{\widetilde R,\Qp}(G)}(0) \xrightarrow{\sim} \\ \Lambda_{\widetilde R,\Qp}(G) \otimes_{\Lambda_{R}(G)} \left\{ \Det_{\Lambda_{R}(G)} R\Gamma_{\Iw}(K_\infty, M) \cdot \Det_{\Lambda_{R}(G)} \left(\Lambda_{R}(G) \otimes_{R} \Dcris(M)\right)\right\},
 \end{multline*}
 where $\Dcris(M)=N(M)/\pi N(M)$.

Our construction is to define the isomorphism over the total quotient ring $\cK_{\widetilde R}(G)$ of the distribution algebra $\cH_{\widetilde R}(G):=\cH_{\Qp}(G)\htimes \widetilde{R}^{\rig}$, and then descend it to $\Lambda_{\widetilde R,\Qp}(G)$.

In order to carry out our calculations, we will need to specialize at closed points of $\Spec R[1/p]$ and utilize the results of~\cite{loeffler-venjakob-zerbes}.  However, as we are interested in finite modules over $\Lambda_{\Rt}(G)[1/p]$, which has many more closed points than $\Lambda_R(G)[1/p]$, we record the following useful lemmas:

\begin{lemma}\label{lemma:qpbar-spec-inj}
The natural map $\Rt[1/p]\rightarrow \prod_x \widehat{\Qp^{\nr}}\htimes_{\Qp} K_x$ is injective, where the product runs over closed points of $\Spec R[1/p]$ and $K_x$ denotes the residue field at $x$.
\end{lemma}
\begin{proof}
According to Appendix \ref{sec:app1} we can exhaust the generic fiber $\Spf (R)^\rig$ of $\Spf R$, whose points corresponds to the closed points of $\Spec R[1/p]$, by affinoid subdomains $\Sp A_m$, which are  reduced since by one of our running hypotheses $R[1/p]$ is reduced.  Since $\widehat{\Zp^{\nr}}$ is unramified over $\Zp$, the maximal ideal of definition $I$ of $R$ generates the maximal ideal of definition of $\widehat{\Zp^{\nr}}\htimes_{\Zp}R$.  For every $m\geq 0$, there is a natural map
\[	\widehat{\Zp^{\nr}}\otimes_{\Zp}(R[I^m/p])^\wedge\rightarrow (\widehat{\Zp^{\nr}}\htimes_{\Zp}R)[I^m/p]	\]
which is an isomorphism after $p$-adically completing (since it is an isomorphism modulo every power of $p$).

Thus, we obtain a map
\[\widetilde{R}[1/p]\hookrightarrow \widetilde{R}^{\rig}\cong \varprojlim_m (\widetilde{R}[I^m/p])^\wedge[1/p]\cong \varprojlim_m (\widehat{\Zp^{\nr}}\htimes_{\Zp}(R[I^m/p])^\wedge)[1/p]\cong \varprojlim_m \widehat{\Qp^{\nr}}\htimes_{\Qp} A_m	\]
and we see that it is enough to consider the analogous map $\widehat{\Qp^{\nr}}\htimes A\rightarrow \prod_x \widehat{\Qp^{\nr}}\htimes K_x$ when $A$ is a reduced $\Qp$-affinoid algebra and $x$ runs over the points of $\Sp(A)$.

If $M$ is a non-archimedean Banach space and $I$ is an index set, let $c_I(M)$ denote the set of functions $h:I\rightarrow M$ such that $\lim_{i\in I}h(i)=0$.  Then $c_I(M)$ is again a Banach space.  Every Banach space over a discretely valued field $k$ is isomorphic to $c_I(k)$ for some index set $I$, by~\cite[Proposition 10.1]{schneider}.  In particular, $\widehat{\Qp^{\nr}}$ is isomorphic (as a $\Qp$-Banach space) to $c_I(\Qp)$ for some index set $I$.

Furthermore, the isomorphism $\widehat{\Qp^{\nr}}\xrightarrow{\sim}c_I(\Qp)$ behaves well with respect to completed tensor products: If $M$ is a $\Qp$-Banach space, then there is a natural isomorphism of $\Qp$-Banach spaces $\widehat{\Qp^{\nr}}\htimes M\xrightarrow{\sim} c_I(M)$, and it is functorial in $M$~\cite[Lemma A.2.2]{bellovin13}.

We may apply this in our situation to obtain maps
\[	\widehat{\Qp^{\nr}}\htimes A\xrightarrow{\sim}c_I(A)\rightarrow c_I(K_x)\xleftarrow{\sim}\widehat{\Qp^{\nr}}\htimes K_x	\]
for every $x\in \Sp(A)$.  Then if $f=(f_i)_{i\in I}\in\ker\left(\widehat{\Qp^{\nr}}\htimes A\rightarrow \prod_x \widehat{\Qp^{\nr}}\htimes K_x\right)$, we see that $f_i\in \ker\left(A\rightarrow\prod_xK_x\right)$ for all $i$.  But that implies that $f_i=0$ for all $i$ (since $A$ is assumed reduced), so $f=0$.
\end{proof}

\begin{lemma}\label{lemma:qpbar-spec-bdd}
Let $A$ be a reduced $\Q_p$-affinoid algebra, and suppose that for every $x:A\rightarrow L$, $L/\Q_p$ a finite extension with ring of integers $\cO_L$, the induced map $\widehat{\Qp^{\nr}}\htimes_{\Qp}A\rightarrow \widehat{\Qp^{\nr}}\otimes_{\Qp}L$ carries
$f\in \widehat{\Qp^{\nr}}\htimes_{\Qp}A$ to $\widehat{\Z_p^{\nr}}\otimes_{\Z_p}\cO_L$.  Then $\lvert f(x)\rvert\leq 1$ for all points $x\in\Sp(\widehat{\Q_p^{\nr}}\htimes_{\Qp}A)$.
\end{lemma}
\begin{proof}
If $A=T_n$, the result is clear because points of $\Sp(\widehat{\Qp^{\nr}}\htimes_{\Qp}T_n)$ lying over points of $\Sp(T_n)$ are dense with respect to the canonical topology.  To treat the general case, we will again assume $\Sp(A)$ is irreducible and use Noether normalization to find finite torsion-free monomorphisms $T_n\rightarrow A$ and $\widehat{\Z_p^{\nr}}\htimes_{\Z_p}T_n^\circ\rightarrow\widehat{\Z_p^{\nr}}\htimes_{\Z_p}A^\circ$.

Since the supremum norm on $A$ and $\widehat{\Qp^{\nr}}\htimes_{\Qp}A$ is power-multiplicative, we may replace $f$ by a power of itself and assume that $\lvert f\rvert_{\sup}\in\lvert\Qp\rvert$.

Suppose that $f\notin \widehat{\Z_p^{\nr}}\htimes_{\Z_p}A^\circ$.  Then there is some $k>0$ such that $p^kf\in \widehat{\Z_p^{\nr}}\htimes_{\Z_p}A^\circ$ and $\lvert p^kf(x)\rvert\leq p^{-1}$ for every point $x$ lying over a point of $\Sp(A)$.  Choose the minimal such $k$.  The first condition implies that $\lvert p^kf\rvert_{\sup}\leq 1$, and the minimality of $k$ implies that $\lvert p^kf\rvert_{\sup}>p^{-1}$.  Indeed, if $\{e_i\}$ is a basis for $\overline{\F}_p$ over $\F_p$, then the Teichm\"uller lifts $\{[e_i]\}$ form a basis for $\widehat{\Qp^{\nr}}$ over $\Qp$, and we may write $p^kf$ uniquely as $p^kf=\sum_if_i[e_i]$ with $f_i\in A^\circ$.  If $\lvert p^kf\rvert_{\sup}\leq p^{-1}$, then for every point $x\in\Sp(A)$ the reduction $\overline{p^kf(x)}=\sum_i e_i\overline{f_i(x)}$ modulo $p$ is $0$.  But since the $e_i$ are linearly independent, we see that $\overline{f_i(x)}=0$ for all $i$, and therefore $f_i\in pA^\circ$ for all $i$, contradicting the minimality of $k$.

The maximum modulus principle implies that $f$ actually attains its supremum at a point of $\Sp(\widehat{\Qp^{\nr}}\htimes_{\Qp}A)$, which lies over a point $x\in \Sp(\widehat{\Qp^{\nr}}\htimes_{\Qp}T_n)$.  As in the proof of the previous lemma, we may $p$-adically approximate $x$ by a series of points $x_n\in\Sp(T_n)$.  As before, $x(p^kf)\equiv (\widehat{\Z_p^{\nr}}\htimes_{\Z_p}x_n)(p^kf)\mod{p^n}$.  But $\lvert x(p^kf)\rvert_{\sup}=1$ (because $\lvert p^kf\rvert_{\sup}=1$) whereas $\lvert(\widehat{\Z_p^{\nr}}\htimes_{\Z_p}x_n)(p^kf)\rvert_{\sup}\leq p^{-1}$ (by our assumption on $f$ at points over points of $\Sp(A)$), which is a contradiction.  Thus, $k=0$ and we are done.
\end{proof}

We can think of these lemmas as saying that ``classical points, i.e., those lying above points of $\Spf(R)^{\rig}$, are dense in $\Spf(\Rt)^{\rig}$''.


 \subsection{Construction of $\Theta$ over \texorpdfstring{$\mathcal{K}_{\widetilde R}(G)$}{a large ring}}
  \label{sect:construction1}

  Over $\mathcal{K}_{\widetilde R}(G)$, the construction of the isomorphism $\Theta$ is very simple:



  \begin{proposition}
   The regulator $\cL^G_{M, \xi}$ induces an isomorphism
   \[ \Det_{\mathcal{K}_{\widetilde R}(G)}(0) \xrightarrow{\sim} \Det_{\mathcal{K}_{\widetilde R}(G)}\left((\mathcal{K}_{\widetilde R}(G)) \otimes_{\Lambda_{R}(G)} R\Gamma_{\Iw}(K_\infty, M)\right) \cdot \Det_{\mathcal{K}_{\widetilde R}(G)}\left((\mathcal{K}_{\widetilde R}(G)) \otimes_{R^{\rig}} \Dcris(M^\rig)\right).\]
  \end{proposition}

  \begin{proof}
   By Proposition \ref{prop:HIwrk} and property \ref{B.h)} in Appendix \ref{determinants},
   \[ \Det_{\Lambda_{\cO,R}(G)}\left(R\Gamma_{\Iw}(K_\infty, M)\right) \cong \Det_{\Lambda_{\cO,R}(G)}\left(\H^1_{\Iw}(K_\infty, M)\right)^{-1}\]
up to $\Lambda_{\Zp}(G)$-torsion.  We therefore consider the map
\[	\cL^G_{M, \xi}:\cH_{\Rt}(G)\otimes_{\Lambda_R(G)}\H_{\Iw}^1(K_{\infty},M)\rightarrow \mathcal{H}_{\Rt}(G) \otimes_{R^{\rig}} \Dcris(M^\rig).	\]
This map is defined by base extension from a map $\H_{\Iw}^1(K_\infty,M)\rightarrow (\cH_{\Qp}(G)_{\tau_p}\htimes R^{\rig})\otimes_{R^{\rig}}\Dcris(M^{\rig})$; choosing an $\cH_{\Qp}(G)$-basis of $\cH_{\Qp}(G)_{\tau_p}$, it suffices to show that the map of $\cH_{\Qp}(G)\htimes R^{\rig}$-modules
\[	\cL^G_{M, \xi}:\cH_{R}(G)\otimes_{\Lambda_R(G)}\H_{\Iw}^1(K_{\infty},M)\rightarrow (\cH_{\Qp}(G)_{\tau_p}\htimes R^{\rig}) \otimes_{R^{\rig}} \Dcris(M^\rig)	\]
is generically an isomorphism.

Since $\H_{\Iw}^1(K_\infty,M)$ is generically free of rank $d$ and $\Dcris(M^{\rig})$ is projective of rank $d$, it suffices to prove that this map is generically surjective.  But after specializing at any closed point $x$ of $\Spf(R)^{\rig}$, \cite[Proposition 3.4.1]{loeffler-venjakob-zerbes} implies that the image of the regulator map contains a basis of $\Frac\cH_{\Qp}(G)\otimes_{\Qp}\Dcris(M_x)$, so generic surjectivity follows.
  \end{proof}

  \begin{definition}
   \label{def:isooverK}
   Let $\Theta_{\mathcal{K}_{\widetilde R}(G), \xi}(M)$ be the isomorphism
   \[ \Det_{\mathcal{K}_{\Rt}(G)}(0) \xrightarrow{\sim} \Det_{\mathcal{K}_{\Rt}(G)}\left(\mathcal{K}_{\Rt}(G) \otimes_{\Lambda_{R}(G)} R\Gamma_{\Iw}(K_\infty, M)\right) \cdot \Det_{\mathcal{K}_{\Rt}(G)}\left(\mathcal{K}_{\Rt}(G) \otimes_{R^{\rig}} \Dcris(M^\rig)\right)\]
   given by
   \[ \Theta_{\mathcal{K}_{\Rt}(G), \xi}(M) = \ell(M)^{-1} \Det\left(\cL^G_{M, \xi}\right)\]
   where $\ell(M) \in \cH_{\Qp}(G)$ is  defined below.
  \end{definition}

  \begin{definition}\label{def:ellV}
   \mbox{~}
   \begin{enumerate}
    \item Let $\gamma\in\Gamma$ be any non-torsion element.  Then we define
\[	\ell_j:=\frac{\log \gamma}{\log\chi(\gamma)}-j	\]

    \item For $n \in \ZZ$, define the element $\mu_n \in \operatorname{Frac} \cH_{\Qp}(\Gamma)$ by
    \[ \mu_n =
     \begin{cases}
      \ell_0 \cdots \ell_{n-1} & \text{ if $n \ge 1$}\\
      1 & \text{if $n = 0$}\\
      \left(\ell_{-1} \cdots \ell_{n}\right)^{-1} & \text{if $n \le -1$.}
     \end{cases}
    \]
    \item For $V$ a Hodge--Tate representation of $G_{\Qp}$, with Hodge--Tate weights $n_1, \dots, n_d$, let
    \[ \ell(V) = \prod_{i = 1}^d \mu_{n_i}.\]
   \end{enumerate}
      \end{definition}

The $p$-adic Hodge type of $M$ is locally constant on $\Spf(R)^{\rig}$.  Since we assume that $R[1/p]$ is integral, the Hodge--Tate weights are constant.  If the Hodge--Tate weights are $n_1,\ldots, n_d$, we set $\ell(M):=\prod_{i = 1}^d \mu_{n_i}$.  This is an element of $\cH_{\Qp}(G)\subset \cH_{\widetilde R}(G)$.

Since $\Dcris(M^{\rig})=R^{\rig} \otimes_{R}   \Dcris(M)$ by Corollary~\ref{DcrisMandMrig}, we can rewrite the above isomorphism as
\[	\Det_{\mathcal{K}_{\Rt}(G)}(0)\xrightarrow{\sim} \mathcal{K}_{\Rt}(G)\otimes_{\Lambda_{R}(G)}\left\{\Det_{\Lambda_{R}(G)}\left( R\Gamma_{\Iw}(K_\infty, M)\right)\cdot \Det_{\Lambda_{R}(G)}\left(\Lambda_{R}(G) \otimes_{R} \Dcris(M)\right)\right\}	\]


  \subsection{Definition of $\Theta$ over $\Lambda_R(G)[1/p]$}
  \label{sect:alt-const}

We wish to show that the isomorphism $\Theta_{\cK_{\Rt}(G),\xi}(M)$ can be defined over $\Lambda_{\Rt,\Qp}(G)$.  More precisely, we want to show the following:
\begin{proposition}\label{prop:theta1/p}
 There is a trivialization of line bundles
\[	\Theta_{\Lambda_{R,\Qp}(G),\xi}:\Det_{\Lambda_{\widetilde R,\Qp}(G)}(0)\xrightarrow{\sim} \Det_{\Lambda_{\Rt,\Qp}(G)}\left(R\Gamma_{\Iw}(K_\infty,M)[1/p]\right)\cdot\Det_{\Lambda_{\Rt,\Qp}(G)}\left(\Lambda_{\Rt,\Qp}(G)\otimes_{R[1/p]}\Dcris(M)[1/p]\right)	\]
which agrees with $\Theta_{\cK_{\Rt}(G),\xi}(M)$ after extending scalars.  Moreover,   for every specialization $R[1/p]\rightarrow L$, the base change $\Theta_{\Lambda_{R,\Qp}(G),\xi}\otimes_{R[1/p]}L$ agrees with the trivialization
\[	\Theta_{\Lambda_L(G),\xi}:\Det_{\Lambda_{\Lt}(G)}(0)\rightarrow \Det_{\Lambda_{\Lt}(G)}\left(\Lt\otimes_LR\Gamma_{\Iw}(K_\infty,M_L)\right) \cdot \Det_{\Lambda_{\Lt}(G)}\left(\Lambda_{\Lt}(G)\otimes_L\Dcris(M_L\right)	\]
defined in~\cite[Theorem 4.2.1]{loeffler-venjakob-zerbes}.
\end{proposition}

We first recall that $\Theta_{\cK_{\Rt}(G),\xi}(M)$ is defined via base extension from an isomorphism over the total ring of fractions of $\cH_{\Qp}(G)_{\tau_p}\htimes R^{\rig}$.  We therefore descend the latter isomorphism to an isomorphism over $\Lambda_{\Qp}(G)_{\tau_p}\htimes R[1/p]$ (using the fact that $\Lambda_{\Qp}(G)_{\tau_p}$ is free of rank $1$ over $\Lambda_{\Qp}(G)$).

To construct a homomorphism of line bundles
\[	\Det_{\Lambda_{\Rt,\Qp}(G)}R\Gamma(K_\infty,M)[1/p]^{-1}\rightarrow \Det_{\Lambda_{\Rt,\Qp}(G)}\left(\Lambda_{\Rt,\Qp}(G)\otimes_{R[1/p]}\Dcris(M)[1/p]\right)	\]
it suffices to construct homomorphisms on an affine cover and check that they agree on overlaps.  Thus, our strategy is to work locally on $\Spec \Lambda_{\Rt,\Qp}(G)$ (to trivialize $\Det_{\Lambda_{\Rt,\Qp}(G)}\left(R\Gamma(K_\infty,M)[1/p]\right)$ and $\Det_{\Lambda_{\Rt,\Qp}(G)}\left(\Lambda_{\Rt,\Qp}(G)\otimes_{R[1/p]}\Dcris(M)[1/p]\right)$), choose integral generators on both sides, and calculate that the determinant actually lands in $\Lambda_{\Rt,\Qp}(G)$.

However, we will have to be a bit careful, for two reasons: First, our previous computations concerned $\Det_{\Lambda_{\Rt,\Qp}(G)}\H_{\Iw}^1(K_\infty,M)[1/p]$, not $\Det_{\Lambda_{\Rt,\Qp}(G)}\left(R\Gamma_{\Iw}(K_\infty,M)[1/p]\right)^{-1}$, and they are only canonically isomorphic when $\H_{\Iw}^2(K_\infty,M)[1/p]=0$.  Second, we can only compute with $\Det_{\Lambda_{\Rt,\Qp}(G)}\H_{\Iw}^1(K_\infty,M)$, and it is only compatible with base change on $R$ when $\H_{\Iw}^2(K_\infty,M)=0$.  Thus, we will first work away from the support of $\H_{\Iw}^2(K_\infty,M)$ and then use the Cohen--Macaulay-ness of $R$ to extend over this locus.

First of all, $\det\cL_{M,\xi}^G$ is defined by a homomorphism between modules over $\mathcal{H}_{\Rt}(G)$.  If we can show that the valuation of $\ell(V)^{-1}\det\cL_{M,\xi}^G$ is bounded (with respect to suitable choices of local generators), then normality of $\Lambda_{\Rt}(G)$ will imply that it actually lives in $\Lambda_{\Rt,\Qp}(G)$, as desired.  We therefore study specializations of $\ell(V)^{-1}\det\cL_{M,\xi}^G$.

We first work over open affine subspaces $U=\Spec A_U\subset \Spec \Lambda_{\Rt}(G)$ where $\H_{\Iw}^2(K_\infty,M)$ vanishes; on such subspaces, $\Det_{A_U}R\Gamma_{\Iw}(K_\infty,M)\otimes_RA_U\cong\Det_{A_U}\left(\H_{\Iw}^1(K_\infty,M)\right)^{-1}\otimes_RA_U$ by property \ref{B.h)} in Appendix \ref{determinants}.

Over these subspaces, we will consider specializations of the determinant of $\cL^G_{M, \xi} $ along maps $R[\frac{1}{p}]\to L$, where $L/\Q_p$ is a finite field extension.  A priori, such base changes involve derived tensor products and thus higher $\Tor$ groups.  However, Theorem~\ref{thm:specialization} (2) together with the spectral sequence \eqref{f:spectralsequence} (for $Y=L$) shows that when $\H_{\Iw}^2=0$, these $\Tor$ groups vanish.

In general, we will use \eqref{H1modulopseudonull} to check that no higher $\Tor$-groups are involved in the descent calculation. Indeed, for a Cohen-Macaulay ring pseudo-null modules possess a canonical trivialisation by the same proof as for \cite[Lemma 2.2]{venjakob11}. From this comment and by construction below, it will be clear that --- once $ \Theta_{\Lambda_{R,\Qp}(G),\xi}$ exists --- it satisfies the desired compatibility with specialisations.


\begin{lemma}\label{lemma:descent-affine-hiw2-0}
Let $U=\Spec A_U\subset \Spec \Lambda_{R}(G)$ be an open affine subspace where $\H_{\Iw}^2(K_\infty,M)=0$ and $\Det_{\Lambda_{R}(G)}\left(\H_{\Iw}^1(K_\infty,M)\right)$ and $\Det_{\Lambda_{R}(G)}\left(\Lambda_{R}(G)\otimes_{R}\Dcris(M)\right)$ are trivial.  Then $\ell(V)^{-1}\det\cL_{M,\xi}^G$ can be defined over the pre-image of $U$ in $\Spec\Lambda_{\Rt}(G)$.
\end{lemma}
\begin{proof}
Without loss of generality, we may write $A_U=\Lambda_R(G)[1/f_U]$ for some $f_U\in\Lambda_R(G)$.  Choosing generators of $\Det_{\Lambda_{\Rt}(G)}\left(\H_{\Iw}^1(K_\infty,M)\right)|_U$ and $\Det_{\Lambda_{\Rt}(G)}\left(\Lambda_{\Rt}(G)\otimes_{R}\Dcris(M)\right)|_U$, we may view $\ell(V)^{-1}\det\cL_{M,\xi}^G$ as multiplication by an element $a_U:=a'_U/f_U^k$, where $a'_U\in \cH_{\Rt}(G)$.  In fact, since $\cL_{M,\xi}^G$ is the base extension of a morphism of $\cH_R(G)$-modules
\[	\cH_R(G)\otimes_{\Lambda_R(G)}\H_{\Iw}^1(K_\infty,M)\rightarrow (\cH_{\Qp}(G)_{\tau_p}\htimes R^{\rig})\otimes_{R^{\rig}}\Dcris(M^{\rig})	\]
of the same rank, we may assume that $a'_U\in\cH_R(G)$.

The formation of $\H_{\Iw}^1(K_\infty,M)$ commutes with specialization on $U$, since $\H_{\Iw}^2(K_\infty,M)|_U=0$ and $R\Gamma_{\Iw}(K_\infty,M) $ does.  Let $L/\Qp$ be a finite extension with ring of integers $\cO_L$, and let $x:R[1/p]\rightarrow L$ be a closed point of $\Spec R[1/p]$ such that $A_U[1/p]\otimes_{R[1/p]}L$ is non-zero.  Then the natural homomorphism $R\rightarrow \cO_L$ induces a map $\Dcris(M)\rightarrow \Dcris(M\otimes_R\cO_L)$ with cokernel annihilated by $p^{\alpha(h)}$.  Therefore, there is a natural homomorphism $\Det_R\Dcris(M)\rightarrow \Det_{\cO_L}\Dcris(M\otimes_R\cO_L)$ with cokernel annihilated by $p^{d\alpha(h)}$.  Now~\cite[Corollary 4.3.8]{loeffler-venjakob-zerbes} implies that the specialization of $\ell(V)^{-1}\det\cL_{M,\xi}^G$ at $x$ is given by multiplication by an element of $\Lambda_{\cOt_L}(G)$ with respect to generators of $\H_{\Iw}^1(K_\infty,M\otimes_R\cO_L)$ and $\Dcris(M\otimes_R\cO_L)$.

This implies that for every such $x$, $p^{d\alpha(h)}f_U(x)^ka_U(x)\in \Lambda_{\cOt_L}(G)$.  But if $x$ is a closed point of $\Spec R[1/p]$ where $A_U[1/p]\otimes_{R[1/p]}L$ is zero, then $f_U(x)=0$ and $p^{d\alpha(h)}f_U(x)^ka_U(x)\in \Lambda_{\cOt_L}(G)$.  Thus, for every point $x'\in\Spf(\Lambda_{R}(G))^{\rig}$,
\[	\lvert p^{d\alpha(h)}a'_U(x')\rvert = \lvert p^{d\alpha(h)}f_U(x')^ka_U(x')\rvert\leq 1	\]
Now by Lemma~\ref{lemma:qpbar-spec-bdd} (and exhausting $\Spf(R)^{\rig}$ by affinoid subdomains)
\[	\lvert p^{d\alpha(h)}a'_U(x')\rvert = \lvert p^{d\alpha(h)}f_U(x')^ka_U(x')\rvert\leq 1	\]
for all points $x\in\Spf(\Lambda_{\Rt}(G))^{\rig}$.
Since $R$ is assumed normal, Proposition~\cite[Proposition 7.3.6]{dejong} implies that $a'_U\in \Lambda_{\Rt,\Qp}(G)$.

It remains to check that $\ell(V)^{-1}\det\cL_{M,\xi}^G|_{U[1/p]}$ is an isomorphism, i.e., that $a_U$ has no zeroes.  But by construction, $\ell(V)^{-1}\det\cL_{M,\xi}^G|_{U[1/p]}$ is compatible with specialization on $R[1/p]$, so comparison with the classical case shows that it cannot vanish.
\end{proof}

If $U,U'\subset\Spec\Lambda_{\Rt}(G)$ are open affine subspaces as above, then $\ell(V)^{-1}\det\cL_{M,\xi}^G|_U$ and $\ell(V)^{-1}\det\cL_{M,\xi}^G|_{U'}$ agree on $U\cap U'$, because they agree after extending scalars to $\cK_{\Rt}(G)$.

\begin{lemma}\label{lemma:affine-descent}
Let $U=\Spec A_U\subset\Spec \Lambda_{R}(G)$ be an open affine subspace where $\Det_{\Lambda_{R}(G)}\left(R\Gamma_{\Iw}(K_\infty,M)\right)$ and $\Det_{\Lambda_{R}(G)}\left(\Lambda_{R}(G)\otimes_{R}\Dcris(M)\right)$ are trivial.  Then $\ell(V)^{-1}\det\cL_{M,\xi}^G$ can be defined over the pre-image $\widetilde U=\Spec A_{\widetilde U}$ of $U$ in $\Spec \Lambda_{\Rt}(G)$.
\end{lemma}
\begin{proof}
Choose bases of $\Det_{\Lambda_{R}(G)}\left(R\Gamma_{\Iw}(K_\infty,M)\right)^{-1}$ and $\Det_{\Lambda_{R}(G)}\left(\Lambda_{R}(G)\otimes_{R}\Dcris(M)\right)$ as in the proof of the previous lemma; an isomorphism
\[	\Det_{A_{\widetilde U}[1/p]}\left(A_{\widetilde U}[1/p]\otimes_{\Lambda_{R}(G)}R\Gamma_{\Iw}(K_\infty,M)\right)^{-1}\xrightarrow{\sim} \Det_{A_{\widetilde U}[1/p]}\left(A_{\widetilde U}[1/p]\otimes_{R}\Dcris(M)\right)	\]
corresponds to multiplication by a unit of $A_{\widetilde U}[1/p]$.  The support of $\H_{\Iw}^2(K_\infty,M)$ is a closed subscheme $V\subset U$ of codimension at least $2$; for each open affine subspace $U'=\Spec A_{U'}\subset U\smallsetminus V$ we obtain a unit $a_{U'}\in A_{U'}[1/p]$ by Lemma~\ref{lemma:descent-affine-hiw2-0} and these units agree on overlaps.  Therefore, the $a_{U'}$ glue to a section $a_U'$ on all of $(U\smallsetminus V)[1/p]$.  But $\Spec \Lambda_{\Rt}(G)$ is Cohen--Macaulay, so $a_U'$ extends to a section $a_U$ over all of $U[1/p]$.

It remains to check that $a_U\in A_U[1/p]^\times$.  Let $L/\Qp$ be a finite extension with ring of integers $\cO_L$, and let $x:R[1/p]\rightarrow L$ be a closed point of $\Spec R[1/p]$ such that $A_U[1/p]\otimes_{R[1/p]}L$ is non-zero.  Then we know that the image of $a_U$ in $A_U[1/p]\otimes_{R[1/p]}L$ extends to a unit of $\Lambda_{\Lt}(G)$ by the classical case, since the formation of $\Det_{\Lambda_R(G)}\left(R\Gamma_{\Iw}(K_\infty,M)\right)^{-1}$ commutes with specialization on $R$ and $\Dcris(M)[1/p]$ commutes with specialization on $R[1/p]$.  Thus the claim follows from Lemma \ref{lemma:specialization}.
\end{proof}

It is again clear that if $U,U'\subset \Spec \Lambda_{\Rt}(G)$ are open affine subspaces as above, then $\ell(V)^{-1}\det\cL_{M,\xi}^G|_U$ and $\ell(V)^{-1}\det\cL_{M,\xi}^G|_{U'}$ agree on $U\cap U'$. This proves Proposition \ref{prop:theta1/p}.

\begin{remark}
Although the proofs of Lemma~\ref{lemma:descent-affine-hiw2-0} and Lemma~\ref{lemma:affine-descent} construct isomorphisms over affine subspaces of $\Lambda_R(G)[1/p]$, we only claim to have descended $\ell(M)^{-1}\cL_{M,\xi}^G$ to $\Lambda_{\Rt}(G)[1/p]$.  This is because the proofs implicitly involved making a non-canonical choice of a generator of $\Lambda_{\Z_p}(G)_{\tau_p}$ over $\Lambda_{\Z_p}(G)$.
\end{remark}

\subsection{Definition of the epsilon-isomorphism}


Recall from~\cite[\textsection 2.3,\textsection 2.4]{loeffler-venjakob-zerbes} that if $V$ is a $d$-dimensional $L$-linear crystalline representation of $G_{\Qp}$, then the $\varepsilon$-factor of $\D_{\pst}(V)$ is $1$.  Further, multiplication by $t^{-m(V)}\varepsilon(\D_{\pst}(V),\xi)$ defines an isomorphism $\widetilde{L}\otimes_L\Det_L(\D_{\dR}(V))\rightarrow \widetilde{L}\otimes_L\Det_LV$, where $m(V)=r_1+\ldots +r_d$ is the sum of the opposites of the Hodge--Tate weights.  This multiplication takes place inside the canonical isomorphism $\B_{\dR}\otimes_{\Qp}\Det_L\D_{\dR}(V)\xrightarrow{\sim}\B_{\dR}\otimes_{\Qp}\Det_LV$.

If $M$ is a crystalline family of Galois representations over $R$, then $\D_{\dR}(M^{\rig})$ and $\D_{\cris}(M^{\rig})$ are well-defined and the Hodge--Tate weights are locally constant on $\Spf(R)^{\rig}$.

  \begin{theorem}\label{thm:Dcris-M}
   Let $M$ be a crystalline family of  Galois representations over $R$. Then there is a unique isomorphism
   \[ \varepsilon_{\Rt,\Qp,\xi, \dR}(M): \Det_{\Rt[1/p]}(\Rt[1/p]\otimes_R \Dcris(M)) \xrightarrow{\sim} \Det_{\Rt[1/p]}(\Rt[1/p]\otimes_{R} M) \]
   whose image under specialization $x:R\rightarrow \cO_L$ is the isomorphism defined above.
  \end{theorem}
\begin{remark}
We cannot hope to construct $\varepsilon_{\Rt,\xi, \dR}(M)$ such that its \emph{integral} specializations agree with the isomorphisms $\varepsilon_{\cO_L,\xi, \dR}(M)$ constructed in~\cite{loeffler-venjakob-zerbes}, because $M$ commutes with base change on $R$ but $\Dcris(M)$ does not.
\end{remark}

  \begin{proof}
   By \cite[thm.\ 1.1.14]{bellovin13} and Lemma \ref{lemma:rig} there is a canonical isomorphism
\[\mathrm{can}_M: ( R^{\rig}\htimes_{\mathbb{Q}_p}\Bmax)\otimes_{ R^{\rig}} \Dcris(M^{\rig})\cong ( R^{\rig}\htimes_{\mathbb{Q}_p}\Bmax)\otimes_{ R^{\rig}} \Gamma(\Spf(R)^{\rig},M^{\rig})=( R^{\rig}\htimes_{\mathbb{Q}_p}\Bmax)\otimes_{R}M.\]

Both sides have natural $R$-linear structures, given by $\Dcris(M)$ and $M$, respectively.  We wish to show that the determinant of $\mathrm{can}_M$ with respect to these $R$-modules is an element of $t^{m(M)}\Rt[1/p]^\times$.  For this, we work locally on $\Spec R[1/p]$ to trivialize $M[1/p]$ and $\Dcris(M)[1/p]$ and show that for any $R[1/p]$-bases of $M[1/p]$ and of $\Dcris(M)[1/p]$, the matrix of $\mathrm{can}_M$ has determinant in $t^{m(M)}\Rt[1/p]^\times$ (viewed as a subgroup of $(R^\rig\htimes\Bmax)^\times$).

In order to show this we follow the argument of~\cite[Proposition V.1.2]{berger04}.  After twisting, we may assume that $M$ is positive.


By Corollary~\ref{cor:det-dcris-nm-bdd} and Proposition~\ref{partial-localize-units}, locally on $\Spec R[1/p]$ the determinant of the (injective) map $(\B_{\rig,\Qp}^+\widehat\otimes  R^{\rig})\otimes_{ R^{\rig}}\D_{\cris}(M^{\rig})\rightarrow \NN(M)^{\rig}$ (with respect to any choice of bases) is an element of $(t/\pi)^{m(M)}(R[1/p]\htimes\B_{\rig,\Qp}^+)^{\times}=(t/\pi)^{m(M)}(R\htimes\A_{\Qp}^+)[1/p]^{\times}$.

Next, we have an inclusion $(R\widehat\otimes\A^+)\otimes_{R\widehat\otimes\A_{\Qp}^+}\NN(M)\subset (R\widehat\otimes\A^+)\otimes_RM$.  Then Corollary~\ref{cor:det-wach-rep} tells us that after inverting $p$, locally on $\Spec R[1/p]$ the determinant of the inclusion is the ideal $(\pi^{m(M)})\subset (R\widehat\otimes\A^+)[1/p]$.

Extending scalars to $R^{\rig}\htimes\B_{\max}$ and composing the two maps, we see that the determinant of the canonical isomorphism $\operatorname{can}_M$ is an element of $t^{m(M)}(R\htimes\AA^+)[1/p]^\times$.  We wish to show that it actually lives in $t^{m(M)}\Rt[1/p]^\times$.

Since the formation of $\D_{\cris}(M^{\rig})$ is compatible with taking exterior powers, we may replace $M$ with $\det(M)$ and consider only families of characters over $R$.  Furthermore, by twisting, we may assume that $M$ is unramified (so has Hodge--Tate weight $0$), so that the determinant of $\mathrm{can}_M$ is $\varphi$-invariant up to a unit of $R$, and we may assume $\det(\mathrm{can}_M))\in R\htimes\A^+$.  We consider the equation $\varphi(\det(\mathrm{can}_M))=r\cdot \det(\mathrm{can}_M)$ modulo successive powers of $\mathfrak{m}_R$.  For each $n$, there is some $s$ such that $\varphi^s(\det(\mathrm{can}_M))=\det(\mathrm{can}_M)$ modulo $\mathfrak{m}_R^n$; viewing this as an equation inside $(R/\mathfrak{m}_R^n)\otimes_{\Z_p}\widetilde{\A}^+$ implies that $\det(\mathrm{can}_M)\in (R/\mathfrak{m}_R^n)\htimes W(\FF_{p^s})$.  Passing to the limit yields the desired result.
\end{proof}

  \begin{definition}\label{def:epsilon}
   We define
   \[
    \varepsilon_{\Lambda_{R,\Qp}(G), \xi}(M): \Det_{\Lambda_{\Rt,\Qp}(G)}(0) \xrightarrow{\sim}
    \left[ \Det_{\Lambda_{\Rt,\Qp}(G)} \left(\Lambda_{\Rt,\Qp}(G)\otimes_{\Lambda_R(G)} R\Gamma_{\Iw}(K_\infty, M)\right)\right]
    \left[ \Det_{\Lambda_{\Rt,\Qp}(G)} \left(\Lambda_{\Rt,\Qp}(G) \otimes_R M\right)\right]
   \]
   to be the isomorphism given by
   \[ (-\gamma_{-1})^d (-1)^{m(M)} \cdot \Theta_{\Lambda_{\Rt,\Qp}(G), \xi}(M) \cdot \varepsilon_{\Rt, \Qp,\xi, \dR}(M).\]
  \end{definition}
Here we regard $\varepsilon_{\Rt, \Qp,\xi, \dR}(M)$ as an isomorphism
   \[ \Det_{\Lambda_{\Rt,\Qp}(G)}(\Lambda_{\Rt,\Qp}(G) \otimes_{R} \Dcris(M)) \xrightarrow{\sim} \Det_{\Lambda_{\Rt,\Qp}(G)}(\Lambda_{\Rt,\Qp}(G) \otimes_{R} M)\]
   via base-extension.  Recall also that $\chi(\gamma_{-1})=-1\in\Z_p^\times$.

\begin{theorem}\label{thm:conclusion}
There is an isomorphism of line bundles
\[	\varepsilon_{\Lambda_{R}(G), \xi}(M):\Det_{\Lambda_{\Rt}(G)}(0) \xrightarrow{\sim}
    \left[ \Det_{\Lambda_{\Rt}(G)} \left(\Lambda_{\Rt}(G)\otimes_{\Lambda_R(G)} R\Gamma_{\Iw}(K_\infty, M)\right)\right]
    \left[ \Det_{\Lambda_{\Rt}(G)} \left(\Lambda_{\Rt}(G) \otimes_R M\right)\right]	\]
such that $\varepsilon_{\Lambda_{R}(G), \xi}(M)$ agrees with
$\varepsilon_{\Lambda_{R,\Qp}(G), \xi}(M)$ after inverting $p$.  The specializations $\varepsilon_{\Lambda_R(G),\xi}(M)\otimes_R\cO_L$ agree with the isomorphisms $\varepsilon_{\Lambda_{\cO_L}(G), \xi}(M\otimes_R\cO_L)$ of~\cite{loeffler-venjakob-zerbes} by construction.
\end{theorem}
\begin{proof}
There is some integer $m$ such that $p^m\varepsilon_{\Lambda_{R,\Qp}(G), \xi}(M)$ is a homomorphism of line bundles over $\Lambda_{\Rt}(G)$.  But the construction of $\varepsilon_{\Lambda_{R,\Qp}(G), \xi}(M)$ is compatible with base change $R[1/p]\rightarrow L$, and $\varepsilon_{\Lambda_{R,\Qp}(G), \xi}(M)\otimes_R\cO_L$ is an isomorphism of line bundles over $\Lambda_{\cOt_L}(G)$ by~\cite[Corollary 4.3.8]{loeffler-venjakob-zerbes}.  This implies that $m=0$ and $\varepsilon_{\Lambda_{R,\Qp}(G), \xi}(M)$ is an isomorphism of line bundles over $\Lambda_{\Rt}(G)$.
\end{proof}

\subsection{Properties of the epsilon-isomorphism} \label{properties}

  In the following we are going to apply a principle of specialization based on:

\begin{lemma} \label{lemma:specialization}
The canonical map
 \begin{equation*}
        K_1(\Lambda_{\widetilde{R}})\to \prod_{R\to \cO} K_1(\Lambda_{\widetilde{\cO}}) ,
 \end{equation*}
where $R\to \cO$ runs through all $\Zp$-algebra homomorphisms with $\cO$ the valuation ring of any finite extension of $\Qp$ (and such that the image has the same quotient field as $\cO$) is injective.
\end{lemma}
 Cp.\ also Corollary \ref{cor:injectivityOfSpecializations} below for a similar statement.

\begin{proof}
Since $SK_1(\Lambda_{\widetilde{R}})= SK_1(\Lambda_{\widetilde{\cO}}) =1$ the statement is equivalent to the injectivity of
\begin{equation*}
     \Lambda_{\widetilde{R}}^\times\to \prod_{R\to \cO} \Lambda_{\widetilde{\cO}}^\times,
\end{equation*}
or - because $\bigcap_{R\to \cO}\left( 1+ \ker\left(\Lambda_{\widetilde{R}} \to   \Lambda_{\widetilde{\cO}} \right)\right) =1+ \bigcap_{R\to \cO} \ker\left(\Lambda_{\widetilde{R}} \to   \Lambda_{\widetilde{\cO}} \right) $ - to the injectivity of
 \begin{equation}
     \label{f:specializationadd}
       \Lambda_{\widetilde{R}} \to \prod_{R\to \cO} \Lambda_{\widetilde{\cO}}.
\end{equation}
Since $\Lambda_{\widetilde{R}}$ is a product of rings of the type of $R$ dealt with in Lemma \ref{lemma:qpbar-spec-inj} and as $\widetilde{\Lambda_{R}}=\Lambda_{\widetilde{R}},$   we obtain an injective map  $\Lambda_{\widetilde{R}}\hookrightarrow\prod_x \widehat{\Qp^{\nr}}\htimes_{\Qp} K_x$, where the closed points $x$ corresponds to $\Zp$-algebra homomorphisms $\Lambda_{R}\to K_x$ (such that the image has   quotient field $K_x$). One easily sees that this map factorises as
\[
\xymatrix{
                & {\prod_{R\to \cO} \Lambda_{\widetilde{\cO}}} \ar[dr]^{ }             \\
{\Lambda_{\widetilde{R}}} \ar[ur]^{ } \ar@{^(->}[rr]^{ } & &    { \prod_x \widehat{\Qp^{\nr}}\htimes_{\Qp} K_x , }      }
\]
because each $\Zp$-algebra homomorphisms $\Lambda_{R}\to K_x$ factorizes as $\Lambda_{R}\to \Lambda_{\cO_x}\to K_x$ with the first map being induced on the coefficients by a $\Zp$-algebra homomorphisms $R\to \cO_x$ where $\cO_x$ is the valuation ring of $K_x$. The injectivity of \eqref{f:specializationadd}, whence the claim, follows.
\end{proof}

%
%
%
%
%

  We note for later use some properties of the isomorphisms 
  $\varepsilon_{\Lambda_{R}(G), \xi}(M)$:\\


  \begin{proposition}[Compatibility with short exact sequences]
   \label{prop:regulatorSEScompatible}
   Let
   \[ 0 \rTo M' \rTo M \rTo M''\rTo 0\]
   be a short exact sequence of $R$-linear crystalline representations of $G_{\Qp}$. Then
   \[\varepsilon_{\Lambda_{R}(G), \xi}(M) = \varepsilon_{\Lambda_{R}(G), \xi}(M')\cdot \varepsilon_{\Lambda_{R}(G), \xi}(M'').\]
  \end{proposition}

  \begin{proof}
We first observe that the determinants
\[	\left[ \Det_{\Lambda_{R,\Qp}(G)} \left(R\Gamma_{\Iw}(K_\infty, M)\right)\right] \cdot \left[ \Det_{\Lambda_{R,\Qp}(G)} \left(R\Gamma_{\Iw}(K_\infty, M')\right)\right]^{-1}	\cdot\left[ \Det_{\Lambda_{R,\Qp}(G)} \left(R\Gamma_{\Iw}(K_\infty, M'')\right)\right]^{-1}	\]
and
\[	\left[ \Det_{\Lambda_{\Rt,\Qp}(G)} \left(\Lambda_{\Rt,\Qp}(G) \otimes_R M\right)\right] \cdot \left[ \Det_{\Lambda_{\Rt,\Qp}(G)} \left(\Lambda_{\Rt,\Qp}(G) \otimes_R M'\right)\right]^{-1} \cdot\left[\Det_{\Lambda_{\Rt,\Qp}(G)} \left(\Lambda_{\Rt,\Qp}(G) \otimes_R M''\right)\right]^{-1} \cdot	\]
have canonical trivializations, so their product (with scalars extended to $\Lambda_{\Rt,\Qp}(G)$) does, as well.

On the other hand, the product $\varepsilon_{\Lambda_R(G),\xi}(M)\cdot \varepsilon_{\Lambda_R(G),\xi}(M')^{-1}\cdot\varepsilon_{\Lambda_R(G),\xi}(M'')^{-1}$ provides another trivialization of the product of these determinants.  But since
\[	\varepsilon_{\Lambda_R(G),\xi}(\cO\otimes_RM)=\varepsilon_{\Lambda_R(G),\xi}(\cO\otimes_RM')\varepsilon_{\Lambda_R(G),\xi}(\cO\otimes_RM'')	\]
for every map $R\rightarrow \cO$ where $\cO$ is finite flat over $\Z_p$, the proposition follows.
  \end{proof}

  \begin{proposition}[Change of coefficient ring]
   \label{prop:regulatorbaseextension}
   Let $R'$ be an $R$-algebra which is again a complete local noetherian $\Z_p$-algebra which is Cohen--Macaulay, normal, and $\Z_p$-flat with finite residue field. Then
   \[ \varepsilon_{\Lambda_{R'}(G), \xi}(R' \otimes_R M) = R' \otimes_R \varepsilon_{\Lambda_{R'}(G), \xi}(M).\]
  \end{proposition}

  \begin{proof}
Clear.
  \end{proof}

  The next compatibility property takes a little more notation to state. For brevity let us write $\Lambda$ for $\Lambda_{R}(G)$. For $\eta$ a continuous $R$-valued   character of $G$, we have a twisting homomorphism $\tw_\eta: \Lambda \to \Lambda$ which maps a group element $g \in G$ to $\eta(g) g$. Hence we obtain a pullback functor $(\tw_\eta)^*$ from the category of $\Lambda$-modules to itself:
  \[ (\tw_\eta)^* \mathfrak{M} := \Lambda \otimes_{\Lambda, \tw_\eta} \mathfrak{M}.\]
  This can also be described in terms of tensoring with the $\Lambda$-bimodule $\Lambda \otimes_{\Lambda, \tw_\eta} \Lambda$, which is free of rank one as a $\Lambda$-module.  Hence the twisting functor extends to a functor from the category $\underline{\Det}(\Lambda)$ to itself, and is compatible with the functor $\Det$.

  Note that we have an isomorphism
  \[ (\tw_\eta)^* (\Lambda \otimes_{R} M) \cong \Lambda \otimes_{R} M(\eta^{-1}),
   \; a\otimes b\otimes v\mapsto a \operatorname{Tw}_{\eta}(b)\otimes (v \otimes m_{\eta^{-1}})\]
  as $(\Lambda, G_{\Qp})$-modules, if $\Lambda$ acts on $\Lambda \otimes_{\cO} T$ via left multiplication on the left factor, while $g \in G_{\Qp}$ sends $\lambda \otimes v$ to $\lambda \bar{g}^{-1} \otimes gv$ where $\bar{g}$ denotes the image of $g$ in $G$ (and analogously for the action on $\Lambda \otimes_{R} M(\eta^{-1})$). Here $m_{\eta^{-1}} $ a (fixed) basis of  $M(\eta^{-1}). $

  We clearly have $(\tw_\eta)^* \circ (\tw_{\eta^{-1}})^* = \id$. Similar definitions apply to other coefficient rings than $\Lambda_{\cO}(G)$, including $\Lambda_L(G)$, $\cH_L(G)$ or $\Lambda_{\Rt}(G)$.

  Finally note that for a $\Lambda$-module $\mathfrak{M}$ we have a canonical isomorphism
  \[\Lambda \otimes_{\Lambda, \operatorname{Tw}_{\eta^{-1}} } \mathfrak{M} = \mathfrak{M}\otimes_R R m_\eta,\; \lambda\otimes m\mapsto \tw_\eta(\lambda) m \otimes m_\eta,\]
  of $\Lambda$-modules, where the $\Lambda$-module structure on the right hand side is induced by the diagonal action of $G$ upon it.

  \begin{proposition}[Invariance under crystalline twists]\label{prop:twistinvariance}
   If $\eta$ is a crystalline character with values in $R$, then
   \[\operatorname{Tw}_{\eta^{-1}}\left(\varepsilon_{\Lambda_{R}(G), \xi}(M)\right)\cong \varepsilon_{\Lambda_{R}(G), \xi}(M(\eta))\]
  \end{proposition}

  \begin{proof}
This again follows by noting that $\operatorname{Tw}_{\eta^{-1}}\left(\varepsilon_{\Lambda_{R}(G), \xi}(M)\right)$ and $\varepsilon_{\Lambda_{R}(G), \xi}(M(\eta))$ both provide trivializations of
\[	\left[ \Det_{\Lambda_{\Rt,\Qp}(G)} \left(\Lambda_{\Rt,\Qp}(G)\otimes_{\Lambda_R(G)} R\Gamma_{\Iw}(K_\infty, M(\eta))\right)\right]
    \left[ \Det_{\Lambda_{\Rt,\Qp}(G)} \left(\Lambda_{\Rt,\Qp}(G) \otimes_R M(\eta)\right)\right]	\]
Since they agree after extending scalars along $R\rightarrow\cO$ when $\cO$ is finite flat over $\Z_p$, they are the same.
  \end{proof}

%
%
We use these results to extend the definition of $\varepsilon_{\Lambda_{R}(G), \xi}(M)$   to lattices $M$ a crystalline family of Galois representations with arbitrary (bounded) Hodge-Tate weights, by tensoring the corresponding maps for $M(j)$ with $R(-j)$, where $j \gg 0$ is such that $M(j)$ has non-negative Hodge--Tate weights. Clearly, $\varepsilon_{\Lambda_{R}(G), \xi}(M)$ is again compatible with specializations in $R.$

\subsection{Epsilon-isomorphisms for more general modules}

As before we let $L$ denote   a finite extension of $\Qp.$ We recall the following definition from \cite[\S1.4]{fukayakato06}.

  \begin{definition}
   A ring $A$ is of
   \begin{description}
     \item[(type 1)] if there exists a two sided ideal $I$ of $A$ such that $A/I^n$ is finite of order a power of
     $p$ for any $n \geq 1$, and such that $A \cong \varprojlim_n A/I^n$.
     \item[(type 2)] if $A$ is the matrix-algebra $M_n(E)$ of some finite extension $E$ over $\Qp$ and some dimension $n\geq 1$.
   \end{description}
  \end{definition}

  By Lemma 1.4.4 in (loc.\ cit.), $A$ is of type 1 if and only if the defining condition above holds with $I$ equal to the Jacobson ideal $J=J(A)$. Such rings are always semi-local and $A/J$ is a finite product of matrix algebras over finite fields. For a ring $A$ of type (1) or (2) we define
  \[ \At := \widehat{\ZZ_p^{\nr}} \htimes_{\Zp} A\]
  where $\widehat{\mathbf{Z}_p^{\nr}}$ denotes the completion of the ring of integers of the maximal unramified extension of $\Qp$.

  Now let $M$ be a  crystalline family of representation of $G_{\Qp}$ with coefficients in $R$. We set $\TT(M) := \Lambda_{R}(G)\otimes_{R} M$, which we consider as $\Lambda(G)$-module by multiplication on the left tensor factor and as $G_{\Qp}$-module via $g(\lambda\otimes t)=\lambda \bar{g}^{-1}\otimes gt$. The following isomorphism (essentially a version of Shapiro's lemma) is well known:

  \begin{proposition}
   We have
   \[ R\Gamma(\Qp, \TT(M)) \cong R\Gamma_{\Iw}(K_\infty, M)\]
   as $\Lambda(G)$-modules.
  \end{proposition}

  \begin{proof}
   See e.g.~\cite[8.4.4.2 Proposition]{nekovar06}.
  \end{proof}

  Let $\Lambda = \Lambda_{R}(G)$, which is a ring of type 1, with $\tilde\Lambda = \Lambda_{\Rt}(G)$. Then we have constructed an isomorphism
  \[ \varepsilon_{\Lambda, \xi}(M): \Det_{\tilde\Lambda}(0) \rTo^\cong
    \tilde\Lambda \otimes_{\Lambda} \left\{ \Det_{\Lambda} R\Gamma(\Qp, \TT(M)) \cdot \Det_{\Lambda} \TT(M)\right\}.
  \]
  We shall establish that this satisfies the properties predicted by \cite[Conjecture 3.4.3]{fukayakato06} for the module $\TT(M)$. For that purpose it is convenient to write also $\varepsilon_{\Lambda, \xi}(\TT(M))$ for the above $\varepsilon$-isomorphism, and to extend it to a slightly more general class of modules.

  We consider quadruples $(A, Y, M, \xi)$ where
  \begin{itemize}
   \item $A$ is a $p$-torsion-free $R$-algebra which is also a ring of type (1) or (2) above,
   \item $\xi$ is a compatible system of $p^n$-th roots of unity (as before),
   \item $M$ is a crystalline family of representation of $G_{\Qp}$ with coefficients in $R$,
   \item $Y$ is a finitely-generated projective left $A$-module, equipped with a continuous $A$-linear action of $G$.
  \end{itemize}

  Given such a quadruple, we define $\TT = Y \otimes_{R} M$, which we equip with the obvious left $A$-module structure and an action of $G_{\Qp}$ via $g \cdot (y \otimes t) = y\bar{g}^{-1} \otimes g t$. Then $(A, \TT, \xi)$ is a triple satisfying the conditions of \cite[\S 3.4.1]{fukayakato06}. Moreover, the action of $G$ on $Y$ extends to a $\Lambda$-module structure, and we have
  \[ \TT = Y \otimes_{\Lambda} \TT(M)\]
  where $\TT(M)$ is as above. So we may define
  \[ \varepsilon_{A, \xi}(\TT) := Y\otimes_{\Lambda} \varepsilon_{\Lambda, \xi}(\TT(M)),\]
  which is an isomorphism
  \[
    \Det_{\At}(0) \xrightarrow{\sim} \At \otimes_{A} \left\{ \Det_A R\Gamma(\Qp, \TT) \cdot \Det_A \TT \right\};
  \]
  here we have used the fact that
  \[ Y\otimes^{\mathbf{L}}_{\Lambda} R\Gamma(\Qp, \TT(M))\cong R\Gamma(\Qp, Y\otimes_{\Lambda} \TT(M))\]
  by \cite[1.6.5]{fukayakato06}.

  \begin{remark}
   Note that $A$ need not be commutative, and $Y$ need not be either projective or finitely-generated as a $\Lambda$-module.
  \end{remark}

  \begin{proposition}
   \label{prop:compat2}
   Suppose $A = \cO_E$ for some finite extension $E / \Qp$, and the finite-dimensional $E$-vector space $E \otimes_{A} Y$ is de Rham as a representation of $G$. Then $E \otimes_{\cO_E} \TT$ is also de Rham, and $E \otimes_{\cO_E}\varepsilon_{A, \xi}(\TT)$ coincides with the canonical isomorphism $\varepsilon_{E, \xi}(E \otimes_{\cO_E} \TT)$ of \cite[\S 2.4]{loeffler-venjakob-zerbes}, the 'standard trivialisation'.
  \end{proposition}

  \begin{proof}
  We may  assume that $E$ is sufficiently large that all the Jordan--H\"older constituents of $E \otimes_{A} Y$ are one-dimensional. By the compatibility with short exact sequences, it suffices to assume $E \otimes_{A} Y$ is itself one-dimensional, so $E \otimes_{A} Y = E(\eta)$ for a de Rham character $\eta$ of $G.$ Moreover the $R$-algebra structure of $A$ provides a homomorphism $R\to \cO_E$, whence $\Lambda_R\to \Lambda_{\cO_E},$ such that
  \begin{eqnarray*}
 E \otimes_{\cO_E}\varepsilon_{A, \xi}(\TT)&=&E \otimes_{\cO_E}Y\otimes_{\Lambda_R} \varepsilon_{\Lambda_R, \xi}(\TT(M))\\
 &\cong&E \otimes_{\cO_E}Y\otimes_{\Lambda_{\cO_E}} \Lambda_{\cO_E}\otimes_{\Lambda_R}\varepsilon_{\Lambda_R, \xi}(\TT(M))\\
 &\cong&E \otimes_{\cO_E}Y\otimes_{\Lambda_{\cO_E}} \varepsilon_{\Lambda_{\cO_E}, \xi}(\Lambda_{\cO_E}\otimes_{\Lambda_R}\TT(M))\\
 &\cong&E \otimes_{\cO_E}Y\otimes_{\Lambda_{\cO_E}} \varepsilon_{\Lambda_{\cO_E}, \xi}(\TT( {\cO_E}\otimes_{ R}M))\\
  &\cong&E \otimes_{\cO_E} \varepsilon_{A, \xi}( \TT )\\
  &\cong&\varepsilon_{E, \xi}(E \otimes_{\cO_E} \TT).
  \end{eqnarray*}

  Here the first identity is the definition of $\varepsilon_{A, \xi}(\TT) $, the second comes from the fact that $\Lambda_R\to  {\cO_E}$ factors through $\Lambda_R\to \Lambda_{\cO_E},$ the third is the compatibility from Theorem \ref{thm:conclusion} (and its extension to arbitrary crystalline families after Proposition \ref{prop:twistinvariance}) of $\varepsilon_{\Lambda_R, \xi}(\TT(M)) $ with respect to specialization in $R$, the fourth follows from the functoriality of the definition of $\TT(M), $ while the fifth uses the definition of $\varepsilon_{A, \xi}( \TT )$ and the fact that
  \[  \TT =Y\otimes_{\Lambda_R}\TT(M)\cong  Y\otimes_{\Lambda_{\cO_E}}\TT({\cO_E}\otimes_{ R}M). \] Finally, the last identification is just \cite[Prop.\ 4.6.4]{loeffler-venjakob-zerbes}.
  \end{proof}

  \begin{corollary}\label{cor:injectivityOfSpecializations}
   Suppose that the pair $(A, \TT)$ satisfies the following condition:
   \begin{itemize}
    \item if $\Phi_{\TT}$ is the set of all $R$-algebra homomorphisms $\rho: A \to M_n(E)$ (where $E/L$ is a finite extension and $n$ an integer, both depending on $\rho$) such that $E^n \otimes_{A, \rho} \TT$ is de Rham, then
    \[ K_1(A) \to \prod_{\rho \in \Phi_{\TT}} E^\times\]
    is injective.
   \end{itemize}
   Then $\varepsilon_{A, \xi}(\TT)$ depends only on $\xi$ and on the isomorphism class of $\TT$ as an $A[G_{\Qp}]$-module.
  \end{corollary}

  \begin{proof}
   Clear from the preceding proposition, since the isomorphism $\varepsilon_{A, \xi}(\TT)$ must be consistent with the de Rham $\varepsilon$-isomorphisms $\varepsilon_{E, \xi}(E^n \otimes_{A, \rho} \TT)$, which are uniquely determined by $(A, \TT, \xi)$.
  \end{proof}

  \begin{remark}
   We suspect that the uniqueness statement of the corollary is true for arbitrary type 1 $R$-algebras $A$, but this is much more difficult to prove in general. For instance, if $M_1, M_2$ are two   crystalline families of representations over $R$ such that $M_1 / p^n \cong M_2 / p^n$ for some $n \ge 1$, then on taking $A = R / p^n$ this would imply that the $\varepsilon$-isomorphisms for $M_1$ and $M_2$ are congruent modulo $p^n$. This should certainly be true, and actually one motivation for this paper is to show it at least for $R=\cO_E$ with $E$ any finite extension of $\Qp,$ see section \ref{sec:deformation} below.
  \end{remark}

  We shall now show that the association $(A, Y, M, \xi) \to \varepsilon_{A, \xi}(\TT)$ satisfies properties corresponding to conditions (i)---(iv) and (vi) of \cite[Conjecture 3.4.3]{fukayakato06}.

  \subsubsection*{Property (i) (additivity)} The first condition of \emph{op.cit.} states that for any three triples $(A, \TT_i, \xi)$, $i = 1,2,3$, with common $A$ and $\xi$, and an exact sequence
  \[ 0 \rTo \TT_1 \rTo \TT_2 \rTo \TT_3 \rTo 0, \]
  we have
  \[ \varepsilon_{A, \xi}(\TT_2) = \varepsilon_{A, \xi}(\TT_1)\varepsilon_{A, \xi}(\TT_3).\]
  By assumption our $\TT_i$ are of the form $Y_i \otimes M_i$, for crystalline $R$-representations $M_i$ and $A$-modules $Y_i$ with $G$-action. We shall consider only the cases when the exact sequence arises from an exact sequence of $Y_i$'s with a common $M$, or an exact sequence of $M_i$'s with a common $Y$. The first case is obvious from the construction of $\varepsilon_{A, \xi}(-)$. The latter case follows from the additivity of the standard trivialisation (see \cite[Prop. 2.4.3]{loeffler-venjakob-zerbes})  or from the additivity statement (property (i) in \S 4.6 in (loc.~cit.)).

  \subsubsection*{Property (ii) (base change)} The second condition is a compatibility with base-change in $A$; this is immediate from our construction.

  \subsubsection*{Property (iii) (change of \texorpdfstring{$\xi$}{xi})} Let $c \in \Zp^\times$, and let $\gamma_c$ be any element of $G_{\Qp}$ acting trivially on $\Qpnr$ and such that $\chi(\gamma_c) = c$. Then we must show that
  \[ \varepsilon_{A, c\xi}(\TT) = [M, \gamma_c] \varepsilon_{A, \xi}(\TT), \]
  where $[M, \gamma_c]$ is the class in $K_1(A)$ of the $A$-linear automorphism of $\TT$ given by $\gamma_c$. (This is well-defined, as $\gamma_c$ is uniquely determined up to conjugation in $G_{\Qp}$.) It suffices to check this when $A = \Lambda$ and $\TT = \TT(M)$; but this is immediate from the corresponding property of the specializations (property (iii) in \S 4.6 in (loc.~cit.)).

  \subsubsection*{Property (iv) (Galois equivariance)} Let $\vp$ denote the arithmetic Frobenius automorphism of $\widehat{\Zp^{nr}}$. Then we must show that
  \[
   \varepsilon_{A, \xi}(\TT) \in \Isom\left(\Det_A(0), \Det_A R\Gamma(\Qp, \TT) \cdot \Det_A \TT\right) \times^{K_1(A)} \left\{ x \in K_1(\At) : \vp(x) = [\TT, \sigma_p]^{-1} x\right\}
  \]
  where $\sigma_p$ is the arithmetic Frobenius element of $\Gal(\Qp^{ab} / \Qpi)$. Again, it suffices to assume $(A, \TT) = (\Lambda, \TT(M))$ and the result is now clear from the Galois-equivariance properties of the specializations (property (iv) in \S 4.6 in (loc.~cit.)).

  \subsubsection*{Property (v) (compatibility with de Rham $\varepsilon$-isomorphisms)}

   If $A$ is the ring of integers of a finite extension $F / L$, and $F \otimes_{A} \TT$ is de Rham, we must check that $\varepsilon_{A, \xi}(\TT)$ is consistent with the standard trivialisation. This is exactly Proposition~\ref{prop:compat2} above.

  \subsubsection*{Property (vi) (local duality)}

  Let $\TT$ be a free $A$-module with compatible $G_{\Qp}$-action as above. Then
  \[ \TT^* := \Hom_{A}(\TT,A) \]
  is a free $A^\circ$-module -- for the action $h\mapsto h(-)r,$ $r$ in the opposite ring $A^\circ$ of $A$ -- with compatible $G_{{\Qp}}$-action
  given by $h\mapsto h\circ \sigma^{-1}$. Recall that in Iwasawa theory we have the canonical involution $\iota:\Lambda^\circ\to \Lambda$, induced by $g\mapsto g^{-1},$ which allows to consider (left) $\Lambda^\circ$-modules again as (left) $\Lambda$-modules, e.g.\ one has $\TT^*(M)^\iota\cong \TT(M^*)$ as $(\Lambda, G_{{\Qp}})$-module, where $M^\iota := \Lambda\otimes_{\iota,\Lambda^\circ}M$ denotes the $\Lambda$-module with underlying abelian group $M,$ but on which $g\in G$ acts as $g^{-1}$ for any $\Lambda^\circ$-module $M$.

  Given ${\varepsilon}_{A^\circ,-\xi}(\TT^*(1))$ we may apply the dualising functor $-^*$ to obtain an isomorphism
  \[
   \varepsilon_{A^\circ,-{\xi}}(\TT^*(1))^* : (\Det_{A^\circ}(R\Gamma({\Qp},\TT^*(1) ))_{\widetilde{A^\circ}})^* (\Det_{A^\circ}(\TT^*(1) )_{\widetilde{{A^\circ}}})^* \to \u_{\widetilde{{A^\circ}}},
  \]
  while the local Tate duality isomorphism \cite[\S 1.6.12]{fukayakato06}
  \[\psi(\TT): R\Gamma({\Qp},\TT)\cong R\Hom_{A^\circ}(R\Gamma({\Qp},\TT^*(1)),A^\circ)[-2]\]
  induces an isomorphism
  \begin{align*}
   \overline{\Det_A(\psi(\TT))_{\widetilde{A }}}^{-1}:&
   \left((\Det_{A^\circ}(R\Gamma({\Qp},\TT^*(1)))_{\widetilde{A^\circ}})^*\right)^{-1}\cong\\
   &\Det_A(R\Hom_{A^\circ}(R\Gamma({\Qp},\TT^*(1)),A^\circ))_{\widetilde{A }}^{-1} \to
   \Det_A(R\Gamma({\Qp},\TT))_{\widetilde{A }}^{-1}.\notag
  \end{align*}
  Here, for a map $f:A\to B$ in $\underline{\Det}(A),$ we write $\overline{f}:B\to A$ for its
inverse with respect to composition, while $f^{-1}=:\overline{\id_{B^{-1}}\cdot f\cdot
\id_{A^{-1}}}:A^{-1}\to B^{-1}$ for its inverse with respect to the multiplication in
$\underline{\Det}(A),$ i.e., $f\cdot f^{-1}=\id_{\Det_A(0)}$.

  Consider the product
  \[
   {\varepsilon}_{A,{\xi}}(\TT)\cdot {\varepsilon}_{A^\circ,-{\xi}}(\TT^*(1))^*\cdot \overline{\Det_A(\psi(\TT))_{\widetilde{A }}}^{-1}:
   \Det_A(\TT(-1))_{\widetilde{A}}\cong\Det_A(\TT^*(1)^*)_{\widetilde{A }}\to\Det_A(\TT)_{\widetilde{A}}
  \]
  and the isomorphism $\xymatrix{ {\TT(-1)} \ar[r]^{\cdot{\xi}} & \TT }$ which sends $t\otimes {\xi}^{\otimes -1}$ to $t$.

  \begin{proposition}[Duality]\label{duality}
   Let $\TT$ be as above such that $\TT\cong Y\otimes_\Lambda\TT(M)$ for some $(A,\Lambda)$-bimodule $Y$, which is projective as $A$-module. Then
   \[{\varepsilon}_{A,{\xi}}(\TT)\cdot {\varepsilon}_{A^\circ,-{\xi}}(\TT^*(1))^*\cdot \overline{\Det_A(\psi(\TT))_{\widetilde{A }}}^{-1}=
   \Det_A\left(\xymatrix{{\TT(-1) } \ar[r]^(0.6){\cdot{\xi}} & \TT }\right)_{\widetilde{A}}.\]
  \end{proposition}

  \begin{proof}
   First note that the statement is stable under applying $Y'\otimes_A-$, for some $(A',A)$-bimodule $Y'$ which is projective as a $A'$-module, by the functoriality of local Tate duality and the lemma below. Thus we are reduced to the case $(A, \TT)=(\Lambda, \TT(M))$ where $M$ is a crystalline family of representation of $G_{\Qp}$ with coefficients in $R$. Thus the claim follows again by specialisation and (property (vi) in \S 4.6 in (loc.~cit.))  using Lemma \ref{lemma:specialization}.

  \end{proof}

  \begin{lemma}
  Let $Y$ be a $(A',A)$-bimodule such that $Y\otimes_A\TT\cong \TT'$ as $(A',G_{{\Qp}})$-module and let
  $Y^*=\Hom_{A'}(Y,A')$ the induced $(A'^\circ,A^\circ)$-bimodule. Then there is a natural
  \begin{enumerate}
   \item equivalence of functors \[Y\otimes_A\Hom_{A^\circ}(-, A^\circ)\cong \Hom_{A'^\circ}(Y^*\otimes_{A^\circ} -,A'^\circ)\] on $P(A^\circ)$;
   \item isomorphism $Y^*\otimes_{A^\circ} \TT^*\cong (\TT')^*$ of $(A'^\circ,G_{{\Qp}})$-modules.
  \end{enumerate}
  \end{lemma}

  \begin{proof}
  This is easily checked using the adjointness of $\Hom$ and $\otimes$.
  \end{proof}

\section{Application to deformation rings}\label{sec:deformation}

Consider the following situation: For $i=1,2$ let $T_i$ be a crystalline $G_F$-representation over $\cO$ with Hodge-Tate weights in  $[a,b]$ (with integers $a,b$), such that there exist some $n_0\geq 1$ such that
\[T_1/\varpi^n T_1\cong T_2 /\varpi^n T_2=: T_{(n)}\]
for $n_0\geq n\geq 1.$  Now let $R_{T_{(1)}}^{\square,\cris,[a,b]}$ be the universal framed deformation ring parameterising crystalline lifts $T$ of $T_{(1)}$ over $\cO$-algebras $R$ with Hodge-Tate weights in $[a,b]$.  This ring is constructed (as a quotient of the universal framed deformation ring $R_{T_{(1)}}^\square$) in~\cite[Theorem 2.5.5]{kisin-pst-def-rings}.  It is equipped with  the universal crystalline lift $T_{T_{(1)}}^{\square,\cris,[a,b]}$.

\begin{proposition}
If $R_{T_{(1)}}^{\square,\cris,[a,b]}$ is a $\Zp$-flat Cohen-Macaulay ring such that $R_{T_{(1)}}^{\square,\cris,[a,b]}[\frac{1}{p}]$ is reduced and $R_{T_{(1)}}^{\square,\cris,[a,b]}$ is normal in $R_{T_{(1)}}^{\square,\cris,[a,b]}[\frac{1}{p}]$, then for all $n_0\geq n\geq 0,$ we have
\[\cO/\varpi^n\cO\otimes_{\cO} \varepsilon_{\cO,\xi}(T_1)=\cO/\varpi^n\cO\otimes_{\cO} \varepsilon_{\cO,\xi}(T_2).\]
In particular, this allows us to define $\varepsilon_{p,\cO/\varpi^n\cO}(T_{(n)}) .$
\end{proposition}

\begin{proof}
 Since $T_i$ and $T_{(n)}$ are lifts of $T_{(1)}$, for $i=1,2,$ they corresponds to  $\cO$-algebra homomorphisms $\pi_i:R_{T_{(1)}}^{\square,\cris,[a,b]}\to \cO$ and $\pi_{(n)}:R_{T_{(1)}}^{\square,\cris,[a,b]}\to \cO/\varpi^n\cO $, respectively, giving rise to a commutative diagram for each $i$
 \[ \xymatrix{
  R_{T_{(1)}}^{\square,\cris,[a,b]} \ar[dr]_{\pi_{(n)}} \ar[r]^{\pi_i}
                & {\cO} \ar[d]^{pr}  \\
                & {\cO/\varpi^n\cO}.             }\]
 Thus we obtain for $i=1,2$
 \begin{align*}
  {\cO/\varpi^n\cO}\otimes_{\cO}\varepsilon_{\cO,\xi}(T_i)&= {\cO/\varpi^n\cO}\otimes_{\cO}\left(\cO\otimes_{R_{T_{(1)},\pi_i}^{\square,\cris,[a,b]}}\varepsilon_{R_{T_{(1)}}^{\square,\cris,[a,b]},\xi}(T_{T_{(1)}}^{\square,\cris,[a,b]})\right)\\
  &={\cO/\varpi^n\cO}\otimes_{R_{T_{(1)},\pi_{(n)}}^{\square,\cris,[a,b]}}\varepsilon_{R_{T_{(1)}}^{\square,\cris,[a,b]},\xi}(T_{T_{(1)}}^{\square,\cris,[a,b]}),
 \end{align*}
 whence the claim.
\end{proof}

\begin{remark}
\begin{enumerate}
\item  It follows from~\cite[\textsection 2.4]{CHT} that $R_{T_{(1)}}^{\square,\cris,[a,b]}$ satisfies the conditions of the proposition, if $b-a<p-1$. Indeed, in that case it is isomorphic to a power series ring over $\cO.$
\item A similar statement holds also for $\varepsilon_{\Lambda_{\cO},\xi}(\TT(T_i)),$ of course.
\end{enumerate}

\end{remark}


\section{Relation to Nakamura's work}\label{sec:nakamura}

Let $M$ be a  crystalline family of representations of $G_{\Qp}$ with coefficients in $R$ and let $\TT(M) := \Lambda_{R}(\Gamma)\otimes_{R} M$, with $G_{\Qp}$ action given by $g(x\otimes y):=g(x)\otimes[\overline g]^{-1}y$.  Then the generic fiber $M^{\rig}$ is a family of Galois representations over $Y:=\Spf(R)^{\rig}$, and by~\cite[Thm.\ 3.11]{kedlaya-liu08} there is a $(\varphi,\Gamma)$-module $\D_{\rig}(M^{\rig})$ over the relative Robba ring $\mathcal{R}_Y$ attached to $M^\rig$.  We also define its universal cyclotomic deformation $\mathcal{M}:=\D_{\rig}(M^{\rig})\htimes_{R^{\rig}}\Lambda_R(\Gamma)^{\rig}$, where $\varphi$ is taken to act trivially on $\Lambda_R(\Gamma)^{\rig}$ and $\Gamma$ acts via $\gamma(d\otimes x)=\gamma(d)\otimes [\gamma]^{-1}x$.  Since $M^{\rig}$ is a crystalline family, it is trianguline after making a finite extension of $R$, so $\mathcal{M}$ is trianguline as well; let $\mathcal{F}$ denote a triangulation of $\mathcal{M}$.  In this situation, Nakamura \cite[Cor.\ 3.12]{nak13} has constructed an $\varepsilon$-isomorphism
\begin{equation}
 \varepsilon_{\mathcal{F},\cH_R(\Gamma), \xi}^N(\mathcal{M}): \Det_{\cH_R(\Gamma)}(0) \xrightarrow{\sim} \Delta_{\Lambda_R(\Gamma)^{\rig}}(\mathcal{M})
\end{equation}
where $\Delta_{\Lambda_R(\Gamma)^{\rig}}(\mathcal{M})$ is a (graded) invertible $\Spf(\Lambda_R(\Gamma))^{\rig}$-module such that
\begin{equation}\label{f:Delta}
 \cH_{\tilde R}(\Gamma)\otimes_{\cH_R(\Gamma)}(\Delta_{\Lambda_R(\Gamma)^{\rig}}(\mathcal{M}))\cong \cH_{\tilde R}(\Gamma)\otimes_{\Lambda_{R}(\Gamma)} \left\{ \Det_{\Lambda_{R}(\Gamma)} R\Gamma(\Qp, \TT(M)) \cdot \Det_{\Lambda_{R}(\Gamma)} \TT(M)\right\}
\end{equation}
by Cor.\ 3.2 in (loc.\ cit.).

\begin{corollary}\label{cor:nakamura} Under the isomorphism \eqref{f:Delta} we have
\[\cH_{\Rt}(\Gamma)\otimes_{\Lambda_{\Rt}(G)}\varepsilon_{\Lambda_{R}(G), \xi}(M)= \cH_{\tilde R}(\Gamma)\otimes_{\cH_R(\Gamma)}\varepsilon_{\mathcal{F},\cH_R(\Gamma), \xi}^N(\mathcal{M}).\]
In particular, the isomorphism $\varepsilon_{\mathcal{F},\cH_R(\Gamma), \xi}^N(\mathcal{M}) $ does not depend on the choice of $\mathcal{F}.$
\end{corollary}

\begin{proof}
Recall that $\Spf(\Lambda_{\Rt}(\Gamma))^{\rig}\xrightarrow{\sim}\Spf(\Rt)^{\rig}\times\Spf(\Lambda_{\widetilde{\Zp}}(\Gamma))^{\rig}$, so that a point of $\Spf(\Lambda_{\Rt}(\Gamma))^{\rig}$ is a pair $(x,\eta)$, where $x\in\Spf(\Rt)^{\rig}$ and $\eta$ corresponds to a character $G\rightarrow \overline{\widehat{\Qp^{\nr}}}^{\times}$ (which we also write $\eta$).  If $(x,\eta)$ lies over a point $(x',\eta')\in\Spf(\Lambda_R(\Gamma))^{\rig}$ and $\eta'$ is a de Rham character, we can compute the specializations of both sides at $(x,\eta)$.  Indeed, both sides specialize to $\widehat{\Qp^{\nr}}\otimes_{\Qp}\varepsilon_{\kappa(x'),\xi}(M\otimes_R\kappa(x')^{\circ})[1/p](\eta'^{-1})$, that is, the base-change to $\widetilde{\kappa(x')}$ of Fukaya--Kato's $\varepsilon$-isomorphism for de Rham representations.

Points of $\Spf(\Lambda_{\Zp}(\Gamma))^{\rig}$ corresponding to de Rham characters are Zariski dense in the rigid analytic space.  Combined with Lemma~\ref{lemma:qpbar-spec-inj}, we conclude that the two maps agree.
\end{proof}

%

\appendix

\section{Generic fibers}

\subsection{Quasi-Stein spaces}\label{sec:app1}

\begin{definition}
A rigid analytic space $Y$ over $K$ is said to be \emph{quasi-Stein} if it admits an admissible covering by a rising union of affinoid subdomains $Y_0\subset Y_1\subset\cdots$ such that the transition maps $\Gamma(Y_{m+1},\mathcal{O}_{Y_{m+1}})\rightarrow \Gamma(Y_m,\mathcal{O}_{Y_m})$ are flat with dense image.
\end{definition}

\begin{example}
Fix $s>0$, and let $X$ be the coordinate on the closed unit disk.  Then the half-open annulus $0<v_p(X)\leq 1/s$ is a quasi-Stein space, as it is the rising union of the closed annuli $1/s'\leq v_p(X)\leq 1/s$ as $s'\rightarrow\infty$.
\end{example}

By Kiehl's results on coherent sheaves on rigid analytic spaces, a coherent sheaf $\mathscr{F}$ on $Y$ is simply a compatible system of coherent sheaves $\{\mathscr{F}_m\}$ on $\{Y_m\}$.

Quasi-Stein spaces behave much as affine schemes do in algebraic geometry.  In particular, Kiehl proved the following theorem on the cohomology of coherent sheaves on quasi-Stein spaces.
\begin{theorem}[{\cite[Satz 2.4]{kiehl}}]\label{kiehl}
Let $Y$ be a quasi-Stein space, and let $\mathscr{F}$ be a coherent sheaf on $Y$.  Then
\begin{enumerate}
\item	$\H^i(Y,\mathscr{F})=0$ for $i>0$,
\item	the image of $\mathscr{F}(Y)$ in $\mathscr{F}(Y_m)$ is dense for all $m$,
\item	for every point $y\in Y$, the image of $\mathscr{F}(Y)$ in $\mathscr{F}_y$ generates $\mathscr{F}_y$.
\end{enumerate}
\end{theorem}
In particular, $A_\infty:=\Gamma(Y,\mathcal{O}_Y)=\varprojlim_m\Gamma(Y_m,\mathcal{O}_{Y_m})$ is a Fr\'echet-Stein algebra and $F_\infty:=\Gamma(Y,\mathscr{F})=\varprojlim_m\Gamma(Y_m,\mathscr{F}_m)$ is a coadmissible module over $A_\infty$, in the sense of~\cite{sch-teit}.  By \cite[\textsection 3]{sch-teit}, the natural morphisms $F_\infty\rightarrow \Gamma(Y_m,\mathscr{F}_m)$ have dense image, and $\R^i\varprojlim_m\Gamma(Y_m,\mathscr{F}_m)=0$ for $i>0$.

There is no {\it a priori} reason for $F_\infty$ to be a finite module over $A_\infty$.  For example, let $X_m=\Sp(\prod_{i=0}^n\Qp(\zeta_{p^i}))$, where $\Sp(A)$ denotes the affinoid space associated to $A,$ and let $\mathscr{F}_m$ be the sheaf on $Y_m$ associated to $\prod_{i=0}^n\Qp(\zeta_{p^i})^{\oplus i}$.  Then $F_\infty$ is not $A_\infty$-finite, because the fiber of $F_\infty$ at $\Sp(\Qp(\zeta_{p^n}))$ is a $\Qp(\zeta_{p^n})$-vector space of dimension $n$.  Happily, this is the only thing that can go wrong.

\begin{lemma}[{\cite[Lemma 2.4.4]{bellovin13}}]\label{fibral-ranks}
Let $\mathscr{F}$ be a coherent sheaf over a finite-dimensional quasi-Stein space $Y$ over $E$.  Then $\H^0(Y,\mathscr{F})$ is finitely generated as an $\H^0(Y,\mathcal{O}_Y)$-module if and only if there is some integer $d$ such that $\dim_{\kappa(y)}\mathscr{F}(y)\leq d$ for all $y\in Y$.
\end{lemma}

\begin{corollary}[{\cite[Corollary 2.4.5]{bellovin13}}]\label{flat-fibral-ranks}
Suppose that $\mathscr{F}$ is flat over $\mathcal{O}_Y$, where $Y$ is a finite-dimensional quasi-Stein space.  Then $\H^0(Y,\mathscr{F})$ is projective of rank $d$ over $\H^0(Y,\mathcal{O}_Y)$ if and only if $\dim_{\kappa(y)}\mathscr{F}(y)= d$ for all $y\in Y$.
\end{corollary}

An important family of examples of quasi-Stein spaces are those arising as the generic fibers of affine formal schemes of the form $\Spf(R)$ with $R=\mathcal{O}\langle Y_1,\ldots,Y_{d_1}\rangle[\![X_1,\ldots,X_{d_2}]\!]/I$.  Recall that this is a construction due to Berthelot (and exposited in~\cite[\textsection 7]{dejong}).  Briefly, we construct the generic fiber $\Spf(R)^{\rig}$ of $\Spf(R)$ by giving a cover, as follows: If $J$ is an ideal of definition of $R$, then for each $m\geq 0$, we consider the $\varpi$-adic   completion $R_m$ of $R[J^m/\varpi]$ inside $R[1/p]$.  It turns out that $R_m[1/p]$ is $K$-affinoid and $\{\Sp(R_m[1/p])\}$ is a rising union, with $\Sp(R_m[1/p])$ identified with an affinoid subdomain of $\Sp(R_{m+1}[1/p])$.  The global sections of the structure sheaf on $\Spf(R)^{\rig}$ will be the ring $ R^{\rig}:=\varprojlim_m R_m[1/p]$. In general, $ R^{\rig}$ will be non-Noetherian (unless $d_2=0$).

If $R$ is a formal power series ring, this construction corresponds to exhausting the open unit polydisk by closed balls of radius $\lvert\varpi\rvert^{1/m}$.

There is a canonical homomorphism $R\rightarrow  R^{\rig}$, and it follows from~\cite[Lemma 7.1.9]{dejong} that this map is flat.  More precisely, maximal ideals of $R[1/p]$ are in bijection with points of $\Spf(R)^{\rig}$, and the canonical homomorphism induces isomorphisms on complete local rings.  Thus, $R[1/p]\rightarrow  R^{\rig}$ is faithfully flat.

Given a finite $R$-module $M$, we can construct a coherent sheaf $M^{\rig}$ on $\Spf(R)^{\rig}$.  Explicitly, $M^{\rig}_{\Sp(R_m[1/p])}:=M\otimes_RR_m[1/p]$ (in fact, since the homomorphisms $R\rightarrow R_m[1/p]$ factor through $R[1/p]$, we can construct a coherent sheaf on $\Spf(R)^{\rig}$ from any finite $R[1/p]$-module).  This is sufficient to define $M^{\rig}$ on any admissible open subset of $\Spf(R)^{\rig}$.  Moreover, the global sections of $M^{\rig}$ behave exactly as we would hope:
\begin{lemma}\label{lemma:rig}
$\Gamma(\Spf(R)^{\rig},M^{\rig})=M\otimes_R R^{\rig}$.
\end{lemma}
\begin{proof}
We have a natural homomorphism $M\otimes_R R^{\rig}\rightarrow \varprojlim_m(M\otimes_RR_m[1/p])$, and we need to show it is an isomorphism.  This is clear if $M$ is finite free.  Next, we check that if $R^{\oplus n_1}\twoheadrightarrow M$ finitely presented, ${(R^{\rig})}^{\oplus n_1}\twoheadrightarrow \Gamma(\Spf(R)^{\rig},M^{\rig})$.

So suppose we have an exact sequence
\[	0\rightarrow K_1\rightarrow R^{\oplus n_1}\rightarrow M\rightarrow 0	\]
for some $R$-finite module $K$. Since $R\rightarrow R_m[1/p]$ is flat, for each $m$ we have an exact sequence
\[	0\rightarrow K_1\otimes_RR_m[1/p]\rightarrow (R_m[1/p])^{\oplus n_1}\rightarrow M\otimes_RR_m[1/p]\rightarrow 0	\]
But then $\R^1\varprojlim_m(K_1\otimes_RR_m[1/p])=0$, so ${(R^{\rig})}^{\oplus n_1}\twoheadrightarrow \Gamma(\Spf(R)^{\rig},M^{\rig})$.

Now we choose a presentation of $K_1$, yielding an exact sequence
\[	0\rightarrow K_2\rightarrow R^{\oplus n_2}\rightarrow K_1\rightarrow 0	\]
The same argument as above, applied to $K_1$, shows that
\[	0\rightarrow \varprojlim_m K_2\otimes_RR_m[1/p]\rightarrow  {(R^{\rig})}^{\oplus n_2}\rightarrow \varprojlim_m K_1\otimes_RR_m[1/p]\rightarrow 0	\]
remains exact.

In particular, we have an commutative diagram with exact rows
\begin{equation*}
\xymatrix{
 {(R^{\rig})}^{\oplus n_2} \ar[r]\ar[d] &  {(R^{\rig})}^{\oplus n_1} \ar[r]\ar[d] & M\otimes_R R^{\rig} \ar[r]\ar[d] & 0 \\
 {(R^{\rig})}^{\oplus n_2} \ar[r] &  {(R^{\rig})}^{\oplus n_1} \ar[r] &  \Gamma(\Spf(R)^{\rig},M^{\rig}) \ar[r] & 0
}
\end{equation*}
The two left vertical arrows are isomorphisms, so it follows that $M\otimes_R R^{\rig}\xrightarrow{\sim}\Gamma(\Spf(R)^{\rig},M^{\rig})$.
\end{proof}

Now we restrict to the case when $R=\mathcal{O}[\![X_1,\ldots,X_d]\!]$, so that $ R^{\rig}$ is the ring of analytic functions on the open unit $d$-dimensional polydisk.

\begin{definition}
For $0<\rho<1$, the \emph{$\rho$-Gauss norm} on $ R^{\rig}$ is defined by $\lVert \sum_{I\in\ZZ_{\geq 0}^{\oplus d}}a_I\underline X^I\rVert_\rho = \sup_I \lvert a_I\rvert\rho^{\lvert I\rvert}$.
\end{definition}
It turns out that $\lVert f\rVert_\rho = \sup_{x\in B[0,\rho]} \lvert f(x)\rvert$ where the supremum is taken over points $x\in \overline K^{\oplus d}$ with $\lvert x\rvert \leq \rho$, and the supremum is actually attained at some point.  Moreover, the $\rho$-Gauss norms are multiplicative, i.e., $\lVert f_1f_2\rVert_\rho = \lVert f_1\rVert_\rho = \lVert f_2\rVert_\rho$.  This is a special property of polydisks; in general the sup norm on an affinoid is only submultiplicative.

\begin{lemma}\label{bdd-fncns}
Suppose $f\in  R^{\rig}$ and there is a constant $C$ such that $\lvert f(x)\rvert\leq C$ for all $x\in(\overline K)^{\oplus d}$ with $|x|<1$.  Then $f\in R[1/p]$.
\end{lemma}
\begin{proof}
If $\lvert f(x)\rvert$ is bounded by $C$ on the open polydisk, then $\lVert f\rVert_\rho\leq C$ for all $\rho<1$.  Then
\[	\lVert p^{\log_pC}f\rVert_\rho = \sup_{I}\{C^{-1}\lvert a_I\rvert\rho^{\lvert I\rvert} \leq 1	\]
It follows that $\lvert p^{\log_pC}a_I\rvert\rho^{\lvert I\rvert}\leq 1$ for all $I$ and all $\rho <1$, and therefore $\lvert p^{\log_pC}a_I\rvert\leq 1$ for all $I$.  Therefore, $p^{\log_pC}f\in R$.
\end{proof}

Although $ R^{\rig}$ is much bigger than $R[1/p]$, we can exploit the multiplicativity of the $\rho$-Gauss norms to prove that the units of the two rings are the same:
\begin{proposition}\label{units}
${(R^{\rig})}^\times=(R[1/p])^\times$
\end{proposition}
\begin{proof}
We need to show that if $f\in  {(R^{\rig})}^\times$, then $f$ is bounded.  So suppose $f$ is invertible on $\Spf(R)^{\rig}$, and consider a sequence of closed balls with rational radii exhausting $\Spf(R)^{\rig}$, i.e., their radii form a sequence $\{\rho_i\}_{i\geq 1}$ of rational numbers which is strictly increasing and goes to $1$.  Let $c_i:=\lVert f\rVert_{\rho_i}$, so that $c_i$ is the maximum of $\lvert f(x)\rvert$ on the closed ball of radius $\rho_i$ centered at the origin.

Since $f$ is assumed invertible and $ff^{-1}=1$, multiplicativity of the $\rho$-Gauss norm implies that $\lVert f^{-1}\rVert_{\rho_i}=1/c_i$.  But if $i<j$, then $\rho_i<\rho_j$ so
\[	1/c_i=\lVert f^{-1}\rVert_{\rho_i}\leq \lVert f^{-1}\rVert_{\rho_j}=1/c_j	\]
This implies $c_j\leq c_i$.  In particular, $c_j\leq c_1$ for all $j$ and so $f\in R[1/p]$ by Lemma~\ref{bdd-fncns}.
\end{proof}

In fact, we can prove more.
\begin{proposition}\label{localize-units}
Suppose $f,f'\in R^{\rig}$ and $0\neq ff'=g\in R[1/p]$.  Then $f,f'\in R[1/p]$.
\end{proposition}
\begin{proof}
We may assume that $g\in R$, so that $\lvert g(x)\rvert\leq 1$ for every closed point $x\in\Spec R[1/p]$; let $C:=\sup_{x\in B[0,1)}\lvert g(x)\rvert$, where the supremum is taken over points $x\in {\overline K}^{\oplus d}$ with $\lvert x\rvert <1$.  Let $0<\rho<\rho'<1$ be  rational numbers.

Let $c:=\lVert f\rVert_{\rho}$, so that $c$ is the maximum of $\lvert f(x)\rvert$ on the closed ball of radius $\rho$ centered at the origin, and let $c':=\lVert f\rVert_{\rho'}$.  Then multiplicativity of the $\rho$-Gauss norm implies that $\lVert f'\rVert_{\rho}=\lVert g\rVert_{\rho}/c$ and $\lVert f'\rVert_{\rho'}=\lVert g\rVert_{\rho'}/c'$.  But $\rho<\rho'$ so
\[	\lVert g\rVert_{\rho}/c=\lVert f'\rVert_{\rho}\leq \lVert f'\rVert_{\rho'}=\lVert g\rVert_{\rho'}/c'	\]
This implies that
\[	c'\leq (\lVert g\rVert_{\rho'}/\lVert g\rVert_{\rho})c\leq (C/\lVert g\rVert_{\rho})c	\]
This holds for every $\rho'>\rho$, so $f\in R[1/p]$ by Lemma~\ref{bdd-fncns}.
\end{proof}

\subsection{Partial generic fibers}\label{subsect:partial-gen-fib}

Given two complete local noetherian $\cO$-algebras $R_1$ and $R_2$, we wish to define a ring of functions on the rigid analytic space $\Spf(R_1\widehat\otimes R_2)^{\rig}=\Spf(R_1)^{\rig}\times\Spf(R_2)^{\rig}$ consisting of those functions which are ``bounded along $\Spf(R_2)^\rig$'', but not necessarily ``along $\Spf(R_1)^{\rig}$''.  By way of contrast, $(R_1\widehat\otimes R_2)[1/p]$ roughly consists of bounded functions, whereas the ring of all analytic functions on $\Spf(R_1\widehat\otimes R_2)^{\rig}$ imposes no boundedness conditions.

\begin{definition}
If $I_1\subset R_1$ is the biggest ideal of definition of $R_1$, recall that we defined $R_{1,m}$ to be the $\varpi$-adic completion of $R_1[I^m/\varpi]$ inside $R_1[1/p]$.  If $R_{1,m}\widehat\otimes R_2$ is the completed tensor product with respect to the profinite topology on $R_2$, we define
\[	R_1^{\rig}\widehat\otimes R_2[1/p]:=\varprojlim_m \left(R_{1,m}\widehat\otimes R_2[1/p]\right)	\]
\end{definition}

We will actually only need this construction when $R_1=\cO[\![X_1,\ldots,X_k]\!]$, so we work out a more concrete description of $\cO[\![\underline X]\!]^{\rig}\widehat\otimes R$.

\begin{proposition}\label{prop:partial-gen-fiber-desc}
The ring $\cO[\![\underline X]\!]^{\rig}\widehat\otimes R[1/p]$ defined above is equal to the ring
\[	\{\sum_{\underline{n}\in \Z_{\geq 0}^{\oplus k}}r_{\underline n}\underline{X}^{\underline n}\mid r_{\underline{n}}\in R[1/p]\text{ and for all }\underline{\rho}\in\QQ_{>0}^{\oplus k}, \varpi^{n\underline\rho}r_{\underline n}\in R\text{ for }n\gg0	\}	\]
as subrings of the ring of analytic functions on the generic fiber $\Spf(R[\![\underline X]\!])^{\rig}$.
\end{proposition}
\begin{proof}
Both rings are subrings of the ring of functions on $\Spf(R[\![\underline X]\!])^{\rig}$; every such function $f$ has a unique representative of the form $f=\sum_{\underline{n}\in \Z_{\geq 0}^{\oplus k}}r_{\underline{n}}\underline{X}^{\underline{n}}$, where $r_{\underline{n}}\in R^{\rig}$ and for every $k$-tuple of positive rational numbers $\underline\rho$, $\varpi^{\underline{n}\underline{\rho}}r_{\underline{n}}\rightarrow 0$ in $R^{\rig}$.

Suppose that $f=\sum_{\underline{n}\in \Z_{\geq 0}^{\oplus k}}r_{\underline{n}}\underline{X}^{\underline{n}}$, where $r_{\underline{n}}\in R[1/p]$ and for every $k$-tuple of positive rational numbers $\underline\rho$, $\varpi^{\underline{n}\underline{\rho}}r_{\underline{n}}\in R$ for sufficiently large $\underline{n}$.  Taking $\underline\rho=(1/2s,\ldots,1/2s)$ for $s\in \Z_{>0}$, our assumption implies that $f\in (\cO\langle \underline X,\frac{\underline{X}^s}{\varpi}\rangle\widehat\otimes R)[1/p]$.  Since this holds for every such $s$, we see that $f\in \cO[\![\underline X]\!]^{\rig}\widehat\otimes R[1/p]$.

On the other hand, if $f\in \left(\cO\langle \underline{X},\frac{\underline{X}^s}{\varpi}\rangle\widehat\otimes R\right)[1/p]$, there is some $k\in\Z$ such that $p^kf\in \cO\langle \underline{X},\frac{\underline{X}^s}{\varpi}\rangle\widehat\otimes R$.  This implies that $\varpi^{k+n/s}r_{\underline{n}}\in R$ for all $n$, which implies that for $n\geq ks$, $\varpi^{2n/s}r_{\underline{n}}\in R$.  Thus, if $f\in \cO[\![\underline{X}]\!]^{\rig}\widehat\otimes R[1/p]$, the coefficients satisfy the desired growth condition.
\end{proof}

\begin{remark}
This alternate characterization of $\cO[\![X]\!]^{\rig}\widehat\otimes R[1/p]$ makes it clear that if $\mathfrak{m}\subset R[1/p]$ is a maximal ideal, then $\cO[\![X]\!]^{\rig}\widehat\otimes (R[1/p]/\mathfrak{m})\cong (\cO[\![X]\!]^{\rig}\widehat\otimes R[1/p])/\mathfrak{m}$.
\end{remark}

\begin{proposition}\label{partial-localize-units}
Suppose that $R$ is normal and $R[1/p]$ is reduced.  If $f,f'\in R[1/p]\htimes\cO[\![\underline X]\!]^{\rig}$ satisfy $ff'=g\in R[1/p]$ and $g$ is not identically zero on any irreducible component of $\Spec R[1/p]$, then $f,f'\in (R\htimes\cO[\![\underline X]\!])[1/p]$.
\end{proposition}
\begin{proof}
We may assume that $g\in R$, so that $\lvert g(x)\rvert\leq 1$ for every closed point $x\in\Spec R[1/p]$.  For a closed point $x\in\Spec R[1/p]$ and a rational number $0<\rho<1$, let $\lVert f\rVert_{x,\rho}$ denote the supremum of $\lvert f\rvert$ on the subset $\{x\}\times B[0,\rho]\subset\MaxSpec R[1/p]\times (\Spf \cO[\![\underline X]\!])^{\rig}$.  Let $\lVert f\rVert_{\rho}$ denote the supremum of $\lvert f\rvert$ over all of $\MaxSpec R[1/p]\times B[0,\rho]$.

If $x\in\Spec R[1/p]$ is a closed point such that $g(x)\neq 0$, then the proof of Proposition~\ref{units} shows that for every pair of rational numbers $0<\rho<\rho'<1$,
\[	\lVert f\rVert_{x,\rho'}\leq \lVert f\rVert_{x,\rho}\leq \lVert f\rVert_{\rho}	\]
Thus, $f$ is bounded on the complement of the zero locus of $g$; suppose $\lvert f(x)\rvert\leq C$ for all $x\in \Spec R[1/p,1/g]\times (\Spf \cO[\![\underline X]\!])^{\rig}$.

Now let $x\in\Spec R[1/p]$ be a closed point with $g(x)=0$, and let $\{x_i\}_i$ be a sequence of closed points of $\Spec R[1/p,1/g]$ converging to $x$.  Then for any $x'\in \Spf \cO[\![\underline X]\!]^{\rig}$, we have shown that $\lvert f(x_i,x')\rvert \leq C$.  But this implies that $\lvert f(x,x')\rvert\leq C$, so $\lVert f\rVert_{\rho}\leq C$ for all $\rho<1$ and we are done.
\end{proof}

\section{Determinant functors}\label{determinants} In this appendix we recall some details of the formalism of determinant functors
 introduced by Fukaya and Kato in \cite{fukayakato06}, see also \cite{deligne87}.

We fix an associative unital noetherian ring $R$. We write ${ B}(R)$ for the category of bounded
complexes of (left) $R$-modules, ${ C}(R)$ for the  category of bounded complexes of finitely
generated (left) $R$-modules, $P(R)$ for the category of finitely generated projective (left)
$R$-modules, $ {C}^{\rm p}(R)$ for the category of bounded (cohomological) complexes of finitely
generated projective (left) $R$-modules. By $D^{\rm p}(R)$ we denote the category of perfect
complexes as full triangulated subcategory of the derived category $D^{\rm b}(R)$ of $B(R).$ We
write $(C^{\rm p}(R),{\rm quasi})$ and $(D^{\rm p}(R),{\rm is})$ for the subcategory of
quasi-isomorphisms of $C^{\rm p}(R)$ and isomorphisms of $D^{\rm p}(R),$ respectively.

For each complex $C = (C^\bullet,d_C^\bullet)$ and each integer $r$ we define the $r$-fold shift
$C[r]$ of $C$ by setting $C[r]^i= C^{i+r}$ and $d^i_{C[r]}=(-1)^rd^{i+r}_C$ for each integer $i$.

We first recall that there exists a Picard category $\C_R$ and a determinant functor $\,\d_R:(
{C}^{\rm p}(R),{\rm quasi})\to \C_R$ with the following properties (for objects $C,C'$ and $C''$ of
$\mathrm{C}^{\rm p}(R)$)

\begin{enumerate}[label=B.\alph*),align=left, leftmargin=4.5em]
\item $\C_R$ has an associative and commutative product
structure $(M,N) \mapsto M\cdot N$ (which we often write more simply as $MN$) with canonical unit
object ${\bf 1}_R = \d_R(0)$. If $P$ is any object of $P(R)$, then in $\C_R$ the object $\d_R(P)$
has a canonical inverse $\d_R(P)^{-1}$. Every object of $\C_R$ is of the form $\d_R(P)\cdot
\d_R(Q)^{-1}$ for suitable objects $P$ and $Q$ of $P(R)$.

\item  All morphisms in $\C_R$ are isomorphisms and elements of
the form $\d_R(P)$ and $\d_R(Q)$ are isomorphic in $\C_R$ if and only if $P$ and $Q$ correspond to
the same element of the Grothendieck group $K_0(R)$. There is a natural identification ${\rm
Aut}_{\C_R}({\bf 1}_R) \cong K_1(R)$ and if ${\rm Mor}_{\C_R}(M,N)$ is non-empty, then it is a
$K_1(R)$-torsor where each element $\alpha$ of $K_1(R)\cong {\rm Aut}_{\C_R}({\bf 1}_R)$
 acts on $\phi \in {\rm Mor}_{\C_R}(M,N)$ to give $\alpha\phi: M =
{\bf 1}_R\cdot M \xrightarrow{\alpha\cdot \phi}{\bf 1}_R\cdot N = N$.

\item  $\d_R$ preserves the product structure: specifically,
 for each $P$ and $Q$ in $P(R)$ one has $\d_R(P\oplus Q) =
\d_R(P)\cdot\d_R(Q)$.
\item 
If $C'\to C\to C''$ is a short exact sequence of complexes,
then there is a canonical isomorphism $\,\d_R(C)\cong \d_R(C')\d_R(C'')$ in $\C_R$ (which we
usually take to be an identification). 
\item  If $C$ is acyclic, then the quasi-isomorphism $0\to C$
induces a canonical isomorphism $\,\u_R\to\d_R(C).$
\item  For any integer $r$ one has
$\d_R(C[r])=\d_R(C)^{(-1)^r}$.
\item  the functor $\d_R$ factorises over the image of
$C^{\rm p}(R)$ in $D^{\rm p}(R)$ and extends (uniquely up to unique isomorphisms) to $(D^{\rm
p}(R),{\rm is}).$ Moreover, if $R$ is regular, also property B.d) extends to all distinguished triangles.
\item \label{B.h)} For each $C$ in $D^{\rm b}(R)$ we write $\H(C)$ for the
 complex which has $\H(C)^i = H^i(C)$ in each degree $i$ and in which all differentials are $0$. If
 $\H(C)$ belongs to $D^{\rm p}(R)$ (in which case one says that $C$ is {\em cohomologically perfect}), then $C$ belongs to $D^{\rm p}(R)$ and
 there are canonical
 isomorphisms
\[\d_R(C) \cong \d_R(\H(C)) \cong \prod_{i\in \Z} \d_R(H^i(C))^{(-1)^i}.\]
(For an explicit description of the first isomorphism see \cite[\S 3]{km}.)
\item  If $R'$ is another (associative unital noetherian) ring and $Y$ an $(R',R)$-bimodule
 that is both finitely generated and projective as an $R'$-module,
 then the functor $Y\otimes_R-: P(R)\to P(R')$ extends to a
commutative diagram
\[\xymatrix{
  (\mathrm{D}^p(R), is) \ar[d]_{Y\otimes_R^\mathbb{L}-} \ar[rr]^{\d_R} & & {\C_R} \ar[d]^{Y\otimes_R-} \\
  (\mathrm{D}^p(R'), is) \ar[rr]^{\d_{R'}} & & {\C_{R'}}   }\]
In particular, if $R\to R'$ is a ring homomorphism and $C$ is in $\mathrm{D}^{\rm p}(R),$ then we
often simply write $\d_R(C)_{R'}$ in place of $R'\otimes_R\d_R(C).$

\item \label{Aj} Let $R^\circ$ be the opposite ring of $R.$ Then the functor $\mathrm{Hom}_R(-,R)$ induces an
anti-equivalence between $\C_R$ and $\C_{R^\circ}$ with quasi-inverse induced  by
$\mathrm{Hom}_{R^\circ}(-,R^\circ);$ both functors will be denoted by $-^*.$ This extends to give a
diagram
 \[\xymatrix{
   (\mathrm{D}^p(R), is) \ar[d]_{\mathrm{RHom}_R(-,R)} \ar[rr]^{\d_R}& & {\C_R }\ar[d]^{-^*} \\
   (\mathrm{D}^p(R^\circ),is) \ar[rr]^{\d_{R^\circ}} & & {\C_{R^\circ} } } \]
   which commutes (up to unique isomorphism);   similarly  we have such a commutative diagram for $\mathrm{RHom}_{R^\circ}(-,{R^\circ}).$
\end{enumerate}

If $R$ denotes any commutative ring, then by \cite[\S 1.2.4 (3)]{fukayakato06}  the Fukaya-Kato/Deligne determinant functor is closely related to Knudsen's and Mumford's which takes values in the category of line bundles plus ranks over $\Spec(R).$ In case $SK_1(R)=1$ both determinants can be naturally identified. Since in this article we often use geometric (specialisation) arguments we thus take  frequently this point of view without indicating this in the notation.

\begin{remark}
Recall from \cite[III.Lem.\ 1.4]{weibel} that for a commutative semi-local ring $R$ we have $SK_1(R)=1$ and $K_1(R)\cong R^\times.$ This applies in particular to $\Lambda_R(G)$ and $\Lambda_{\Rt}(G)$ as defined in \eqref{f:tilde}.
\end{remark}

We end this section by considering the example where $R=K$ is a field and $V$ a finite dimensional
vector space over $K.$ Then, according to \cite[1.2.4]{fukayakato06}, $\d_K(V)$ can be identified
with the highest exterior product $\bigwedge^{\rm{top}}V$ of $V$ and for an automorphism $\phi: V\to V$
the determinant $\d_K(\phi)\in K^\times=K_1(K)$ can be identified with the usual determinant
$\det_K(\phi).$ In particular, we identify $\d_K=K$ with canonical basis $1.$ Then a map
$\u_K \xrightarrow{\psi}\u_K$ corresponds uniquely to the value $\psi(1)\in
K^\times.$

\begin{example}\label{finitemodules}
Note that every {\em finite} $\Zp$-module $A$ possesses a free resolution $C$, i.e.\ $\d_{\Zp}(A)\cong\d_{\Zp}(C)^{-1}=\u_{\Zp}.$ Then modulo $\Zp^\times$
the composite $\u_{\Qp}\xrightarrow{\rm{acyc}}\d_{\Zp}(C)_{\Qp} \xrightarrow{\sim}\u_{\Qp}$ corresponds to the cardinality $|A|^{-1}\in\Qp^\times.$
\end{example}

\bibliographystyle{amsalpha}
\bibliography{references}

\providecommand{\bysame}{\leavevmode\hbox to3em{\hrulefill}\thinspace}
\providecommand{\MR}{\relax\ifhmode\unskip\space\fi MR }
\providecommand{\MRhref}[2]{%
  \href{http://www.ams.org/mathscinet-getitem?mr=#1}{#2}
}
\providecommand{\href}[2]{#2}
\begin{thebibliography}{{Wan}15}

\bibitem[BB08]{benoisberger08}
Denis Benois and Laurent Berger, \emph{Th\'eorie d'{I}wasawa des
  repr\'esentations cristallines. {II}}, Comment. Math. Helv. \textbf{83}
  (2008), no.~3, 603--677. \MR{2410782}

\bibitem[BB10]{berger-breuil}
Laurent Berger and Christophe Breuil, \emph{Sur quelques repr\'esentations
  potentiellement cristallines de {${\rm GL}_2(\bold Q_p)$}}, Ast\'erisque
  (2010), no.~330, 155--211. \MR{2642406}

\bibitem[BC08]{berger-colmez}
Laurent Berger and Pierre Colmez, \emph{Familles de repr\'esentations de de
  {R}ham et monodromie {$p$}-adique}, Ast\'erisque (2008), no.~319, 303--337,
  Repr{\'e}sentations $p$-adiques de groupes $p$-adiques. I.
  Repr{\'e}sentations galoisiennes et $(\phi,\Gamma)$-modules.

\bibitem[Bel15]{bellovin13}
Rebecca Bellovin, \emph{$p$-adic {H}odge theory in rigid-analytic families},
  Algebra \& Number Theory \textbf{9} (2015), 371--433.

\bibitem[Ber02]{berger02}
Laurent Berger, \emph{Repr\'esentations $p$-adiques et \'equations
  diff\'erentielles}, Invent. Math. \textbf{148} (2002), 219--284.

\bibitem[Ber04]{berger04}
\bysame, \emph{Limites de repr{\'e}sentations cristallines}, Compos. Math.
  \textbf{140} (2004), no.~6, 1473--1498. \MR{2098398}

\bibitem[CC98]{cherbonnier-colmez}
Frederic Cherbonnier and Pierre Colmez, \emph{Repr\'esentations $p$-adiques
  surconvergentes}, Invent. Math. \textbf{133} (1998), 581--611.

\bibitem[CHT08]{CHT}
Laurent Clozel, Michael Harris, and Richard Taylor, \emph{Automorphy for some
  {$l$}-adic lifts of automorphic mod {$l$} {G}alois representations}, Publ.
  Math. Inst. Hautes \'Etudes Sci. (2008), no.~108, 1--181, With Appendix A,
  summarizing unpublished work of Russ Mann, and Appendix B by Marie-France
  Vign{\'e}ras.

\bibitem[Col08]{colmez08}
Pierre Colmez, \emph{Espaces vectoriels de dimension finie et repr\'esentations
  de de {R}ham}, Ast\'erisque (2008), no.~319, 117--186, Repr{\'e}sentations
  $p$-adiques de groupes $p$-adiques. I. Repr{\'e}sentations galoisiennes et
  $(\phi,\Gamma)$-modules. \MR{2493217}

\bibitem[Dee01]{dee01}
Jonathan Dee, \emph{{}{}{}{}{}{}{}{}{}{}{}{$\Phi$}-{${\Gamma}$} modules for
  families of {G}alois representations}, J. Algebra \textbf{235} (2001), no.~2,
  636--664. \MR{1805474}

\bibitem[Del87]{deligne87}
Pierre Deligne, \emph{Le d\'eterminant de la cohomologie}, Current trends in
  arithmetical algebraic geometry (Arcata, Calif., 1985), Contemp. Math.,
  vol.~67, Amer. Math. Soc., Providence, RI, 1987, pp.~93--177.

\bibitem[dJ95]{dejong}
A.J. de~Jong, \emph{Crystalline dieudonn\'e theory via formal and rigid
  geometry}, Inst. Hautes \'Etudes Sci. Publ. Math. (1995), no.~82, 5--96
  (1996). \MR{1383213 (97f:14047)}

\bibitem[FK06]{fukayakato06}
Takako Fukaya and Kazuya Kato, \emph{A formulation of conjectures on {$p$}-adic
  zeta functions in noncommutative {I}wasawa theory}, Proceedings of the {S}t.
  {P}etersburg {M}athematical {S}ociety. {V}ol. {XII} (Providence, RI), Amer.
  Math. Soc. Transl. Ser. 2, vol. 219, Amer. Math. Soc., 2006, pp.~1--85.

\bibitem[Fon90]{fontaine90}
Jean-Marc Fontaine, \emph{Repr\'esentations {$p$}-adiques des corps locaux.
  {I}}, The {G}rothendieck {F}estschrift, {V}ol.\ {II}, Progr. Math., vol.~87,
  Birkh\"auser Boston, Boston, MA, 1990, pp.~249--309. \MR{1106901 (92i:11125)}

\bibitem[HP17]{hu-paskunas}
Yongquan Hu and Vytautas Paskunas, \emph{On crystabelline deformation rings of
  $\mathrm{Gal}({\overline{\Qp}}/\mathbb{Q}_p)$ (appendix by {J}ack
  {S}hotton)}, preprint, 2017.

\bibitem[Kie67]{kiehl}
Reinhardt Kiehl, \emph{Theorem {A} und {T}heorem {B} in der nichtarchimedischen
  {F}unktionentheorie}, Invent. Math. \textbf{2} (1967), 256--273. \MR{0210949
  (35 \#1834)}

\bibitem[Kis08]{kisin-pst-def-rings}
Mark Kisin, \emph{Potentially semi-stable deformation rings}, J. Amer. Math.
  Soc. \textbf{21} (2008), no.~2, 513--546. \MR{2373358}

\bibitem[Kis09]{kisin-bt}
\bysame, \emph{Moduli of finite flat group schemes, and modularity}, Ann. of
  Math. (2) \textbf{170} (2009), no.~3, 1085--1180. \MR{2600871}

\bibitem[KL10]{kedlaya-liu08}
Kiran Kedlaya and Ruochuan Liu, \emph{On families of {$\phi$},
  {$\Gamma$}-modules}, Algebra Number Theory \textbf{4} (2010), no.~7,
  943--967. \MR{2776879}

\bibitem[KM76]{km}
F.~F. Knudsen and D.~Mumford, \emph{The projectivity of the moduli space of
  stable curves. {I}. {P}reliminaries on ``det'' and ``{D}iv''}, Math. Scand.
  \textbf{39} (1976), no.~1, 19--55.

\bibitem[LVZ15]{loeffler-venjakob-zerbes}
David Loeffler, Otmar Venjakob, and Sarah~Livia Zerbes, \emph{Local epsilon
  isomorphisms}, Kyoto J. Math. \textbf{55} (2015), no.~1, 63--127.
  \MR{3323528}

\bibitem[LZ14]{loefflerzerbes11}
David Loeffler and Sarah~Livia Zerbes, \emph{Iwasawa theory and {$p$}-adic
  {$L$}-functions over {$\Bbb{Z}\sb p\sp 2$}-extensions}, Int. J. Number Theory
  \textbf{10} (2014), no.~8, 2045--2095. \MR{3273476}

\bibitem[Mat89]{matsumura}
Hideyuki Matsumura, \emph{Commutative ring theory}, second ed., Cambridge
  Studies in Advanced Mathematics, vol.~8, Cambridge University Press,
  Cambridge, 1989, Translated from the Japanese by M. Reid. \MR{1011461
  (90i:13001)}

\bibitem[{Nak}13]{nak13}
K.~{Nakamura}, \emph{{A generalization of {K}ato's local epsilon-conjecture for
  ($\varphi,{\Gamma}$)-modules over the {R}obba ring}}, ArXiv e-prints (2013).

\bibitem[Nak14]{nak15}
K.~Nakamura, \emph{Iwasawa theory of de {R}ham {$(\varphi,\Gamma)$}-modules
  over the {R}obba ring}, J. Inst. Math. Jussieu \textbf{13} (2014), no.~1,
  65--118.

\bibitem[Nek06]{nekovar06}
Jan Nekov{\'a}{\v{r}}, \emph{Selmer complexes}, Ast{\'e}risque (2006), no.~310,
  viii+559. \MR{2333680}

\bibitem[Sch02]{schneider}
Peter Schneider, \emph{Nonarchimedean functional analysis}, Springer monographs
  in mathematics, Springer, Berlin ; Heidelberg [u.a.], 2002 (eng).

\bibitem[Ser95]{serre95}
Jean-Pierre Serre, \emph{Classes des corps cyclotomiques (d'apr\`es {K}.
  {I}wasawa)}, S\'eminaire {B}ourbaki, {V}ol.\ 5, Soc. Math. France, Paris,
  1995, pp.~Exp.\ No.\ 174, 83--93. \MR{1603459}

\bibitem[SRV13]{SSV}
Peter Schneider, Sujatha Ramdorai, and Otmar Venjakob,
  \emph{\textbf{Discussions among Otmar, Peter, and Sujatha} on the paper
  ``wach modules and iwasawa theory for modular forms'' by lei/loeffler/zerbes
  (february 2013, ubc)}, 2013.

\bibitem[ST03]{sch-teit}
Peter Schneider and Jeremy Teitelbaum, \emph{Algebras of $p$-adic distributions
  and admissible representations}, Invent. Math. \textbf{153} (2003), no.~1,
  145--196. \MR{1990669 (2004g:22015)}

\bibitem[SU14]{skinner-urban}
C.~Skinner and E.~Urban, \emph{The {I}wasawa main conjectures for {$\rm GL\sb
  2$}}, Invent. Math. \textbf{195} (2014), no.~1, 1--277.

\bibitem[Ven07]{ven-BSD}
Otmar Venjakob, \emph{From the {B}irch and {S}winnerton-{D}yer conjecture to
  non-commutative {I}wasawa theory via the equivariant {T}amagawa number
  conjecture---a survey}, {$L$}-functions and {G}alois representations, London
  Math. Soc. Lecture Note Ser., vol. 320, Cambridge Univ. Press, Cambridge,
  2007, pp.~333--380.

\bibitem[Ven13]{venjakob11}
\bysame, \emph{On {K}ato's local {$\epsilon$}-isomorphism conjecture for
  rank-one {I}wasawa modules}, Algebra Number Theory \textbf{7} (2013), no.~10,
  2369--2416. \MR{3194646}

\bibitem[{Wan}15]{wang-erickson15}
Carl {Wang Erickson}, \emph{Algebraic families of {G}alois representations and
  potentially semi-stable pseudodeformation rings}, preprint, 2015.

\bibitem[Wei13]{weibel}
C.~A. Weibel, \emph{The {$K$}-book}, Graduate Studies in Mathematics, vol. 145,
  American Mathematical Society, Providence, RI, 2013, An introduction to
  algebraic $K$-theory.

\end{thebibliography}

\end{document}